%% file: 20210506_BV_structure_for_arXiv.tex
\documentclass[reqno,12pt,draft]{amsart}
\usepackage{bm,amscd,textcomp,mathrsfs,lscape}
\usepackage{amsmath,amsthm,amssymb,verbatim,mathtools}
\usepackage{color}
\usepackage{enumerate}
\usepackage[dvipdfm]{graphicx}
\usepackage{latexsym}
\usepackage[all]{xy}
\usepackage{calc}
\usepackage[top=3cm, bottom=3cm, left=2cm, right=2cm]{geometry}

\newcommand{\bA}{\overline{A}} 
\newcommand{\balpha}{\overline{\alpha}} 
\newcommand{\ba}{\overline{a}} 
\newcommand{\bb}{\overline{b}}

\newcommand{\bv}{\overline{v}}
\newcommand{\bu}{\overline{u}}
\newcommand{\bx}{\overline{x}}
\newcommand{\by}{\overline{y}}
\newcommand{\nuina}{\overline{\nu^{-1} a}}

\newcommand{\HH}{\mathrm{HH}}
\newcommand{\Hh}{\mathrm{H}}
\newcommand{\Hom}{\mathrm{Hom}} 
\newcommand{\Ext}{\operatorname{Ext}\nolimits}

\newcommand{\id}{\operatorname{id}}

\renewcommand{\Im}{\mathrm{Im}\,}

\renewcommand{\labelenumi}{(\arabic{enumi})}

 \makeatletter
 
 \@addtoreset{equation}{section}
 \makeatother	

\newtheorem{defi}{Definition}[section]
\newtheorem{theo}[defi]{Theorem}
\newtheorem{lem}[defi]{Lemma}
\newtheorem{rem}[defi]{Remark}
\newtheorem{cor}[defi]{Corollary}
\newtheorem{prop}[defi]{Proposition}
\newtheorem{criterion}[defi]{Criterion}
\newtheorem{main}{Main result\hspace{-5pt}}

\title
[{\fontsize{7pt}{0pt}\selectfont A Batalin-Vilkovisky structure on the complete cohomology ring of a Frobenius algebra}]
{A Batalin-Vilkovisky structure on the complete cohomology ring of a Frobenius algebra}

\author[T.\,Itagaki]{Tomohiro Itagaki}
\address[T.\,Itagaki]{Faculty of Economics, Takasaki City University of Economics, 1300 Kaminamie-machi, Takasaki-shi, Gunma 370-0801, Japan}
\email{titagaki@tcue.ac.jp}

\author[K.\,Sanada]{Katsunori Sanada}
\address[K.\,Sanada]{Department of Mathematics,
					Tokyo University of Science, 1-3 Kagurazaka, Shinjuku-ku, Tokyo 162--8601, Japan}
\email{sanada@rs.tus.ac.jp}

\author[S.\,Usui]{Satoshi Usui}
\address[S.\,Usui]{\fontsize{10pt}{0pt}\selectfont 
						Department of Mathematics,
						Tokyo University of Science, 1-3 Kagurazaka, Shinjuku \hspace{-3pt}-ku, Tokyo 162--8601, Japan}
\email{1119702@ed.tus.ac.jp}


\begin{document}

\maketitle


\begin{abstract}
	We study the existence of a Batalin-Vilkovisky differential on the complete cohomology ring of a Frobenius algebra.
	We construct a Batalin-Vilkovisky differential on the complete cohomology ring in the case of Frobenius algebras 
	with diagonalizable Nakayama automorphisms. 
\end{abstract}


\section*{Introduction}

	In 1945, Hochschild \cite{Hochschild} introduced the {\it Hochschild cohomology group} $\mathrm{H}^{*}(A, M)$ of an associative algebra $A$ 
	with coefficients in an $A$-bimodule $M$.
	In the earlier of the 1960s, 
	Gerstenhaber \cite{Ger} discovered that there is a rich algebraic structure on Hochschild cohomology ring
	$\mathrm{H}^{\bullet}(A, A) := \bigoplus_{r \geq 0} \mathrm{H}^{r}(A, A)$.
	To be more precise, 
	$\mathrm{H}^{\bullet}(A, A)$ has a {\it Gerstenhaber structure}, that is, 
	there is a Lie bracket on Hochschild cohomology such that the Lie bracket satisfies 
	the graded Leibniz rule with respect to the cup product.
	The Lie bracket is also called the {\it Gerstenhaber bracket}.

	During recent several decades, a new algebraic structure, the so-called {\it Batalin-Vilkovisky (BV) structure}, on Hochschild cohomology 
	has been studied. 
	A BV structure gives rise to an operator on Hochschild cohomology which squares to zero 
	and which can generate the Gerstenhaber bracket together with the cup product. 
	Let us remark that BV structure is originally defined for a graded commutative algebra and gives rise to a Lie bracket making 
	the graded commutative algebra together with itself into a Gerstenhaber algebra (cf. \cite{Getzler94}). 
	From this context, the existence of BV operator generating 
	the Gerstenhaber bracket is particularly important. 
	It is known that there exists such a BV structure over Hochschild cohomology of certain classes of algebras, 
	such as Calabi-Yau algebras, 
	finite dimensional symmetric algebras, 
	finite dimensional Frobenius algebras whose Nakayama automorphisms are semisimple and so on 
	(cf. \cite{Ginzburg, Lam, Menichi, NielUlrich2014, Tradler}). 
	Here, 
	the semisimplicity of the Nakayama automorphism means that it is diagonalizable over the algebraic closure of the underlying field.

	In the 1980s, Buchweitz \cite{Buch} introduced the notion of  {\it singularity category} in order to provide a framework for Tate cohomology of Gorenstein algebras.
	In 2015, under this framework, Wang defined the $r$-th {\it Tate-Hochschild cohomology group} of $A$ as 
		\begin{align*}
			\underline{\mathrm{Ext}}_{A \otimes_{k} A^{\rm op}}^{r}(A, A) := \Hom_{\mathcal{D}_{{\rm sg}}(A \otimes_{k} A^{\rm op})}(A, A[r]),
		\end{align*}
	where $r$ is arbitrary integer and $\mathcal{D}_{{\rm sg}}( A \otimes_{k} A^{\rm op})$ is the singularity category of $A \otimes_{k} A^{\rm op}$
	(cf. \cite{wang,wang2}).
	Wang discovered in \cite{wang} that Tate-Hochschild cohomology $\underline{\mathrm{Ext}}_{A \otimes_{k} A^{\rm op}}^{\bullet}(A, A) $
	has a Gerstenhaber structure equipped with the Yoneda product $\smile_{{\rm sg}}$ and the Lie bracket $[ \ ,\ ]_{{\rm sg}}$. 
	Moreover, he determined in \cite{wang2018_squarezero} the Tate-Hochschild cohomology of radical square zero algebras and 
	their Gerstenhaber structures for some classes of such algebras. 
  
	Many authors have investigated Tate-Hochschild cohomology  
	in case of Frobenius algebras (cf. \cite{BerghJorgensen,EuSchedler2009,Nakayama,Nguyen,wang,wang2}).
	One of the first attempts was made by Nakayama \cite{Nakayama}.
	In the 1950s, as an analogy to Tate cohomology for a finite group, Nakayama introduced a {\it complete cohomology groups} 
	$\widehat{\mathrm{H}}^{*}(A, M)$ of a Frobenius algebra $A$ with coefficients in an $A$-bimodule $M$.
	Here, complete cohomology groups coincide with Tate-Hochschild cohomology groups (cf. \cite{Buch}).
	A complex which is used to compute complete cohomology groups $\widehat{\mathrm{H}}^{*}(A, M)$ is a \textit{complete complex}.
	Roughly speaking, the complete complex is an unbounded complex having the Hochschild cochain complex 
	$C^{\bullet}(A, M)$ in non-negative degrees and the Hochschild chain complex $C_{\bullet}(A, M_{\nu^{-1}})$ in negative degrees, 
	where $\nu$ is the Nakayama automorphism of the Frobenius algebra $A$.
	In 1992, Sanada \cite{Sanada1992} constructed a cup product on complete cohomology groups  by means of a diagonal approximation and 
	investigated a periodicity of  complete cohomology groups. 
	Recently, Wang have discovered in \cite{wang2} that there is a graded commutative 
	product, called {\it $\star$-product}, on complete cohomology groups such that the complete cohomology ring is isomorphic to 
	the Tate-Hochschild cohomology ring.
	Furthermore, as an application of BV differential on Hochschild cohomology, 
	he constructed a BV differential on $\widehat{\mathrm{H}}^{*}(A, A)$ in the case that $A$ is a finite dimensional symmetric algebra, 
	where the BV differential consists of Tradler's BV differential \cite{Tradler} and the Connes operator. 
	In particular, the induced Lie bracket on complete cohomology is isomorphic to Wang's bracket on Tate-Hochschild cohomology. 
	
	In this paper, we generalize Wang's result to the case of finite dimensional Frobenius algebras with diagonalizable Nakayama automorphisms. 
	Namely, the aim of the present paper is to prove the following result:
		\begin{main}\label{intro-maintheo} 
			Let $A$ be a finite dimensional Frobenius $k$-algebra. 
			If the Nakayama automorphism of $A$ is diagonalizable, then the complete cohomology ring $\widehat{\mathrm{H}}^{\bullet}(A, A)$ is a BV algebra
			together with a BV differential consisting of Lambre-Zhou-Zimmermann's BV differential \cite{Lam} and (twisted) Connes operator
			(see Theorem \ref{BV-theo2} and Corollary \ref{BV-cor2}). 
		\end{main}
	This paper is organized as follows: 
	In Section \ref{Preliminaries}, we recall some definitions and basic results on Tate-Hochschild cohomology and complete cohomology.
	Section \ref{Frobenius} is devoted to recalling the complete complex of a Frobenius algebra, and then to relating Tate-Hochschild cohomology groups 
	with complete cohomology groups.
	In Section \ref{Decomposition}, following Lambre-Zhou-Zimmermann \cite{Lam}, we define a subcomplex of the complete complex associated with 
	the product of eigenvalues of the Nakayama automorphism. 
	We then show that the cohomology groups of the subcomplex has nice properties when the Nakayama automorphism is diagonalizable.
	In Section \ref{BV}, we give the main result and the proof. 
    Our proof is different from Lambre-Zhou-Zimmermann \cite{Lam}. 
    More precisely, our proof uses  the Gerstenhaber bracket $[\ ,\ ]_{\rm sg}$ defined by Wang  \cite{wang}, while Lambre, Zhou and Zimmermann's proof makes use of the notions of Tamarkin-Tsygan calculi and calculi with duality. 
    Section \ref{Examples} contains three examples of BV structures on the complete cohomology rings of some self-injective Nakayama algebras over algebraically closed fields with diagonalizable Nakayama automorphisms.



\section{Preliminaries} \label{Preliminaries}
	Throughout this paper, let $k$ be a field, $A$ a finite dimensional, associative and unital $k$-algebra. Let  
	$A^{\rm e}$ be the enveloping algebra $A \otimes_{k} A^{\rm op}$ of $A$. Here we denote by $A^{\rm op}$ the opposite algebra of $A$. 
	We can identify an $A$-bimodule $M$ with a left (right) $A^{\rm e}$-module $M$ whose structure is given by 
	$(a \otimes_{k} b^{\circ}) m := a m b \ (m (a \otimes_{k} b^{\circ}) := b m a)$ for $m \in M$ and $a \otimes_{k} b^{\circ} \in A^{\rm e}$.
	For simplicity, we write $\otimes$ for $\otimes_{k}$  and $\Hom$ for $\Hom_{k}$.
	We denote by $\overline{A}$ the quotient space of $A$ by the subspace $k1_{A}$ generated by unit $1_{A}$. 
	Let $\sigma : A \rightarrow A$ be an algebra automorphism of $A$ and $\pi : A \rightarrow \overline{A}$ the canonical epimorphism of $k$-vector
	spaces. We denote by
	$\overline{a}$ 
	denote the image of $a \in A$ under  the epimorphism $\pi : A \rightarrow \overline{A}$. 
	We write 
	$a_{1, \,m} \in A^{\otimes m}$ 
	for
	$a_{1} \otimes \cdots \otimes a_{m} \in A^{\otimes m}$, 
	$\bb_{1,\,n} \in \bA^{\otimes n}$
	for
	$\bb_{1} \otimes \cdots \otimes \bb_{n} \in \bA^{\otimes n}$
	and 
	$\overline{\sigma c}_{1,\,l} \in \bA^{\otimes l}$
	for
	$\overline{\sigma (c_{1})} \otimes \overline{\sigma (c_{2})} \otimes \cdots \otimes \overline{\sigma (c_{l})} \in \bA^{\otimes l}$
	when no confusion occurs.
%
%
%
%
\subsection{Gerstenhaber algebras and Hochschild (co)homology} 
	Let us start with the definition of Gerstenhaber algebras.	
		\begin{defi} 
			A {\it Gerstenhaber algebra} is a graded $k$-module 
			$\mathcal{H}^{\bullet} = \bigoplus_{r \in \mathbb{Z}} \mathcal{H}^{r}$ 
			equipped with two bilinear maps: 	a cup product of degree zero 
				\begin{align*}
					\smile : \mathcal{H}^{|\alpha|} \otimes \mathcal{H}^{|\beta|} 
					\rightarrow
					\mathcal{H}^{|\alpha|+|\beta|},\ \ (\alpha, \beta) \longmapsto \alpha \smile \beta
				\end{align*}
			and a Lie bracket of degree $-1$, called the {\it Gerstenhaber bracket},
				\begin{align*} 
					[\ ,\ ] : \mathcal{H}^{|\alpha|} \otimes \mathcal{H}^{|\beta|} 
					\rightarrow 
					\mathcal{H}^{|\alpha|+|\beta|-1},\ \ (\alpha, \beta) \longmapsto [\alpha, \beta]
				\end{align*}
			such that 
				\begin{enumerate}
				\renewcommand{\labelenumi}{(\roman{enumi})}
					\item $(\mathcal{H}^{\bullet}, \smile)$ is a graded commutative algebra with unit 
							$1 \in \mathcal{H}^{0}$, in particular, 
							$\alpha \smile \beta 
							= (-1)^{|\alpha||\beta|}\beta \smile \alpha;$
					\item $(\mathcal{H}^{\bullet}[1], [\ ,\ ])$ is a graded Lie algebra 
							with components $(\mathcal{H}^{\bullet}[1])^{r} = \mathcal{H}^{r+1}$, that is,  
								\[
									 [\alpha, \beta] = -(-1)^{(|\alpha|-1)(|\beta|-1)} [\beta, \alpha] 
								\]	
							and
								\[
									 (-1)^{(|\alpha|-1)(|\gamma|-1)}[[\alpha, \beta], \gamma]
									 +(-1)^{(|\beta|-1)(|\alpha|-1)}[[\beta, \gamma], \alpha]
								\]
								\[
									 +(-1)^{(|\gamma|-1)(|\beta|-1)}[[\gamma, \alpha], \beta] = 0;
								\]
					\item The Lie bracket $[\ ,\ ]$ is compatible with the cup product $\smile$ :
								\begin{align*}
									[\alpha, \beta \smile \gamma] 
									= [\alpha, \beta] \smile \gamma + (-1)^{(|\alpha|-1)|\beta|}\beta \smile [\alpha, \gamma], 
								\end{align*}
				\end{enumerate}
 where $\alpha, \beta, \gamma$ are homogeneous elements in $\mathcal{H}^{\bullet}$ and we denote by $|\alpha|$ the degree of a homogeneous element $\alpha$ in $\mathcal{H}^{\bullet}$.
	\end{defi}
%
%
	One of examples of Gerstenhaber algebras is the Hochschild cohomology of a $k$-algebra $A$. 
	There is a projective resolution $\mathrm{Bar}_{\bullet}(A)$ of $A$ over $A^{\rm e}$, which is the so-called 
	{\it normalized bar resolution}:
	\begin{align*}
		\cdots \rightarrow 
		A \otimes \bA^{\otimes r} \otimes A \xrightarrow{d_{r}}
		A \otimes \bA^{\otimes r-1} \otimes A \rightarrow
		\cdots \rightarrow
		A \otimes \bA  \otimes A \xrightarrow{d_{1}}
		A \otimes A \xrightarrow{d_{0}}	
		A \rightarrow 
		0,
	\end{align*}
	where we set 
		\begin{align*} \lefteqn{	}	
				d_{r}(a_{0} \otimes \ba_{1,\,r} \otimes a_{r+1}) =&\  
				a_{0} a_{1} \otimes \ba_{2,\,r} \otimes a_{r+1} \\
				&+\sum_{i=1}^{r-1} (-1)^{i} 
				a_{0} \otimes \ba_{1,\,i-1} \otimes \overline{a_{i} a_{i+1}} \otimes \ba_{i+2,\,r} \otimes a_{r+1}  \\
				&+(-1)^{r} a_{0} \otimes \ba_{1,\,r-1} \otimes a_{r} a_{r+1}, \\
			     d_0(a_{0} \otimes a_{1}) =&\ a_{0} a_{1}.			
		\end{align*}
	We denote  $\overline{\Omega}^{r}(A) := \mathrm{Im}\, d_{r}$ for all $r \geq 0$.
%
%
	Given an $A$-bimodule $M$, consider the complex $C^{\bullet}(A, M) := \Hom_{A^{\rm e}}(\mathrm{Bar}_{\bullet}(A), M)$ 
	with differential $\Hom_{A^{\rm e}}(d_{\bullet}, M)$. Note that for any $r \geq 0$, we have
		\begin{align*}
			C^{r}(A, M) 
			= \Hom_{A^{\rm e}}(\mathrm{Bar}_{r}(A), M) 
			= \Hom_{A^{\rm e}}(A \otimes \bA^{\otimes r} \otimes A, M)
			\cong \Hom(\bA^{\otimes r}, M).
		\end{align*}
	We identify $C^{0}(A, M)$ with $M$. Thus, the complex $C^{\bullet}(A, M)$ is of the form
		\begin{align*} 
			0 \rightarrow 
			M \xrightarrow{\delta^{0}} 
			\Hom(\bA, M) \rightarrow 
			\cdots \rightarrow 
			\Hom(\bA^{\otimes r}, M) \xrightarrow{\delta^{r}} 
			\Hom(\bA^{\otimes r+1}, M) \rightarrow 
			\cdots 
		\end{align*}
	whose differentials $\delta^{r}$ are defined by 
		\begin{align*} 	 \lefteqn{} 
			\delta^{r}(f)(\ba_{1,\,r+1}) = &\ 
			a_{1} f(\ba_{2,\,r+1} )
			+\sum_{i=1}^{r} (-1)^{i+1}
			f( \ba_{1,\,i-1} \otimes \overline{a_{i} a_{i+1}} \otimes \ba_{i+2,\,r+1} )  \\
			&+(-1)^{r+1} f(\ba_{1,\,r}) a_{r+1} 
		\end{align*}
	for any $f \in \Hom(\bA^{\otimes r}, M)$ and $\ba_{1,\,r+1} \in \bA^{\otimes r+1}$. Then the $r$-th cohomology group 
		\begin{align*}
			\mathrm{H}^{r}(A, M) := \mathrm{H}^{r}(C^{\bullet}(A, M), \delta^{\bullet})
		\end{align*}
	is said to be the $r$-th {\it Hochschild cohomology group} of $A$ with coefficients in $M$. 
	We will write $\mathrm{HH}^{r}(A) := \mathrm{H}^{r}(A, A)$.	
	Since $A$ is projective over $k$, we get $\mathrm{H}^{r}(A, M) \cong \mathrm{Ext}_{A^{\rm e}}^{r}(A, M)$. 
	Namely, Hochschild cohomology groups do not depend on the choice of projective resolution of $A$.
%
%
	For $A$-bimodules $M$ and $N$, the cup product 
		\begin{align*}
			\smile : C^{m}(A, M) \otimes C^{n}(A, N) 	\rightarrow C^{m+n}(A, M \otimes_{A} N) 
		\end{align*}
	is defined by
		\begin{align*}
			(\alpha \smile \beta) (\ba_{1,\,m+n}) := \alpha(\ba_{1,\,m}) \otimes_{A} \beta(\ba_{m+1,\,m+n})
		\end{align*}
	for all $\alpha \in C^{m}(A, M), \beta \in C^{n}(A, N)$ and $\ba_{1,\,m+n} \in \bA^{\otimes m+n}$.	
	The cup product $\smile$ induces a well-defined operator
		\begin{align*}
			\smile : \mathrm{H}^{m}(A, M) \otimes \mathrm{H}^{n}(A, N) \rightarrow \mathrm{H}^{m+n}(A, M \otimes_{A} N).
		\end{align*}

	The Gerstenhaber bracket in the Hochschild cohomology $\mathrm{HH}^{\bullet}(A)$ is defined as follows: 
	let $\alpha \in C^{m}(A, A) $ and $\beta \in C^{n}(A, A) $. We define a $k$-bilinear map
		\begin{align*}
			[\ , \ ] : C^{m}(A, A) \otimes C^{n}(A, A) \rightarrow C^{m+n-1}(A, A)
		\end{align*}
	as 
		\begin{align*}
			[\alpha , \beta] := \alpha \circ \beta -(-1)^{(m-1)(n-1)} \beta \circ \alpha \in C^{m+n-1}(A, A) ,
		\end{align*}
	where we determine $\alpha \circ \beta$ by
	\begin{align*}
		\alpha \circ \beta(\ba_{1,\,m+n-1}) 
		:= \sum_{i=1}^{m} (-1)^{(i-1)(n-1)} \alpha(\ba_{1,\,i-1} \otimes \overline{\beta}(\ba_{i,\,i+n-1}) \otimes \ba_{i+n,\,m+n-1})
	\end{align*}
	with $\overline{\beta} := \pi \circ \beta$.
	This $k$-bilinear map $[\ , \ ]$ induces a well-defined operator
		\begin{align*}	
			[\ , \ ] : \mathrm{HH}^{m}(A) \otimes \mathrm{HH}^{n}(A) \rightarrow \mathrm{HH}^{m+n-1}(A).
		\end{align*}
	Gerstenhaber proved the following result.
%
%
	\begin{theo}[{\cite[page 267]{Ger}}]
		The Hochschild cohomology $\HH^{\bullet}(A)$ equipped with the cup product $\smile$ and the Lie bracket $[\ , \ ]$ is a Gerstenhaber algebra.
	\end{theo}
%
%
	For an $A$-bimodule $M$, consider a complex $C_{\bullet}(A, M) := M \otimes_{A^{\rm e}} \mathrm{Bar}_{\bullet}(A)$ with differential 
	$\mathrm{id}_{M} \otimes_{A^{\rm e}} d_{\bullet}$. Note that for any $r \geq 0$, we have
		\begin{align*}
			C_{r}(A, M) = M \otimes_{A^{\rm e}} \mathrm{Bar}_{r}(A) 
			= M \otimes_{A^{\rm e}} (A \otimes \bA^{\otimes r} \otimes A)
			\cong M \otimes \bA^{\otimes r}.
		\end{align*}
	We identify $C_{0}(A, M)$ with $M$. Thus, the complex $C_{\bullet}(A, M)$ is of the form
		\begin{align*}
			\cdots \rightarrow
			M \otimes \bA^{\otimes r+1} \xrightarrow{\partial_{r+1}} 
			M \otimes \bA^{\otimes r} \rightarrow 
			\cdots \rightarrow 			
			M \otimes \bA\xrightarrow{\partial_{1}} 
			M \rightarrow 
			0,	
		\end{align*}
	where the differentials $\partial_{r+1}$ are defined by 
		\begin{align*} \lefteqn{}
			\partial_{r+1}(m \otimes \ba_{1,\,r+1} ) =&\ 
			m a_{1} \otimes \ba_{2,\,r+1} 
			+\sum_{i=1}^{r} (-1)^{i}
			m \otimes \ba_{1,\,i-1} \otimes \overline{a_{i} a_{i+1}} \otimes \ba_{i+2,\,r+1} \\
			&+(-1)^{r+1} a_{r+1} m \otimes \ba_{1,\,r}  
		\end{align*}
	for all 	$m \otimes \ba_{1,\,r+1} \in M \otimes \bA^{\otimes r+1}$. Then the $r$-th homology group 
		\begin{align*}
			\mathrm{H}_{r}(A, M) := \mathrm{H}_{r}(C_{\bullet}(A, M), \partial_{\bullet})
		\end{align*}
	is said to be the $r$-th {\it Hochschild homology group} of $A$ with coefficients in $M$. 
	We will write $\HH_{r}(A) := \mathrm{H}_{r}(A, A)$. Since $A$ is projective over $k$, we get $\mathrm{H}_{r}(A, M) \cong \mathrm{Tor}_{A^{\rm e}}^{r}(A, M)$, 
	which means that Hochschild homology groups are independent of projective resolutions of $A$.

%
%
	There is an action of Hochschild cohomology on Hochschild homology, called the cap product.  
	For two $A$-bimodules $M$, $N$ and $r, p \geq 0$ with $r \geq p$, a $k$-bilinear map
		\begin{align*}
			\frown : C_{r}(A, M) \otimes C^{p}(A, N) \rightarrow C_{r -p}(A, M \otimes_{A} N)
		\end{align*}
	is defined by 
		\begin{align*}
			(m \otimes \ba_{1,\,r} ) \frown \alpha := m \otimes_{A} \alpha(\ba_{1,\,p}) \otimes \ba_{p+1,\,r}
		\end{align*}
	for all $m \otimes \ba_{1,\,r} \in C_{r}(A, M)$ and $\alpha \in C^{p}(A, N)$.
	The $k$-bilinear map $\frown$ induces a well-defined operator
		\begin{align*}
			\frown : \mathrm{H}_{r}(A, M) \otimes \mathrm{H}^{p}(A, N) \rightarrow \mathrm{H}_{r - p}(A, M \otimes_{A} N).
		\end{align*}

%
%
%
%
\subsection{Gorenstein algebras and complete resolutions }	\label{subsection-Gorenstein}
	This section is devoted to recalling some basic results on complete resolutions of a module over a Gorenstein algebra. 
	For more details, we refer the reader to \cite{AvramovMartsinkovsky,BerghJorgensen,CornockKropholler}.
	Recall that a finite dimensional algebra $A$ is a {\it Gorenstein algebra} of Gorenstein dimension $d$ 
	if the injective dimension of $A$, as a right and as a left $A$-module, are equal to $d$. 
	Assume that $A$ is a Gorenstein algebra of Gorenstein dimension $d$. It follows from \cite{AvramovMartsinkovsky} 
	that any finitely generated left $A$-module $M$ admits a {\it complete resolution}
		\begin{align*}
			\mathcal{T} : \cdots \rightarrow T_{2} \rightarrow T_{1} \rightarrow T_{0} \rightarrow T_{-1} \rightarrow T_{-2} \rightarrow \cdots
		\end{align*}
	satisfying the three conditions:
		\begin{enumerate}
		\renewcommand{\labelenumi}{(\roman{enumi})}
			\item $\mathcal{T}$ is an exact sequence of finitely generated projective $A$-modules. 
					\label{def-comsol-1}
			\item The $A$-dual complex $\Hom_{A}(\mathcal{T}, A)$ is acyclic.
					\label{def-comsol-2}
			\item There exist a projective resolution  $\mathcal{P}$ of $M$ and a chain map $f : \mathcal{T} \rightarrow \mathcal{P}$ such that
					$f_{r}$ is an isomorphism for $r \geq d$.
					\label{def-comsol-3}
		\end{enumerate}
	Given a left $A$-module $N$ and an integer $r$,  the $r$-th cohomology group of the complex $\Hom_{A}(\mathcal{T}, N)$ is said  to be 
	the $r$-th {\it Tate cohomology group} of $M$ with coefficients in $N$ and  is denoted by $\widehat{\mathrm{Ext}}_{A}^{r}(M, N)$. For any right $A$-module $N$,
	the $r$-th homology group of the complex $N \otimes_{A} \mathcal{T}$ is said  to be the $r$-th {\it Tate homology group} of $M$ 
	with coefficients in $N$ and  is denoted by $\widehat{\mathrm{Tor}}_{r}^{A}(M, N)$.
	It follows from \cite{CornockKropholler} that the Tate (co)homology groups do not depend on the choice of complete resolutions of $M$.
	The property (iii) implies that we have
				\begin{align*} \lefteqn{}
					\widehat{\mathrm{Ext}}_{A}^{r}(M, N) \cong \mathrm{Ext}_{A}^{r}(M, N), \quad \widehat{\mathrm{Tor}}_{r}^{A}(M, N) \cong \mathrm{Tor}_{r}^{A}(M, N)
				\end{align*}
	for all $r \geq d+1$. 
		\begin{defi}[Bergh-Jorgensen \cite{BerghJorgensen}]
			Let $A$ be a $k$-algebra such that the enveloping algebra $A^{\rm e}$ is Gorenstein, $N$ an $A$-bimodule and $r \in \mathbb{Z}$ arbitrary. 
			Then the $r$-th {\it complete cohomology group} $\widehat{\mathrm{HH}}^{r}(A, N)$ and the $r$-th {\it complete homology group} 
			$\widehat{\mathrm{HH}}_{r}(A, N)$ of $A$ with coefficients in $N$ are defined by
				\begin{align*} \lefteqn{}
					&\widehat{\mathrm{HH}}^{r}(A, N) := \widehat{\mathrm{Ext}}_{A^{\rm e}}^{r}(A, N),	\quad
					\widehat{\mathrm{HH}}_{r}(A, N) := \widehat{\mathrm{Tor}}_{r}^{A^{\rm e}}(A, N).
				\end{align*}
			We will write $\widehat{\mathrm{HH}}^{r}(A) := \widehat{\mathrm{HH}}^{r}(A, A)$ and $\widehat{\mathrm{HH}}_{r}(A) := \widehat{\mathrm{HH}}_{r}(A, A)$.
		\end{defi}
	In this paper, we only deal with complete cohomology groups. 
	Originally, Bergh-Jorgensen \cite{BerghJorgensen} called the complete cohomology groups Tate-Hochschild cohomology groups. 
	However, throughout the paper, we use the term 
	\lq\lq Tate-Hochschild\rq\rq \,for the cohomology groups defined by Wang \cite{wang, wang2} described in the next subsection. We remark that both of these cohomology groups coincide for an algebra whose enveloping algebra is Gorenstein of Gorenstein dimension zero. 


	A class of algebras which we will be studying is a class of Frobenius algebras (see Section \ref{Frobenius}). 
	Note that a Frobenius algebra is Gorenstein of Gorenstein dimension zero.
	 It follows from \cite[Corollary 3.3]{BerghJorgensen} that the property of being Frobenius algebras is preserved 
	under taking their enveloping algebras, 
	so that the complete cohomology groups $\widehat{\mathrm{HH}}^{*}(A, N)$ of a Frobenius algebra $A$ are defined, and there are isomorphisms 
				\begin{align*} \lefteqn{}
					\widehat{\mathrm{HH}}^{r}(A, N) \cong \Hh^{r}(A, N),\quad
					\widehat{\mathrm{HH}}^{r}(A) \cong \HH^{r}(A)  
				\end{align*}
	for all $r \geq 1$. 

	We will develop the theory of BV differentials on the complete cohomology rings of  Frobenius algebras and 
	give their examples. 
	For this purpose, we use appropriate complete resolutions.

%
%
%
%
\subsection{Tate-Hochschild cohomology and its Gerstenhaber structure}
	This section is devoted to recalling Tate-Hochschild cohomology groups and a Gerstenhaber structure on the Tate-Hochschild cohomology. 
	For more details, we refer the reader to \cite[Section 3 and 4]{wang}.
	Let us recall that for $r \geq 0$, the $r$-th (usual) Hochschild cohomology of $A$ with coefficients in an $A$-bimodule $M$ can be defined as 
	$\mathrm{H}^{r}(A, M) = \Hom_{\mathcal{D}^{b}(A^{\rm e})}(A, M[r])$, where $\mathcal{D}^{b}(A^{\rm e})$ is the bounded derived category of finitely generated 
	left $A^{\rm e}$-modules and the suspension functor $[-]$ denotes the degree shift. For any integer $r$, the $r$-th {\it Tate-Hochschild cohomology}
	group of $A$ is defined by 
		\begin{align*} 
			\underline{\mathrm{Ext}}_{A^{\rm e}}^{r}(A, A) := \Hom_{\mathcal{D}_{sg}(A^{\rm e})}(A, A[r]),
		\end{align*} 
	where $\mathcal{D}_{sg}(A^{\rm e})$ is the {\it singularity category} of $A^{\rm e}$. Recall that $\mathcal{D}_{sg}(A^{\rm e})$ is the Verdier quotient of 
	$\mathcal{D}^{b}(A^{\rm e})$ by the full subcategory of $\mathcal{D}^{b}(A^{\rm e})$ consisting of those complexes quasi-isomorphic to bounded complexes of
	finitely generated projective left $A^{\rm e}$-modules.

Recall that $\overline{\Omega}^{p}(A) = \Im d_p$, where $d_p:\mathrm{Bar}_{p}(A) \rightarrow \mathrm{Bar}_{p-1}(A)$ is the $p$-th differential of the normalized bar resolution $\mathrm{Bar}_{\bullet}(A)$.
We fix an integer $m$ and put ${\rm I}_{(m)} := \big\{ p \in \mathbb{Z} \, \big|\, p \geq 0, m+p \geq 0 \big\}$.
Consider an inductive system
\[\left\{ X_{p}^{(m)},\  \theta_{m+p,\,p}: X_{p}^{(m)} \rightarrow X_{p+1}^{(m)}\right\}_{ p\,\in\,{\rm I}_{(m)} },\]
where  
\begin{align*} 
X_{p}^{(m)} = \Ext_{A^{\rm e}}^{m+p} (A, \overline{\Omega}^{p}(A)), 
\end{align*} 
and $\theta_{m+p,\,p}: X_{p}^{(m)} \rightarrow X_{p+1}^{(m)}$ is the connecting homomorphism 
\begin{align} \label{Pre-Tate-eq1}
\theta_{m+p,\,p} : \Ext_{A^{\rm e}}^{m+p} (A, \overline{\Omega}^{p}(A)) \rightarrow \Ext_{A^{\rm e}}^{m+p+1} (A, \overline{\Omega}^{p+1}(A))
\end{align}   
induced by the short exact sequence
\begin{align*} 			
0	\longrightarrow 
\overline{\Omega}^{p+1} (A)	\longrightarrow
A \otimes \overline{A}^{\otimes p} \otimes A  \longrightarrow
\overline{\Omega}^{p} (A) 		\longrightarrow 
0.
\end{align*} 
Here, we regard $\Ext_{A^{\rm e}}^{m+p} (A, \overline{\Omega}^{p}(A))$ as  $\Hh^{m+p}(A, \overline{\Omega}^{p}(A))$, or equivalently, any element of $\Ext_{A^{\rm e}}^{m+p} (A, \overline{\Omega}^{p}(A))$ is represented by an element in $\Hom_k(\bA^{\otimes m+p}, \overline{\Omega}^{p}(A))$.
Note that the inductive system above has the form
\begin{align*} 
\Ext_{A^{\rm e}}^{m+i} (A, \overline{\Omega}^{i}(A))  
\xrightarrow{\theta_{ m+i,\,i }} 
\Ext_{A^{\rm e}}^{m+i+1}(A, \overline{\Omega}^{i+1} (A) )  
\xrightarrow{\theta_{ m+i+1,\,i+1 }} 		
\Ext_{A^{\rm e}}^{m+i+2}(A, \overline{\Omega}^{i+2} (A) )   
\rightarrow
\cdots, 
\end{align*}
where $i \geq 0$ is the least integer such that $m+i \geq 0$.

%
%
\begin{rem} {\rm
Using the explicit description of the connecting homomorphism (\ref{Pre-Tate-eq1}) in \cite[page 16]{wang}, we see that, for any $m \in \mathbb{Z}$ and $p \in {\rm I}_{(m)}$, the connecting homomorphism 
\begin{align*}
\theta_{m+p,\,p} :\Ext_{A^{\rm e}}^{m+p} (A, \overline{\Omega}^{p}(A)) \rightarrow \Ext_{A^{\rm e}}^{m+p+1} (A, \overline{\Omega}^{p+1}(A))
\end{align*}
sends an element $[f] \in \Ext_{A^{\rm e}}^{m+p} (A, \overline{\Omega}^{p}(A))$ represented by $f\in \Hom_k(\bA^{\otimes m+p}, \overline{\Omega}^{p}(A))$ to the element $[\theta_{m+p,\,p}(f)] \in \Ext_{A^{\rm e}}^{m+p+1} (A, \overline{\Omega}^{p+1}(A))$.
Here, $[\theta_{m+p,\,p}(f)]$  is represented by the $k$-linear map 
\[\theta_{m+p,\,p}(f): \bA^{\otimes m+p+1} \rightarrow \overline{\Omega}^{p+1}(A)\]
taking an element $\ba_{1,\,m+p+1}\in \bA^{\otimes m+p+1}$ into
\begin{align*}
 (-1)^{m+p}\,d_{p+1}(f(\ba_{1,\,m+p}) \otimes \ba_{m+p+1}\otimes 1) \in \Im d_{p+1}= \overline{\Omega}^{p+1}(A),
\end{align*}
 where $d_{p+1}:\mathrm{Bar}_{p+1}(A) \rightarrow \mathrm{Bar}_{p}(A)$ is the $(p+1)$-th differential of $\mathrm{Bar}_{\bullet}(A)$.
}\end{rem}

%
%
	\begin{prop}[{\cite[Proposition 3.1 and Remark 3.3]{wang}}]  \label{Pre-Tate-prop1}
		For any $m \in \mathbb{Z}$, there is an isomorphism  
			 \begin{align*}
					\lim_{\substack{ \longrightarrow \\ p\,\in\,{\rm I}_{(m)}}}
					 \Ext_{A^{\rm e}}^{m+p}(A, \overline{\Omega}^{p}(A))
						\cong
					\Hom_{\mathcal{D}_{\rm{sg}}(A^{\rm e})} (A, A[m]) = \underline{\mathrm{Ext}}_{A^{\rm e}}^{m}(A, A).
			\end{align*}
	\end{prop}

	We now define a Gerstenhaber structure on Tate-Hochschild cohomology defined by Wang (\cite{wang}).
	Let $m, n, p $ and $q$ be integers such that $m, n, p, q \geq 0$. 
	A cup product 
		\begin{align*}
		\smile_{\rm{sg}} : 
			C^{ m } (A, \overline{\Omega}^{p}(A) ) \otimes C^{ n } (A, \overline{\Omega}^{q}(A) ) 	\rightarrow	C^{ m+n } (A, \overline{\Omega}^{ p+q }(A) )
		\end{align*}
	is defined  by
	\begin{align*}
		f \smile_{\rm{sg}} g (\bb_{1,\,m+n}) := \Phi_{p+q}(f(\bb_{1,\,m}) \otimes_{A} g(\bb_{m+1,\,m+n})),
	\end{align*}
	where $f \otimes g \in C^{m} (A, \overline{\Omega}^{p}(A) ) \otimes C^{n} (A, \overline{\Omega}^{q}(A) )$ and
	 $\Phi_{p+q} : \overline{\Omega}^{p}(A) \otimes_{A} \overline{\Omega}^{q}(A) \rightarrow \overline{\Omega}^{p+q}(A)$ is an isomorphism of $A$-bimodules determined by
	\begin{align*}
		\Phi_{p+q}(a_{0} \otimes \ba_{1,\,p} \otimes a_{p+1} \otimes_{A} b_{0} \otimes \bb_{1,\,q} \otimes b_{q+1})
		=a_{0} \otimes \ba_{1,\,p} \otimes \overline { a_{p+1} b_{0} } \otimes \bb_{1,\,q} \otimes b_{q+1}
	\end{align*}	
	for $a_{0} \otimes \ba_{1,\,p} \otimes a_{p+1} \in \overline{\Omega}^{p}(A)$ and $b_{0} \otimes \bb_{1,\,q} \otimes b_{q+1} \in \overline{\Omega}^{q}(A)$, 
	which is given in \cite[Lemma 2.6]{wang2}.

 	Let $m \in \mathbb{Z}_{> 0},\   p \in \mathbb{Z}_{\geq 0}$ and $f  \in C^{ m } (A, \overline{\Omega}^{p}(A) )$
	and let
		$\pi : A \rightarrow \bA$ be the canonical epimorphism.
	We set
		\begin{align*}  \lefteqn{}
			&\pi_{p}^{(l)} := \pi \otimes \mathrm{id}_{\bA}^{\otimes p-1} \otimes \mathrm{id}_{A} :
				A \otimes \bA^{\otimes  p-1} \otimes A
					\rightarrow
				\bA^{\otimes  p} \otimes A, \\
			&\pi_{p}^{(r)} := \mathrm{id}_{A}  \otimes \mathrm{id}_{\bA}^{\otimes p-1} \otimes \pi :
				A \otimes \bA^{\otimes  p-1} \otimes A
					\rightarrow
				A \otimes \bA^{\otimes  p}, \\
			&\pi_{p}^{(b)} := \pi \otimes \mathrm{id}_{\bA}^{\otimes p-1} \otimes \pi :
				A \otimes \bA^{\otimes  p-1} \otimes A
					\rightarrow
				 \bA^{\otimes  p+1} 
		\end{align*}	
	and then denote
		\begin{align*} 
			&f^{(l)} := \pi_{p}^{(l)}  f, \quad
			f^{(r)} := \pi_{p}^{(r)}  f, \quad
			f^{(b)} := \pi_{p}^{(b)}  f.
		\end{align*}	
	Let $m, n, p $ and $q$ be integers such that $m, n > 0$ and $p, q \geq 0$. 	We now define a bilinear map  	
		\begin{align*}
			[\ , \ ]_{\rm{sg}} :  
				C^{ m } (A, \overline{\Omega}^{p}(A) ) \otimes C^{ n } (A, \overline{\Omega}^{q}(A) ) 
					\rightarrow 
				C^{ m+n-1 } (A, \overline{\Omega}^{ p+q }(A) ).
		\end{align*}
 as follows: let 
\[f  \in C^{m}(A, \overline{\Omega}^{p}(A))= \Hom_{k}(\bA^{\otimes m}, \overline{\Omega}^{p}(A)) \] 
and 
\[ g \in  C^{ n } (A, \overline{\Omega}^{q}(A)) = \Hom_{k}(\bA^{\otimes n}, \overline{\Omega}^{q}(A)).\] 
We first define a $k$-linear map $f \bullet_{i} g \in C^{m+n-1} (A, \overline{\Omega}^{p+q}(A))$ for each integer $i$ with $1 \leq i \leq m$.
Consider the following four $k$-linear maps:
\begin{enumerate}

\item 
$\left(\mathrm{id}_{\bA}^{\otimes i-1} \otimes g^{(b)} \otimes \mathrm{id}_{\bA}^{\otimes m-i}\right): \bA^{\otimes m+n-1} \rightarrow \bA^{\otimes m+q}$  is given by 
\begin{align*} \quad \ba_{1,\,m+n-1} \mapsto \ba_{1,\,i-1} \otimes g^{(b)}(\ba_{i,\,i+n-1}) \otimes \ba_{i+n,\,m+n-1}; \end{align*}

\item 
$\left(f^{(r)} \otimes \mathrm{id}_{\bA}^{\otimes q}\right): \bA^{\otimes m+q} \rightarrow A \otimes \bA^{\otimes p+q}$  is given by 
\begin{align*} \quad \ba_{1,\,m+q} \mapsto   f^{(r)}(\ba_{1,\,m}) \otimes \ba_{m+1,\,m+q}; \end{align*}

\item 
$\left(\id_A \otimes \id_{\bA}^{\otimes p+q} \otimes 1 \right): A \otimes \bA^{\otimes p+q} \rightarrow A \otimes \bA^{\otimes p+q} \otimes A$  is given by 
\begin{align*} \quad a_0 \otimes \ba_{1,\,p+q} \mapsto   a_0 \otimes \ba_{1,\,p+q} \otimes 1; \end{align*}

\item $d_{p+q}:A \otimes \bA^{\otimes p+q} \otimes A \rightarrow A \otimes \bA^{\otimes p+q-1} \otimes A$ is the $(p+q)$-th differential of the normalized bar resolution $\mathrm{Bar}_{\bullet}(A)$.

\end{enumerate}
We  then define  a $k$-linear map  $f \bullet_{i} g \in C^{m+n-1} (A, \overline{\Omega}^{p+q}(A))$ by the composition of the above four maps 
\begin{align*}
    f \bullet_{i} g :&= d_{p+q} \circ \left(\id_A \otimes \id_{\bA}^{\otimes p+q} \otimes 1 \right) \circ \left(f^{(r)} \otimes \mathrm{id}_{\bA}^{\otimes q}\right) \circ   \left(\mathrm{id}_{\bA}^{\otimes i-1} \otimes g^{(b)} \otimes \mathrm{id}_{\bA}^{\otimes m-i}\right)  \\[1mm]
&= d_{p+q}( (f^{(r)} \otimes \mathrm{id}_{\bA}^{\otimes q})
(\mathrm{id}_{\bA}^{\otimes i-1} \otimes g^{(b)} \otimes \mathrm{id}_{\bA}^{\otimes m-i}) \otimes 1)
\end{align*}
for $1 \leq i \leq m$.
On the other hand, we assume that $q>0$. 
We also define a $k$-linear map $f \bullet_{-i} g \in C^{m+n-1} (A, \overline{\Omega}^{p+q}(A))$ for each integer $i$ with $1 \leq i \leq q$.
Consider the following four $k$-linear maps:

\begin{enumerate}

\item $\left(g^{(r)} \otimes \mathrm{id}_{\bA}^{\otimes m-1}\right): \bA^{\otimes m+n-1} \rightarrow A \otimes \bA^{\otimes m+q-1}$ is given by 
\begin{align*} \quad \ba_{1,\,m+n-1} \mapsto   g^{(r)}(\ba_{1,\,n}) \otimes \ba_{n+1,\,m+n-1}; \end{align*}

\item $\left(\id_{A} \otimes \id_{\bA}^{\otimes i-1} \otimes f^{(b)} \otimes \mathrm{id}_{\bA}^{\otimes q-i}\right): A\otimes\bA^{\otimes m+q-1} \rightarrow A\otimes\bA^{\otimes p+q}$ is given by 
\begin{align*} \quad a_0 \otimes \ba_{1,\,m+q-1} \mapsto  a_0 \otimes \ba_{1,\,i-1} \otimes  f^{(b)}(\ba_{i,\,i+m-1}) \otimes \ba_{i+m,\,m+q-1}; \end{align*}

\item $\left(\id_A \otimes \id_{\bA}^{\otimes p+q} \otimes 1 \right): A \otimes \bA^{\otimes p+q} \rightarrow A \otimes \bA^{\otimes p+q} \otimes A$ is  the same as above;

\item $d_{p+q}:A \otimes \bA^{\otimes p+q} \otimes A \rightarrow A \otimes \bA^{\otimes p+q-1} \otimes A$ is  the same as above.

\end{enumerate}

Then we define a $k$-linear map  $f \bullet_{-i} g \in C^{m+n-1} (A, \overline{\Omega}^{p+q}(A))$ by the   composition of the above four maps 
\begin{align*}
    f \bullet_{-i} g :&= 
    d_{p+q} \circ 
    \left(\id_A \otimes \id_{\bA}^{\otimes p+q} \otimes 1 \right) \circ 
    \left(\id_{A} \otimes \id_{\bA}^{\otimes i-1} \otimes f^{(b)} \otimes \mathrm{id}_{\bA}^{\otimes q-i}\right) \circ  
    \left(g^{(r)} \otimes \mathrm{id}_{\bA}^{\otimes m-1}\right)  \\[1mm]
&= d_{p+q}((\id_{A} \otimes \id_{\bA}^{\otimes i-1} \otimes f^{(b)} \otimes \mathrm{id}_{\bA}^{\otimes q-i})
(g^{(r)} \otimes \mathrm{id}_{\bA}^{\otimes m-1}) \otimes 1)
\end{align*}
for $1 \leq i \leq q$.
So far, the $k$-linear map $f \bullet_{i} g \in C^{m+n-1} (A, \overline{\Omega}^{p+q}(A))$ has been defined in the following way:
\begin{align*} 
f \bullet_{i} g = \begin{cases}
d_{p+q}( (f^{(r)} \otimes \mathrm{id}_{\bA}^{\otimes q})
(\mathrm{id}_{\bA}^{\otimes i-1} \otimes g^{(b)} \otimes \mathrm{id}_{\bA}^{\otimes m-i}) \otimes 1) 	& \mbox{if\ } 1 \leq i \leq m; \\[5pt]
d_{p+q}( (\mathrm{id} \otimes \mathrm{id}_{\bA}^{\otimes -i-1} \otimes f^{(b)} \otimes \mathrm{id}_{\bA}^{\otimes q+i} )  
( g^{(r)} \otimes \mathrm{id}_{\bA}^{\otimes m-1}) \otimes 1) 	& \mbox{if\ } q >0\  \mbox{and}\  -q \leq i \leq -1.
\end{cases}		
\end{align*}
Now, we define a $k$-linear map $f \bullet g \in C^{m+n-1} (A, \overline{\Omega}^{p+q}(A))$ by
\begin{align*} 
f \bullet g :=\begin{cases}
\sum_{i=1}^{m} (-1)^{r(m, \,p ;\, n, \,q ; \,i)} f \bullet_{i} g + \sum_{i=1}^{q} (-1)^{s(m, \,p ; \,n, \,q ;\, i)} f \bullet_{-i} g 	& \mbox{if\ } q>0; \\[5pt]
\sum_{i=1}^{m} (-1)^{r(m, \,p ;\, n, \,q ; \,i)} f \bullet_{i} g 	& \mbox{if\ } q=0,
\end{cases}		
\end{align*}
	where $r(m, p; n, q; i)$ and $s(m, p; n, q; i)$ are determined by
		\begin{align*} 
			r(m, p ; n, q ; i) := p + q + (i-1)(q-n-1)   &\mbox{ for } 1\leq i \leq m, \\
			s(m, p ; n, q ; i) := p + q + (i-1)(q-n-1)  &\mbox{ for } 1\leq i \leq q.
		\end{align*}
Finally, we are able to define a $k$-linear map $[f, g]_{\rm sg} \in C^{m+n-1} (A, \overline{\Omega}^{p+q}(A))$ as
		\begin{align*}
			[f, g]_{\rm sg} := f \bullet g -(-1)^{(m-p-1)(n-q-1)} g  \bullet f.
		\end{align*}
Wang \cite{wang} showed that the cup product $\smile_{\rm sg}$ and the bilinear map $[ \ , \ ]_{\rm sg}$ induce well-defined operators, 
	still denoted by $\smile_{\rm sg}$ and $[ \ , \ ]_{\rm sg}$, on a graded $k$-vector space
		\begin{align*}
			\bigoplus_{\substack{ m\,\in\,\mathbb{Z},\,p\,\in\,\mathbb{Z}_{\geq\,0}, \\ m+p\,\geq\,0}} \mathrm{Ext}_{A^{\rm e}}^{m+p}(A, \overline{\Omega}^{p}(A))
		\end{align*}
	with grading 
		\begin{align*}
		\left( \bigoplus_{m,\,p} \mathrm{Ext}_{A^{\rm e}}^{m+p}(A, \overline{\Omega}^{p}(A)) \right)^{i} 
			= \bigoplus_{\substack{ l\,\geq\,0, \\ i+l\,\geq\,0 }} \mathrm{Ext}_{A^{\rm e}}^{i+l}(A, \overline{\Omega}^{l}(A))
		\end{align*}
	for $i \in \mathbb{Z}$, which make it into a Gerstenhaber algebra. Furthermore, he proved that the two induced operators 
	$\smile_{\rm sg} $ and $[ \ , \ ]_{\rm sg}$ 
	are compatible with the connecting homomorphisms
 	$\theta_{m,\,p} : \mathrm{Ext}_{A^{\rm e}}^{m}(A, \overline{\Omega}^{p}(A)) \rightarrow \mathrm{Ext}_{A^{\rm e}}^{m+1}(A, \overline{\Omega}^{p+1}(A))$. Therefore, we have the following result.
%
%
	\begin{theo}[{\cite[Theorem 4.1]{wang}}] \label{thm_1}
		Let $A$ be a finite dimensional algebra over a field $k$. Then the graded $k$-vector space
			\begin{align*}
				\bigoplus_{m\,\in\,\mathbb{Z}} \  
					\lim_{\substack{ \longrightarrow \\ p\,\in\,{\rm I}_{(m)}}} \mathrm{Ext}_{A^{\rm e}}^{m+p}(A, \overline{\Omega}^{p}(A))
			\end{align*}
	 	equipped with the cup product $\smile_{\rm sg}$ and the Lie bracket $[\ ,\ ]_{\rm sg}$ is a Gerstenhaber algebra.
	\end{theo}
%
%
\begin{rem} {\rm
	The Gerstenhaber brackets on $\bigoplus_{m,\,p} \mathrm{Ext}_{A^{\rm e}}^{m+p}(A, \overline{\Omega}^{p}(A))$ 
	involving elements of degree zero are defined via the connecting homomorphisms 
		\[
			\theta_{0,\,*} : \mathrm{Ext}_{A^{\rm e}}^{0}(A, \overline{\Omega}^{*}(A)) \rightarrow \mathrm{Ext}_{A^{\rm e}}^{1}(A, \overline{\Omega}^{*+1}(A)),
		\]
	that is, for $f \in \mathrm{Ext}_{A^{\rm e}}^{m+p}(A, \overline{\Omega}^{p}(A))$
	and $\alpha \in \mathrm{Ext}_{A^{\rm e}}^{0}(A, \overline{\Omega}^{q}(A))$, we define
		\begin{align*}
			[f, \alpha]_{\rm sg} := [f, \theta_{0,\,q}(\alpha)]_{\rm sg}.
		\end{align*}
		}
	\end{rem}

\section{Frobenius algebras and complete resolutions} \label{Frobenius}
	Let $k$ be a field and $A$ a finite dimensional $k$-algebra of dimension $d$ over $k$, and let $\sigma$ be an algebra automorphism of $A$. 
	For any $A$-bimodule $M$, we denote by $M_{\sigma}$ the $A$-bimodule which is $M$ as a $k$-vector space and 
	whose $A$-bimodule structure is defined by $a \cdot m \cdot b := a m \sigma(b)$ for $m \in M_{\sigma}$ and $a, b \in A$.
	We also denote by $A^{\vee}$ a right $A^{\rm e}$-module $\Hom_{A^{\rm{e}}}({}_{A^{\rm e}}A, {}_{A^{\rm e}} A^{\rm e})$
	whose structure is given by the multiplication of $A^{\rm{e}}$ on the right hand side. 
	Note that we have an isomorphism of right $A^{\rm e}$-modules 
		\begin{align*}\lefteqn{}
			A^{\vee} 
			\xrightarrow{\sim}& 
			(A \otimes A)^{A}\\ 
			&:= \biggl\{ \sum_{i} x_{i} \otimes y_{i}  \bigg| \sum ax_{i} \otimes y_{i} = \sum x_{i} \otimes y_{i}a \mbox{\ for any\ } a \in A \biggr \};
			f \mapsto f(1),
		\end{align*}
	where a right $A^{\rm e}$-module structure of $(A \otimes A)^{A}$ is defined by the multiplication of $A^{\rm e}$ on the right hand side. 
	Recall that $A$ is a {\it Frobenius algebra} if there is an associative and non-degenerate bilinear form 
	$\langle \ ,\ \rangle : A \otimes A \rightarrow k$. 
	The associativity means that $\langle ab, c \rangle = \langle a, bc \rangle$ for all $a, b$ and $c \in A$. 
	If $(u_{i})_{i=1}^{d}$ is a $k$-basis of $A$, then there is a $k$-basis $(v_{i})_{i=1}^{d}$ of $A$ such that
	$\langle v_{i}, u_{j} \rangle = \delta_{ij}$ with $\delta_{ij}$ Kronecker's delta. In such a case, we call $(u_{i})_{i=1}^{d}, (v_{i})_{i=1}^{d}$ {\it dual bases} of $A$.
	There exists an algebra automorphism $\nu$, up to inner automorphism, of $A$ such that	
	$\langle a, b \rangle = \langle b, \nu(a) \rangle$ for all $a, b \in A$, and the automorphism $\nu$ is said  to be the
	{\it Nakayama automorphism} of $A$.
	In fact, we can write both the Nakayama automorphism $\nu$  and its inverse $\nu^{-1}$, explicitly:  for $x \in A$,
		\begin{align*}
			\nu(x) := \sum_{i=1}^{d} \langle x, v_{i} \rangle u_{i}, \quad \nu^{-1}(x) :=\sum_{i=1}^{d} \langle u_{i}, x \rangle v_{i}.
		\end{align*}
	Another definition of Frobenius algebras is that $A$ is isomorphic to $D(A)$ as right or as left $A$-modules.
	Here the left (right) $A$-module structure of $D(A)$ is defined by $(a f )(x) := f(xa)$ $(( f a )(x) := f(ax))$ for any $f \in D(A)$ and any $a \in A$.
	We can see that the bilinear form $\langle \ ,\ \rangle : A \otimes A \rightarrow k$ induces an isomorphism of left $A$-modules 
		\begin{align*}
				\phi : A \xrightarrow{\sim} D(A); \  a \mapsto \langle -, a \rangle.
		\end{align*} 
	Moreover, this isomorphism gives rise to an isomorphism of $A$-bimodules $A_{\nu} \xrightarrow{\sim} D(A)$.
	
	The first statement of the next lemma appears in \cite[Lemma 2.1.35]{EuSchedler2009}. 
	However, we prove it again in order to get the explicit form of the isomorphism below.
%
%
	\begin{lem} \label{Frobenius-lem1}
		Let $A$ be a finite dimensional Frobenius algebra. With the same notation above, we have the following assertions.
			\begin{enumerate}
				\item There is an isomorphism of right $A^{\rm e}$-bimodules $A_{\nu^{-1}} \cong A^{\vee}$. 
						\label{Frobenius-lem1-1}
				\item  If $(u_{i})_{i}, (v_{i})_{i}$ and $(u^{\prime}_{j})_{j}, (v^{\prime}_{j})_{j}$ are two dual bases of $A$, then we have
						$
						    \sum_{i} u_{i} \otimes v_{i} = \sum_{j} u^{\prime}_{j} \otimes v^{\prime}_{j}.
						$\label{Frobenius-lem1-2}
				\item An element $\sum_{i} u_{i} \otimes v_{i}$ of $A \otimes A$ has the following properties: 
							\begin{enumerate} 
								\item $\sum_{i} u_{i} \otimes v_{i} 	= \sum_{i} v_{i} \otimes \nu^{-1} (u_{i}) = \sum_{i} \nu (v_{i}) \otimes u_{i}$;
								\item $\sum_{i} au_{i}b \otimes v_{i} = \sum_{i} u_{i} \otimes \nu^{-1} (b) v_{i}a$ for any $a, b \in A$.
							\end{enumerate}
						\label{Frobenius-lem1-3}
			\end{enumerate}
	\end{lem}
	\begin{proof}
	For the statements (\ref{Frobenius-lem1-2}) and (\ref{Frobenius-lem1-3}), 
	consider the composition $\eta : A_{\nu^{-1}} \otimes A \rightarrow \Hom_{k}(A, A)$ of isomorphisms
		\[ 
			\xymatrix@R=1pt{
				A_{\nu^{-1}} \otimes A \ar[r] & D(A) \otimes A \ar[r] & \Hom_{k}(A, A) \\
				\sum_{i} x_{i} \otimes y_{i} \ar@{|->}[r] &\sum_{i} \langle -, x_{i} \rangle \otimes y_{i} \ar@{|->}[r] & 
				\big[x \mapsto \sum_{i} \langle x, x_{i} \rangle y_{i} \big].
		}\]
	Since $x = \sum_{i} \langle x, u_{i} \rangle v_{i}$ for any dual bases $(u_{i})_{i=1}^{d} , (v_{i})_{i=1}^{d}$ of $A$ and any $x \in A$, 
	the statements (\ref{Frobenius-lem1-2}) and (\ref{Frobenius-lem1-3}) follow from the injectivity of $\eta$. 
	On the other hand, we define 
			\begin{align*}\lefteqn{}
				&\varphi : A_{\nu^{-1}} \rightarrow A^{\vee};\ x \mapsto \big[a \mapsto \sum_{i} a u_{i} \nu(x) \otimes v_{i} \big], \\
				&\psi : A^{\vee} \rightarrow (A\otimes A)^{A} \rightarrow A_{\nu^{-1}}; \ 
				\alpha \mapsto \alpha(1_{A})= \sum_{i} x_{i} \otimes y_{i} \mapsto \sum_{i} \langle 1, x_{i} \rangle y_{i}. 
			\end{align*}
	Then we get $\varphi$ is a right $A^{\rm{e}}$-module homomorphism. Indeed, if $x \in A_{\nu^{-1}}$ and $a \otimes b^{\circ} \in A^{\rm{e}}$, 
	then we have 
			$\varphi(x\,(a \otimes b^{\circ})) 
			= \sum_{i} u_{i}\,\nu(b x) a \otimes v_{i} 
			= \sum_{i} u_{i} \nu(x) a \otimes b v_{i}
			= (\sum_{i} u_{i} \nu(x) \otimes v_{i})  (a \otimes b^{\circ}).$
	One can easily check that $\varphi \psi = \mathrm{id}_{A^{\vee}}$ and $\psi \varphi = \mathrm{id}_{A_{\nu^{-1}}}$. 
	\end{proof}
%

	As we remarked in Section \ref{subsection-Gorenstein}, 
	if $A$ is a Frobenius algebra $A$, then so is the enveloping algebra $A^{\rm e}$. 
	In particular, $A^{\rm e}$ is a Gorenstein algebra of Gorenstein dimension zero.
	Therefore, $A$ has a complete resolution  over $A^{\rm e}$. Note that we may take a projective resolution of $A$ over $A^{\rm e}$ as 
	the non-negative part of the complete resolution. 
	In fact, Nakayama \cite{Nakayama} constructed a complete resolution $\mathcal{T}$ of $A$ in the following way:
	we set $\mathcal{T}_{r} := \mathrm{Bar}_{r}(A) = A \otimes \overline{A}^{\otimes r} \otimes A$ for every $r \geq 0$ and 
	$\mathcal{T}_{-s} := D(\mathrm{Bar}_{s-1}(A))_{\nu^{-1}}$ for each $s \geq 1$.
	Then we get an augmented  exact sequence\\
	\begin{center}
		\scalebox{0.98}{	
			\xymatrix{
				\cdots \ar[r] &
				\mathcal{T}_{r} \ar[r]^-{d_{r}} &	
				\mathcal{T}_{r-1} \ar[r]	&	
				\cdots \ar[r]^-{d_{1}} &
				\mathcal{T}_{0} \ar[r]^-{d^{\prime}_{0}} \ar[d]_-{d_0}&
				\mathcal{T}_{-1} \ar[r]^-{d_{-1}} &
				\cdots \ar[r] & 
				\mathcal{T}_{-s} \ar[r]^-{d_{-s}} &
				\mathcal{T}_{-s-1} \ar[r]	&
				\cdots 
				\\
				&
				&
				&
				&
				A \ar[r]^-{\sim}_-{\phi}&
				D(A)_{\nu^{-1}}	\ar[u]_-{D(d_0)}&
				&
				&
				&
	}}\\
	\end{center}
	where we put
		\begin{align*}
			D(d_0)(f) = f d_0	\quad (f \in D(A)_{\nu^{-1}}),	\quad
			d^{\prime}_{0} = D(d_0) \phi d_0, \qquad
			d_{-s} (g) = g d_{s} \quad(g \in \mathcal{T}_{-s}).
		\end{align*}
	Sanada \cite[Lemma 1.1]{Sanada1992} proved that for any $A$-bimodule $M$ and any integer $r$, there is an isomorphism between  
	$\Hom_{A^{\rm e}}(\mathcal{T}_{r}, M)$ and $M_{\nu^{-1}} \otimes_{A^{\rm e}} \mathcal{T}_{-r-1}$ which is natural in $M$, so that
	each $\mathcal{T}_{r} \ (r \in \mathbb{Z})$ is projective over $A^{\rm e}$. 
	Furthermore, this isomorphism gives an isomorphism between the cochain complex $\Hom_{A^{\rm e}}(\mathcal{T}, M)$ and 
	the chain complex $M_{\nu^{-1}} \otimes_{A^{\rm e}} \mathcal{T}$.
	Therefore, the following cochain complex $(\mathcal{D}^{\bullet}(A, M), \widehat{d}^{\bullet})$ has the same cohomology groups as 
	$\Hom_{A^{\rm e}}(\mathcal{T}, M)$:
			\begin{align*}
				\cdots \rightarrow
				C_{2}(A, M_{\nu^{-1}}) \xrightarrow{\partial_{2}}
				C_{1}(A, M_{\nu^{-1}}) \xrightarrow{\partial_{1}}
				M_{\nu^{-1}} \xrightarrow{\mu} 				
				M \xrightarrow{\delta^{0}} 		
				C^{1}(A, M) \xrightarrow{\delta^{1}} 	
				C^{2}(A, M) \rightarrow \cdots,
			\end{align*}	
	where	we define $\mu : M_{\nu^{-1}} \rightarrow M$ by $\mu(m) := \sum_{i=1}^{d} u_{i} m v_{i}$ for $m \in M$ and set
		\begin{align*}
			\mathcal{D}^{r}(A, M) = \begin{cases}
										C^{r}(A, M)						& \mbox{if $r \geq 0$},  \\
										C_{-r-1}(A, M_{\nu^{-1}}) 		& \mbox{if $r \leq -1$,}  
									\end{cases} \quad
			\widehat{d}^{r} = \begin{cases}
										\delta^{r}		& \mbox{if $r \geq 0$},  \\ 
										\mu						& \mbox{if $r = -1$},  \\
										\partial_{-r-1}					& \mbox{if $r \leq -2$.}
									\end{cases}
		\end{align*}
	We give the explicit forms of the $0$-th and $(-1)$-th cohomology groups as follows:
		\begin{align*} \lefteqn{}
			\widehat{\mathrm{HH}}^{0}(A) \cong M^{A}/N_{A}(M), \quad
			\widehat{\mathrm{HH}}^{-1}(A) = {}_{N_{A}}M/ I_{A}(M),
		\end{align*}
	where we set
		\begin{align*} \lefteqn{}
			& M^{A} := \{ \, m \in M \ | \ a m = m a \mbox{ for all } a \in A \}, \\[5pt]
			&N_{A}(M) := \mathrm{Im}\, (\mu ) = \bigg \{ \,  \sum_{i} u_{i} m v_{i} \ | \ m \in M \bigg \}, \\
			&{}_{N_{A}}M := \{ \, m \in M \ | \ \sum_{i} u_{i} m v_{i} = 0 \}, \\
			&I_{A}(M) := \bigg \{ \, \sum_{i} ( m_{i} \nu^{-1} (a_{i}) -a_{i} m_{i} ) \mbox{ (finite sum) }	 | \ a_{i} \in A, m_{i} \in M \bigg \}.
		\end{align*}
	Note that for any $x \in A$, $\sum_{i}  u_{i} x v_{i} = 0$ holds if and only if $\sum_{i}  u_{i} \nu(x) v_{i} = 0$ holds.
%
%
	\begin{rem} \label{Frobenius-rem2}  {\rm
	If $M =A$, then $\widehat{\mathrm{HH}}^{0}(A)$ and $\widehat{\mathrm{HH}}^{-1}(A)$ are appeared in the following exact sequence:	
						\begin{align*}
							0 \rightarrow 
							\underline{\mathrm{Ext}}_{A^{\rm{e}}}^{-1}(A, A) \rightarrow
							A_{\nu^{-1}}  \otimes_{A^{\rm{e}}} A  \xrightarrow{\eta}
							\Hom_{A^{\rm{e}}} (A, A)  \rightarrow
							\underline{\mathrm{Ext}}_{A^{\rm{e}}}^{0}(A, A)  \rightarrow
							0,
						\end{align*}
	where the morphism $\eta(x \otimes_{A^{\rm{e}}} a)(b) = \sum_{i} b u_{i} \nu(x) a v_{i}$.
	}
	\end{rem}

	Suppose that $A$ is a finite dimensional self-injective $k$-algebra.
	Recall that $A$ is a {\it self-injective} algebra if $A$ is injective as a left and as a right $A$-module.  	
	Note that the enveloping algebra $A^{\rm{e}}$ is also a self-injective algebra. 
	Observe that if $A$ is a self-injective algebra, then all of the connecting homomorphisms (\ref{Pre-Tate-eq1})
		\[
			\theta_{m+p,\,p} : \mathrm{Ext}_{A^{\rm e}}^{m+p} (A, \overline{\Omega}^{p}(A)) \rightarrow \mathrm{Ext}_{A^{\rm e}}^{m+p+1} (A, \overline{\Omega}^{p+1}(A))
		\] 
		are isomorphisms except for the case $m+p=0$, so that we have an isomorphism	$\mathrm{Ext}_{A^{\rm e}}^{r+p} (A, \overline{\Omega}^{p}(A)) \cong \underline{\mathrm{Ext}}_{A^{\rm{e}}}^{r}(A, A)$
		for all $r, p \in \mathbb{Z}$ such that $p \geq 0$ and $r+p > 0$.
%
%
%
%
	We need  modification for the inductive system $\{ X_{p}^{(m)}, \theta_{m+p,\,p} \}_{ p\,\in\,{\rm I}_{(m)}} $ defined in Section 2.3. 
	Let us recall that
		\[
			\lim_{\substack{ \longrightarrow \\ p\,\in\,{\rm I}_{(m)}}}
					 \mathrm{Ext}_{A^{\rm e}}^{m+p}(A, \overline{\Omega}^{p}(A))
		\] 
	is the inductive limit of the inductive system 
	$\{ X_{p}^{(m)}, \theta_{m+p,\,p} \}_{ p\,\in\,{\rm I}_{(m)} }$
	of which the term $X_{p}^{(m)}$ is defined by 
	$X_{p}^{(m)} := \mathrm{Ext}_{A^{\rm e}}^{m+p} (A, \overline{\Omega}^{p}(A))$
	and whose morphism $\theta_{m+p,\,p}$ is the connecting homomorphism
	$\mathrm{Ext}_{A^{\rm e}}^{m+p} (A, \overline{\Omega}^{p}(A)) 
	\rightarrow 
	\mathrm{Ext}_{A^{\rm e}}^{m+p+1} (A, \overline{\Omega}^{p+1}(A))$. 
	Consider another inductive system 
		\[
			\{ Y_{p}^{(m)}, \varphi_{m+p,\,p} \}_{ p\,\in\,{\rm I}_{(m)}}
		\]
	of which the term $Y_{p}^{(m)}$ is the same as $X_{p}^{(m)}$
	and whose morphism $\varphi_{m+p, p}$ is given by
		\begin{align*} 
			\varphi_{m+p,\,p} := \begin{cases}
												(-1)^{m+i} \theta_{m+i,\,i} &\mbox{ if } p = i, \\
												(-1)^{m} \theta_{m+p,\,p} &\mbox{ if } p > i,
									\end{cases}
		\end{align*}
	where an integer $i \geq 0$ is the least one belonging to ${\rm I}_{(m)}$.
	Then we can readily see 
		\begin{align*} 
			\lim_{\substack{ \longrightarrow \\ p\,\in\,{\rm I}_{(m)}}} Y_{p}^{(m)} 
			\cong 
			\lim_{\substack{ \longrightarrow \\ p\,\in\,{\rm I}_{(m)}}}
					 \mathrm{Ext}_{A^{\rm e}}^{m+p}(A, \overline{\Omega}^{p}(A)).
		\end{align*} 
	We will utilize the inductive system 
	$\{ Y_{p}^{(m)}, \varphi_{m+p,\,p} \}_{p\,\in\,{\rm I}_{(m)}}$
	instead of $\{ X_{p}^{(m)}, \theta_{m+p,\,p} \}_{ p\,\in\,{\rm I}_{(m)} }$	and denote 
	\[\varphi_{m+p,\, p}^{q} := \varphi_{m+p+q-1,\, p+q-1} \circ \cdots \circ \varphi_{m+p,\, p}: \mathrm{Ext}_{A^{\rm e}}^{m+p} (A, \overline{\Omega}^{p}(A))  \rightarrow \mathrm{Ext}_{A^{\rm e}}^{m+p+q} (A, \overline{\Omega}^{p+q}(A)).\]
	Note that $\varphi_{m+p,\, p}^{1} = \varphi_{m+p,\, p}$.	

	The following is a special case of \cite[Corollary 6.4.1]{Buch}, which says that we have a description of the Tate-Hochschild cohomology by using $\mathrm{Ext}$ and $\mathrm{Tor}$ for a self-injective algebra.
%
%
	\begin{prop}
	\label{Frobenius-prop1} 
		Let $A$ be a finite dimensional self-injective $k$-algebra. Denote $A^{\vee} = \Hom_{A^{\rm{e}}}(A, A^{\rm{e}})$. 
		Then we have the following.
			\begin{enumerate}
				\item $\underline{\mathrm{Ext}}_{A^{\rm{e}}}^{r}(A, A) \cong \mathrm{Ext}_{A^{\rm{e}}}^{r}(A, A)$ for all $r \geq 1$. 
				\item $\underline{\mathrm{Ext}}_{A^{\rm{e}}}^{-r}(A, A) \cong \mathrm{Tor}^{A^{\rm{e}}}_{r-1}(A, A^{\vee})$ for all $r \geq 2$. 
				\item There exists an exact sequence of $k$-vector spaces
						\begin{align*}
							0 \rightarrow 
							\underline{\mathrm{Ext}}_{A^{\rm{e}}}^{-1}(A, A) \rightarrow
							A^{\vee} \otimes_{A^{\rm{e}}} A \xrightarrow{\eta}
							\Hom_{A^{\rm{e}}} (A, A)  \rightarrow
							\underline{\mathrm{Ext}}_{A^{\rm{e}}}^{0}(A, A)  \rightarrow
							0,
						\end{align*}
					where the morphism $\eta$ is given by $\eta((\sum_{i} x_{i} \otimes y_{i}) \otimes_{A^{\rm{e}}} a)(b) = \sum_{i} b x_{i} a y_{i}$. 
				\item $\underline{\mathrm{Ext}}_{A^{\rm e}}^{0}(A, A) = \underline{\Hom}_{A^{\rm e}}(A, A)$, 
						which is the set of $A$-bimodule homomorphisms from $A$ to $A$ 
						modulo those homomorphisms passing through projective $A$-bimodules.
			\end{enumerate}
	In particular, for $r \geq 2$ and $p \geq 1$, 	
	\begin{align*}  \lefteqn{}
		&\kappa_{-1,\,p} : 
			\underline{\mathrm{Ext}}_{A^{\rm{e}}}^{-1}(A, A) = \mathrm{Ker}\,(\eta) 	
				\xrightarrow{\sim} 	
			\mathrm{Ext}_{A^{\rm e}}^{p}(A, \overline{\Omega}^{p+1}(A)) \cong \underline{\mathrm{Ext}}_{A^{\rm{e}}}^{-1}(A, A), \\
		&\varphi_{0,\,0}^{p} : 
			\underline{\mathrm{Ext}}_{A^{\rm{e}}}^{0}(A, A) = \mathrm{Coker}\, (\eta)  	
				\xrightarrow{\sim} 
			\mathrm{Ext}_{A^{\rm e}}^{p}(A, \overline{\Omega}^{p}(A)) \cong \underline{\mathrm{Ext}}_{A^{\rm{e}}}^{0}(A, A), \\
		&\kappa_{r-1,\,p} : 
			\mathrm{Tor}_{r-1}^{A^{\rm{e}}}(A, A^{\vee}) 	\xrightarrow{\sim} 	\mathrm{Ext}_{A^{\rm e}}^{p}(A, \overline{\Omega}^{r+p}(A)) \cong \underline{\mathrm{Ext}}_{A^{\rm{e}}}^{-r}(A, A)
	\end{align*}	
	are defined, on the (co)chain level, as
	\begin{align*}  \lefteqn{}
		&\kappa_{-1,\,p} (\alpha \otimes_{A^{\rm e}} a) (\bb_{1,\,p}) = \sum_{i} d_{p+1}(x_{i} a \otimes \by_{i} \otimes \bb_{1,\,p} \otimes 1), \\
		&\varphi_{0,\,0}^{p} (f)(\bb_{1,\,p}) = d_{p}(f(1) \otimes \bb_{1,\,p} \otimes 1), \\
		&\kappa_{r-1,\,p} (\alpha \otimes_{A^{\rm e}} \ba_{1,\,r-1}) (\bb_{1,\,p}) = 
			\sum_{i} d_{r+p}(x_{i} \otimes\ba_{1,\,r-1} \otimes \by_{i} \otimes \bb_{1,\,p} \otimes 1),
	\end{align*}
	where	we write $\alpha(1) = \sum_{i} x_{i} \otimes y_{i}$. We denote $\varphi_{0,\,0}^{1}$ by $\varphi_{0,\,0}$.
	\end{prop} 
%
%
	The third isomorphisms 
	$\kappa_{r-1,\,p} : \mathrm{Tor}_{r-1}^{A^{\rm{e}}}(A, A^{\vee}) \xrightarrow{\sim} \underline{\mathrm{Ext}}_{A^{\rm{e}}}^{-r}(A, A)$
	in Proposition \ref{Frobenius-prop1} are given by Wang \cite[Remark 6.3]{wang}.

	Algebras which we are interested in are Frobenius algebras, which are self-injective algebras. 
	Using Remark \ref{Frobenius-rem2} and Lemma \ref{Frobenius-lem1}, we obtain the following consequence of Proposition \ref{Frobenius-prop1}.
%
%
	\begin{cor} \label{Frobenius-cor1} 
		Let $A$ be a finite dimensional Frobenius $k$-algebra. Then there is an isomorphism
		 $\underline{\mathrm{Ext}}_{A^{\rm{e}}}^{r}(A, A) \cong \widehat{\mathrm{HH}}^{r}(A)$ for all $r \in \mathbb{Z}$.
	\end{cor}

	Assume that $A$ is a finite dimensional Frobenius algebra. Then the $A$-bimodule isomorphism $A_{\nu^{-1}} \cong A^{\vee}$ gives an isomorphism 
	of complexes between $\mathcal{D}^{\bullet}(A, A)$ and the complex $\mathcal{C}^{\bullet}(A, A)$ defined by Wang \cite{wang2}
			\begin{align*}
				\cdots \rightarrow
				C_{2}(A, A^{\vee}) \xrightarrow{\partial_{2}}
				C_{1}(A, A^{\vee}) \xrightarrow{\partial_{1}}
				A^{\vee} \xrightarrow{\mu} 				
				A \xrightarrow{\delta^{0}} 		
				C^{1}(A, A) \xrightarrow{\delta^{1}} 	
				C^{2}(A, A) \rightarrow \cdots
			\end{align*}	
	whose negative part is the Hochschild chain complex $(C_{\bullet}(A, A^{\vee}), \partial_{\bullet})$ and of which the non-negative part
	is the Hochschild cochain complex $(C^{\bullet}(A, A ), \delta^{\bullet})$. Here the map $\mu : A^{\vee} \rightarrow A$ is defined by
	the multiplication of $A$, that is, $\mu(\alpha) = \sum_{i} x_{i} y_{i}$ for $\alpha \in A^{\vee}$ with $\alpha(1) = \sum_{i} x_{i} \otimes y_{i}$.

	Moreover, Wang \cite[Section 6.2]{wang2} defined a product on $\mathcal{C}^{*}(A, A)$, called {\it $\star$-product}, which extends 
	the cup product on $C^{*}(A, A)$ and the cap product between $C^{*}(A, A)$ and $C_{*}(A, A^{\vee})$. 
	Although the $\star$-product is not associative on $\mathcal{C}^{*}(A, A)$ in general, 
	the $\star$-product induces a graded commutative and associative product on $\Hh^{*}(\mathcal{C}^{\bullet}(A, A))$. 
	The following is the product 
		\[
			\star : \mathcal{D}^{*}(A, A) \otimes \mathcal{D}^{*}(A, A) \rightarrow \mathcal{D}^{*}(A, A)
		\]
	on $\mathcal{D}^{*}(A, A)$ via the isomorphism $\mathcal{D}^{\bullet}(A, A) \cong \mathcal{C}^{\bullet}(A, A)$:
	let $f \in C^{m}(A, A), g \in C^{n}(A, A)$ and $\alpha = a_{0} \otimes \ba_{1,\,p} \in C_{p}(A, A_{\nu^{-1}}),$
	$\beta = b_{0} \otimes \bb_{1,\,q} \in C_{q}(A, A_{\nu^{-1}}).$
		\begin{enumerate}
			\item ($m, n \geq 0$) $\star : C^{m}(A, A) \otimes C^{n}(A, A) \rightarrow C^{m+n}(A, A)$ is given by
						\begin{align*}\lefteqn{}
							&f \star g := f \smile g \ ;
						\end{align*} 
			\item ($m \geq 0, p \geq 0, p \geq m$)		
						\begin{enumerate}
							\item $\star : C_{p}(A, A_{\nu^{-1}}) \otimes C^{m}(A, A) \rightarrow C_{p-m}(A, A_{\nu^{-1}})$ is given by
								\begin{align*}\lefteqn{}
									&\alpha \star f := \alpha \frown f = a_{0} \nu^{-1}(f(\ba_{1,\,m})) \otimes \ba_{m+1,\,p} \ ;
								\end{align*}
							\item $\star : C^{m}(A, A) \otimes C_{p}(A, A_{\nu^{-1}}) \rightarrow C_{p-m}(A, A_{\nu^{-1}})$ is given by
								\begin{align*}\lefteqn{}
									&f \star \alpha := f(\ba_{p-m+1,\,p})a_{0} \otimes \ba_{1,\,p-m} \ ;
								\end{align*}
						\end{enumerate}
			\item ($m \geq 0, p \geq 0, p < m$)
						\begin{enumerate}
							\item $\star : C^{m}(A, A) \otimes C_{p}(A, A_{\nu^{-1}}) \rightarrow C^{m-p-1}(A, A)$ is given by
									\begin{align*}\lefteqn{}
										&(f \star \alpha) (\bb_{1,\,m-p-1}):= \sum_{i} f( \bb_{1,\,m-p-1} \otimes \overline{u_{i} \nu(a_{0})} \otimes \ba_{1,\,p}) v_{i} \ ;
									\end{align*}
							\item $\star : C_{p}(A, A_{\nu^{-1}}) \otimes C^{m}(A, A) \rightarrow C^{m-p-1}(A, A)$ is given by
									\begin{align*}\lefteqn{}
										&(\alpha \star f) (\bb_{1,\,m-p-1}):= \sum_{i} u_{i} \nu(a_{0}) f(\ba_{1,\,p} \otimes \bv_{i} \otimes \bb_{1,\,m-p-1} )\ ;
									\end{align*}
						\end{enumerate}
			\item ($p, q \geq 0$) $\star : C_{p}(A, A_{\nu^{-1}}) \otimes C_{q}(A, A_{\nu^{-1}}) \rightarrow C_{p+q+1}(A, A_{\nu^{-1}})$ is given by
					\begin{align*} \lefteqn{}
							& \alpha \star \beta 
							:= \sum_{i} v_{i} b_{0} \otimes \bb_{1,\,q} \otimes \overline{u_{i} \nu(a_{0})} \otimes \ba_{1,\,p}\ .
					\end{align*}
		\end{enumerate}
%
%

	Dual bases of $A$ are used in our definition of $\star$-product, but Lemma \ref{Frobenius-lem1} (\ref{Frobenius-lem1-2}) 
	shows that
		the $\star$-product does not depend the choice of dual bases of $A$.	

	We summarize the results in the following proposition.
%
%
	\begin{prop}[{\cite[Lemma 6.2, Propositions 6.5 and 6.9]{wang2}}]
			Let $A$ be a finite dimensional Frobenius algebra. Then the $\star$-product is compatible with the differential $\widehat{d}$ of the complex $\mathcal{D}(A, A)$. Moreover, the induced product on $\widehat{\mathrm{HH}}^{*}(A)$, still denoted by $\star$,
is graded commutative and associative. In particular, $\widehat{\mathrm{HH}}^{\bullet}(A)$ equipped with $\star$ is isomorphic to
$\underline{\mathrm{Ext}}_{A^{\rm{e}}}^{\bullet}(A, A)$ as graded algebras.
	\end{prop}

\section{Decomposition of complete cohomology associated with the spectrum of the Nakayama automorphism}
\label{Decomposition}
%
%
Let us recall the subcomplexes of the (co)chain Hochschild complexes defined in \cite{Lam}. 
Let $A$ be a (not necessarily Frobenius) $k$-algebra and let $\sigma$ be an algebra automorphism of the algebra $A$.
Let $\Lambda$ be the set of eigenvalues of $\sigma$. Assume that $\Lambda \subset k$. 
We have $0_{A} \not \in \Lambda$ and $1_{A} \in \Lambda$ because $\sigma$ is a ring homomorphism.
Let
$\widehat{\Lambda} := \langle \Lambda \rangle$ be the submonoid of $k^{\times}$ generated by $\Lambda$. We denote by $A_{\lambda}$ 
the eigenspace $\mathrm{Ker}\,(\sigma -\lambda \,\mathrm{id} )$ associated with an eigenvalue $\lambda \in \Lambda$. 
For $\lambda \in \Lambda$, we write $\overline{A}_{\lambda} = A_{\lambda}$ for $\lambda \not = 1$ and 
$\overline{A}_{1} = A_{1} / (k \cdot 1_{A})$ for $\lambda = 1$, and for every $\mu \in \widehat{\Lambda}$ and every integer $r \geq 0$,
we put 
	\begin{align*}\lefteqn{}
		&C^{(\mu)}_{r}(A, A_{\sigma}) 
		:= \bigoplus_{\mu_{i} \in \Lambda,\,\prod \mu_{i} = \mu }
		A_{\mu_{0}} \otimes \bA_{\mu_{1}} \otimes \cdots \otimes \bA_{\mu_{r}}, \\
		&C^{r}_{(\mu)}(A, A) 
		:=
		\big\{ f \in C^{r}(A, A) 
		\big| 
		f(\bA_{\mu_{1}} \otimes \cdots \otimes \bA_{\mu_{r}}) \subset A_{\mu \mu_{1} \cdots \mu_{r}}, \mbox{ for any } \mu_{i} \in \Lambda \big\}.
	\end{align*}
Since $\sigma (x y) = \sigma ( x ) \sigma ( y )$ for $x, y \in A$, we have 
$A_{\lambda} \cdot A_{\lambda^{\prime}} \subset A_{\lambda \lambda^{\prime}}$ for $\lambda, \lambda^{\prime} \in \Lambda$.
It is understood that $A_{\lambda \lambda^{\prime}} = 0$ when $\lambda \lambda^{\prime} \not \in \Lambda$. 
Then these subspaces $C^{(\mu)}_{*}(A, A_{\sigma})$ and $C^{*}_{(\mu)}(A, A)$ are compatible with the differentials 
$\partial_{*}$ and $\delta^{*}$ of the complexes $(C_{\bullet}(A, A_{\sigma}), \partial_{\bullet})$ and $(C^{\bullet}(A, A), \delta^{\bullet})$, respectively.
Thus, we obtain subcomplexes $(C^{(\mu)}_{\bullet}(A, A_{\sigma}), \partial_{\bullet}^{(\mu)})$ and $(C^{\bullet}_{(\mu)}(A, A), \delta_{(\mu)}^{\bullet})$.
	Then we put
		\begin{align*}\lefteqn{}
			&\Hh_{r}^{(\mu)}(A, A_{\sigma}) := \Hh_{r}(C^{(\mu)}_{\bullet}(A, A_{\sigma}), \partial_{\bullet}^{(\mu)}), \\
			&\Hh_{(\mu)}^{r}(A, A) :=\Hh^{r}(C^{\bullet}_{(\mu)}(A, A), \delta_{(\mu)}^{\bullet}).
		\end{align*}
	Hence for all $r \geq 0$, we get morphisms of $k$-vector spaces $\Hh_{r}^{(\mu)}(A, A_{\sigma}) \rightarrow \Hh_{r}(A, A_{\sigma})$
	and $\Hh_{(\mu)}^{r}(A, A)$ $\rightarrow$ $\HH^{r}(A)$.
	Kowalzing and Kr\"{a}hmer \cite{NielUlrich2014} defined a $k$-linear map
		\[ 
			B_{r}^{\sigma} : C_{r}(A, A_{\sigma}) \rightarrow C_{r+1}(A, A_{\sigma})
		\] 
	by
		\[	
			B_{r}^{\sigma} (a_{0} \otimes \ba_{1,\,r} ) 
			= \sum_{i = 1}^{r+1} (-1)^{i r} 
			1 \otimes \ba_{i} \otimes \cdots \otimes \ba_{r} \otimes \ba_{0} \otimes \overline{\sigma (a_{1})} \otimes \cdots \otimes 
			\overline{\sigma (a_{i-1})}.
		\]	
	Let $T : C_{r}(A, A_{\sigma}) \rightarrow C_{r}(A, A_{\sigma})$ be the $k$-linear map defined by
	\[
		T(a_{0} \otimes \ba_{1,\,r} ) = \sigma(a_{0}) \otimes \overline{\sigma (a_{1})} \otimes \cdots \otimes \overline{\sigma (a_{r})}.
	\]
	A direct calculation shows that $\partial_{r+1} B_{r}^{\sigma} - B_{r-1}^{\sigma} \partial_{r} = (-1)^{r+1}(\mathrm{id} - T)$ for all $r \geq 0$.
%
%
%
%
	\begin{prop}[{\cite[Propositions 2.1, 2.2 and 2.5]{Lam}}] \label{Decomposition-prop1}
		The following three assertions hold.
		\begin{enumerate}
			\item For every $1 \not = \mu \in \widehat{\Lambda}$ and every $r \geq 0$, we get 
				\[	
					\Hh_{r}^{(\mu)}(A, A_{\sigma}) = 0.
				\]	
			\item For all $r \geq 0$, the restriction of the map $B_{r}^{\sigma} : C_{r}(A, A_{\sigma}) \rightarrow C_{r+1}(A, A_{\sigma})$ 
					to the subspaces $C_{*}^{(1)}(A, A_{\sigma})$ induces a twisted Connes operator
						\[	
							B^{\sigma}_{r} : \Hh_{r}^{(1)}(A, A_{ \sigma }) \rightarrow \Hh_{r +1}^{(1)}(A, A_{ \sigma }),
						\]	
					 and it satisfies $B^{\sigma}_{ r+1} B^{\sigma}_{ r } = 0$.
			\item If $\sigma$ is diagonalizable, then we have
						\[	
							 \Hh_{r}(A, A_{ \sigma }) \cong \Hh_{r}^{ (1) }(A, A_{ \sigma }) 
						\]	
					for $r \geq 0$.
		\end{enumerate}
	\end{prop}
	The following is an easy consequence of Proposition \ref{Decomposition-prop1}. 
%
%
%
%
		\begin{cor} \label{Decomposition-cor2}
			If the algebra automorphism $\sigma$ of $A$ is diagonalizable, then so is its inverse $\sigma^{-1}$.
			Furthermore, if this is the case, then we have two twisted Connes operators
						\[	
							B^{\sigma}_{*} : \Hh_{*}^{(1)}(A, A_{ \sigma }) \rightarrow \Hh_{* +1}^{(1)}(A, A_{ \sigma }), \quad
							B^{\sigma^{-1}}_{*} : \Hh_{*}^{(1)}(A, A_{ \sigma^{-1}}) \rightarrow \Hh_{* +1}^{(1)}(A, A_{ \sigma^{-1}}).
						\]	
		\end{cor}

	From now on, we assume  $A$ to be a Frobenius algebra with the Nakayama automorphism $\nu$ of $A$. 
	Let $\Lambda = \{ \lambda_{1}, \ldots, \lambda_{t} \}$ be the set of distinct eigenvalues of the Nakayama automorphism $\nu$. 
	Suppose that $\Lambda \subset k$. Let $\widehat{\Lambda} := \langle \Lambda \rangle$ be the submonoid of $k^{\times}$ generated by $\Lambda$. 
    For any $\mu \in \widehat{\Lambda}$,
	we define a subspace $\mathcal{D}_{(\mu)}^{*}(A, A)$ of $\mathcal{D}^{*}(A, A)$ in the following way: for any $\mu \in \widehat{\Lambda}$,
		\begin{align*}
			\mathcal{D}_{(\mu)}^{r}(A, A) :=  \begin{cases}
														C_{(\mu)}^{r}(A, A) &\mbox{ if } r \geq 0, \\[5pt]
														C_{-r-1}^{(\mu)} (A, A_{\nu^{-1}}) &\mbox{ if } r \leq -1.
												  \end{cases}
		\end{align*}
%
%
	\begin{lem} \label{Decomposition-lem2}
			For any $\mu \in \widehat{\Lambda}$, the subspaces $\mathcal{D}_{(\mu)}^{*}(A, A)$ of $\mathcal{D}^{*}(A, A)$ are compatible 
			with the differentials $\widehat{d}^{*}$ of the complex $(\mathcal{D}^{\bullet}(A, A), \widehat{d}^{\bullet})$.
	\end{lem}
%
%
	\begin{proof}
	It is sufficient to show that $\widehat{d}^{-1}(\mathcal{D}_{(\mu)}^{-1}(A, A)) \subset \mathcal{D}_{(\mu)}^{0}(A, A)$.
	If $x \in A_{\mu} = \mathcal{D}_{(\mu)}^{-1}(A, A)$, then we have 
		\begin{align*}\lefteqn{}
			\nu(\widehat{d}^{-1}(x))
			&= \sum_{i} \nu(u_{i}) \nu(x) \nu(v_{i})
			= \sum_{i,\,j} \langle u_{i}, v_{j} \rangle u_{j} \cdot \nu(x) \nu(v_{i})\\
			&= \sum_{j} u_{j} \nu(x) \nu\biggl(\sum_{i} \langle u_{i}, v_{j} \rangle v_{i}\biggr)
			= \sum_{j} u_{j} \nu(x) \nu(\nu^{-1}( v_{j} ))\\
		    &= \sum_{j} u_{j} \nu(x) v_{j}.
		\end{align*}
	Since $0 = (\nu - \mu\, \mathrm{id} ) (x) = \nu (x) - \mu x$, we get 
		\[	
			(\nu - \mu\,\mathrm{id} )(\widehat{d}^{-1}(x))
			= \nu (\widehat{d}^{-1}(x)) - \mu \,\widehat{d}^{-1}(x)
			= \sum_{j} u_{j} \nu(x) v_{j} - \mu \sum_{j} u_{j} x v_{j}
			= 0.
		\]
	Therefore, we have $\widehat{d}^{-1}(x) \in \mathcal{D}_{(\mu)}^{0}(A, A)$.
\end{proof}
%
%
	From Lemma \ref{Decomposition-lem2}, we obtain a subcomplex $(\mathcal{D}_{(\mu)}^{ \bullet }(A, A), \widehat{ d }^{ \bullet }_{(\mu)})$ 
	of $(\mathcal{D}^{\bullet}(A, A), \widehat{d}^{\bullet})$. Put 
		\[
			\widehat{\mathrm{HH}}_{(\mu)}^{r}(A) := \Hh^{r}(\mathcal{ D }_{(\mu)}^{ \bullet }(A, A), \widehat{d}^{\bullet}_{(\mu)})
		\]
	 for all $r \in \mathbb{Z}$.
	Hence we have a morphism of $k$-vector spaces $\widehat{\mathrm{HH}}_{(\mu)}^{r}(A) \rightarrow \widehat{\mathrm{HH}}^{r}(A)$ 
	for $r \in \mathbb{Z}$.
	Before starting with the next proposition, let us recall a well-known duality between Hochschild cohomology and Hochschild homology:
	there is an isomorphism $\Theta : D(C_{*}(A, A_{\nu}))$ $\rightarrow$ $C^{*}(A, A)$
	given by
	\begin{align*} \lefteqn{}
		D(C_{r}(A, A_{\nu})) &= \Hom (A_{ \nu } \otimes_{A^{ \rm{e}} } A \otimes \bA^{ \otimes r} \otimes A, k )\\
		&\cong \Hom_{A^{ \rm{e}} } ( A \otimes \bA^{ \otimes r} \otimes A, \Hom(A_{ \nu }, k) ) \\
		&\cong \Hom_{A^{ \rm{e}} } ( A \otimes \bA^{ \otimes r} \otimes A, A)
		= C^{r}(A, A),
	\end{align*}
	where $r \geq 0$ and the second isomorphism is induced by $A_{ \nu } \cong D(A)$.	
	Then $\Theta : D(C_{*}(A, A_{\nu}))$ $\rightarrow$ $C^{*}(A, A)$ is a morphism of complexes and hence induces a duality 
	$D(\Hh_{r} (A, A_{\nu}))$ $\cong$ $\HH^{r}(A)$. In fact, we can write $\Theta : D(C_{r}(A, A_{\nu}))$ $\rightarrow$ $C^{r}(A, A)$ and its inverse 
	$\Theta^{-1} : C^{r}(A, A) \rightarrow D(C_{r}(A, A_{\nu}))$ as follows: 
		\begin{align*} \lefteqn{}
			&\Theta : D(C_{r}(A, A_{\nu})) \rightarrow C^{r}(A, A); \quad \psi \mapsto \biggl[ \bb_{1,\,r } \mapsto \sum_{ j } \psi(u_{j} \otimes \bb_{1,\,r } ) v_{j} \biggr], \\
			&\Theta^{-1} : C^{r}(A, A) \rightarrow D(C_{r}(A, A_{\nu}));  \quad f \mapsto \bigl[ a_{0} \otimes \ba_{1,\,r} \mapsto \langle f(\ba_{1,\,r} ), a_{0} \rangle\bigr].
		\end{align*}
%
%
	\begin{prop} \label{Decomposition-prop2} 
	Let $A$ be a finite dimensional Frobenius $k$-algebra. 
	If the Nakayama automorphism $\nu$ of $A$ is diagonalizable, then three statements hold.
		\begin{enumerate}
			\item The isomorphism $\Theta : D(C_{\bullet}(A, A_{\nu})) \rightarrow C^{\bullet}(A, A)$ induces an isomorphism of complexes
						\[
							D(C_{\bullet}^{(\mu)}(A, A_{\nu})) \cong C_{(\mu^{-1})}^{\bullet}(A, A)
						\]
					for all $\mu \in \widehat{\Lambda}$.
					\label{Decomposition-prop2-1} 
			\item For  $r \in \mathbb{Z}$ and $\mu \not = 1 \in \widehat{\Lambda}$, 
					we get
						\begin{align*}
							\widehat{\mathrm{HH}}_{(\mu)}^{r}(A) = 0.
						\end{align*}
					\label{Decomposition-prop2-2} 
			\item For each $r \in \mathbb{Z}$, there exists an isomorphism of $k$-vector spaces
						\begin{align*}
							\widehat{\mathrm{HH}}_{(1)}^{r}(A) \cong \widehat{\mathrm{HH}}^{r}(A).
						\end{align*}			
					\label{Decomposition-prop2-3} 		
		\end{enumerate} 
	\end{prop}
%
%
	\begin{proof}
	It follows from Lemma \ref{Decomposition-lem1} below that the inverse of each eigenvalue $\lambda \in \Lambda$ is also an eigenvalue of 
	the Nakayama automorphism $\nu$ of $A$.
	Since $A$ is the (finite) direct sum of the eigenspaces $A_{ \lambda_{1} }, \ldots, A_{ \lambda_{t} }$, 
	we have $D(C_{r}(A, A_{\nu})) \cong \bigoplus_{\mu \in \widehat{\Lambda}} D(C_{r}^{(\mu)}(A, A_{\nu}))$ for all $r \geq 0$. 
%
%
	For the first statement, 
	it is sufficient to show that the inverse $\Theta^{-1} :  C^{\bullet}(A, A) \rightarrow D(C_{\bullet}(A, A_{\nu}))$ induces 
	an isomorphism  $C_{(\mu)}^{r}(A, A)$ $\cong$ $D(C_{r}^{(\mu^{-1})}(A, $ $A_{\nu}))$. 
	Since $\Theta^{-1}(f) \in D(C_{r}(A, A_{\nu}))$
	is a non-zero map for  $0 \not = f \in C_{(\mu)}^{r}(A, A)$,  there exist
	$\mu^{\prime} \in \widehat{\Lambda}$ and $a_{0} \otimes \ba_{1, r} \in C_{r}^{ (\mu^{\prime} )}(A, A_{\nu})$
	such that $\langle f(\ba_{1, r}), a_{0} \rangle \not = 0$, so that we get $(\mu \mu^{\prime} -1 ) \langle f(\ba_{1, r}), a_{0} \rangle = 0$ and hence 
	$\mu^{\prime} = \mu^{-1}$.  As a result, we have shown that if $ \lambda \in \widehat{\Lambda} $ with $ \lambda \not = \mu^{-1}$, 
	then the restriction of $\Theta^{-1}(f)$ to $C_{r}^{(\lambda)}(A, A_{\nu})$ is the zero map.
	Thus, we have a monomorphism 
		\begin{align*}
				\Theta_{(\mu)}^{-1} := \Theta^{-1} |_{ C_{(\mu)}^{r}(A, A) } : C_{(\mu)}^{r}(A, A) \rightarrow D(C_{r}^{(\mu^{-1})}(A, A_{\nu})).
		\end{align*}
	Furthermore, we get $\Theta_{(\mu)}^{-1}$ is surjective. Indeed, for any $\psi \in D(C_{r}^{(\mu^{-1})}(A, A_{\nu}))$,
	there exists $f \in C^{r}(A, A)$ such that $ \psi = \Theta^{-1}(f) $. Let $\mu_{1}, \ldots, \mu_{r} \in \Lambda$ and 
	$\bb_{1, r} \in \bA_{\mu_{1}} \otimes \cdots \otimes \bA_{\mu_{r}}$.
	It follows from $A = \bigoplus_{i} A_{\lambda_{i}}$ and $\psi |_{C_{r}^{(\lambda)}(A, A_{\nu})} = 0$ for all $\lambda \not = \mu^{-1}$ that
		\begin{align*}
			\langle f(\bb_{1, r}), a \rangle 
			= \langle \nu(f(\bb_{1, r})) , \nu(a) \rangle 	
			= \langle \nu(f(\bb_{1, r})) , (\mu_{1} \cdots \mu_{r})^{-1} \mu^{-1} a \rangle 
		\end{align*}
	for any $a \in A$.
	Consequently, we get
	$\nu(f(\bb_{1, r})) 
	= \mu_{1} \cdots \mu_{r} \mu f(\bb_{1, r})$ 
	and hence
	$f \in C_{ (\mu) }^{r}(A, A)$. 
	This shows that  $\Theta_{(\mu)}^{-1} : C_{(\mu)}^{r}(A, A) \rightarrow D(C_{r}^{(\mu^{-1})}(A, A_{\nu}))$ is surjective.

%
%
	For the second statement, let $r$ be an integer and 
	$\mu \in \widehat{\Lambda}$ 
	 such that 
	$ \mu \not = 1$.
	In the case  
	$r \leq -2$, 
	the desired result is a consequence of Proposition \ref{Decomposition-prop1} (1). If 
	$r \geq 1$, 
	then the first statement (\ref{Decomposition-prop2-1}) and Proposition \ref{Decomposition-prop1} (1) imply that there is an isomorphism
		\begin{align*}
				\widehat{\mathrm{HH}}_{(\mu)}^{r}(A) = \Hh_{(\mu)}^{r}(A, A) \cong D( \Hh_{r}^{(\mu^{-1})}(A, A_{\nu}) ) = 0.
		\end{align*}
	We also have 
	$\widehat{\mathrm{HH}}_{(\mu)}^{-1}(A) = 0$ and $\widehat{\mathrm{HH}}_{(\mu)}^{0}(A) = 0$
	because 
	$\widehat{\mathrm{HH}}_{(\mu)}^{-1}(A) \leq \Hh_{0}^{(\mu)}(A, A_{\nu^{-1}})$, and 
	$\widehat{\mathrm{HH}}_{(\mu)}^{0}(A)$ is a quotient space of $\Hh_{(\mu)}^{0}(A, A)$.

%
%
	For the last statement, let $r$ be an integer.
	For the case 
	$r \leq -2$, 
	the desired result is a consequence of Proposition \ref{Decomposition-prop1} (3). If 
	$r \geq 1$, 
	then the first statement (\ref{Decomposition-prop2-1}) and Proposition \ref{Decomposition-prop1} (1) yield that  there are isomorphisms
		\begin{align*}
				\widehat{\mathrm{HH}}^{r}(A) = \HH^{r}(A) 
				\cong D(\Hh_{r}(A, A_{\nu}))
				\cong D(\Hh_{r}^{(1)}(A, A_{\nu}))
				\cong \Hh_{(1)}^{r}(A, A) 
				= \widehat{\mathrm{HH}}_{(1)}^{r}(A).
		\end{align*}
	Since 
	$A_{\nu^{-1}} = \bigoplus_{i} A_{\lambda_{i}}$ 
	as $k$-vector spaces,
	the differential
	$\widehat{d}^{-1}$
	can be decomposed as 
	$\widehat{d}^{-1} = [ \widehat{d}_{\lambda_{1}}^{-1} \cdots \widehat{d}_{\lambda_{t}}^{-1} ]$,
	where
	$\widehat{d}_{\lambda_{j}}^{-1} : A_{\lambda_{j}} \rightarrow A$ 
	is the restriction of 
	$\widehat{d}^{-1}$ to $A_{\lambda_{j}}$.
	Then we have
		\begin{align*}
				\widehat{\mathrm{HH}}^{-1}(A) \cong \bigoplus_{1 \leq i \leq t} \widehat{\mathrm{HH}}_{(\lambda_{i})}^{-1}(A) = \widehat{\mathrm{HH}}_{(1)}^{-1}(A).
		\end{align*} 
	Similarly, we have 
	$\widehat{\mathrm{HH}}^{0}(A) \cong \widehat{\mathrm{HH}}_{(1)}^{0}(A)$. 
	This completes the proof. 
\end{proof}
%
%
%
	\begin{lem}[{\cite[Lemma 3.5]{Lam}}] \label{Decomposition-lem1}
		Let $A$ be a finite dimensional Frobenius $k$-algebra with the Nakayama automorphism $\nu$ diagonalizable. 
		Then we have the following statements.
	\begin{enumerate}
			\item For any 
					$\lambda \in \Lambda$, its inverse $\lambda^{-1}$ belongs to $\Lambda$.
			\item The isomorphism of $A$-bimodules $A_{\nu} \cong D(A)$ induces an isomorphism of vector spaces 
					$A_{\lambda} \cong D(A_{\lambda^{-1}})$ for any $\lambda \in \Lambda$.
	\end{enumerate} 
\end{lem}

	Suppose that the Nakayama automorphism $\nu$ is diagonalizable. 
	For each 
	$\lambda_{i} \in 
	\Lambda = \{ \lambda_{1}, \ldots, \lambda_{t} \}$, we denote by $m_{i}$ its algebraic multiplicity.
	Then we have a $k$-basis
	$( u_{1}^{\lambda_{i}}, \ldots, u_{m_{i}}^{\lambda_{i}} )$
	of the eigenspace
	$A_{\lambda_{i}}$ 
	associated with 
	$\lambda_{i}$.
	Thus $d$-tuple
	$( u_{1}^{\lambda_{1}}, \ldots, u_{m_{1}}^{\lambda_{1}}, \ldots, u_{1}^{\lambda_{t}}, \ldots, u_{m_{t}}^{\lambda_{t}} )$
	forms a $k$-basis of $A$, and we obtain its dual basis
	$( v_{1}^{\lambda_{1}}, \ldots, v_{m_{1}}^{\lambda_{1}}, \ldots, v_{1}^{\lambda_{t}}, \ldots, v_{m_{t}}^{\lambda_{t}} )$
	of $A$ with respect to the bilinear form 	
	$\langle \ , \ \rangle$. 
	It follows from Lemma \ref{Decomposition-lem1} and 
	$\langle v_{k}^{\lambda_{i}}, u_{l}^{\lambda_{j}} \rangle = \delta_{i j} \delta_{k l}$
	that the dual basis vectors 
	$v_{1}^{\lambda_{i}}, \ldots, v_{m_{i}}^{\lambda_{i}}$
	belong to
	$A_{\lambda_{i}^{-1}}$
	for each 
	$\lambda_{i}$.
	We fix the dual bases 
	$ (u_{j}^{\lambda_{i}} )_{i,\, j}, ( v_{j}^{\lambda_{i}} )_{i,\, j} $
	of $A$.
	For simplifying the notation, we will write 
	$( u_{1}, \ldots ,u_{d} )$ and $( v_{1}, \ldots ,v_{d} )$
	for 
	$ ( u_{j}^{\lambda_{i}} )_{i,\, j} $ and $( v_{j}^{\lambda_{i}} )_{i,\, j} $ when there is no danger of confusion. 
%
%
%
%
	\begin{prop} \label{Decomposition-prop3}
		Let $A$ be a finite dimensional Frobenius $k$-algebra with the Nakayama automorphism $\nu$ diagonalizable. 
		For any 
		$\mu, \mu^{\prime} \in \widehat{\Lambda}$, 
		$\star : \mathcal{D}^{*}(A, A) \otimes \mathcal{D}^{*}(A, A) 
		\rightarrow
		\mathcal{D}^{*}(A, A)$
		induces the restrictions
		$\star_{\mu, \mu^{\prime}} : \mathcal{D}_{\mu}^{*}(A, A) \otimes \mathcal{D}_{\mu^{\prime}}^{*}(A, A) 
		\rightarrow
		\mathcal{D}_{\mu \mu^{\prime}}^{*}(A, A)$.
	\end{prop}
%
%
%
%
	\begin{proof}
		We only show that the $\star$-product $\star$ restricts to the subcomplexes in the cases (3) (i). 
		Proofs of the other cases are similar to the proof of the case (3) (i).
		Let 
		$\mu, \mu^{\prime} \in \widehat{\Lambda}$ be arbitrary and $m, p \in \mathbb{Z}$ such that $m \geq 0, p \geq 0$ and $ p > m$,
		and let
		$f \in C_{(\mu)}^{m}(A, A)$ 
		and 
		$\alpha = a_{0} \otimes \ba_{1,\,p} 
		\in 
		A_{\mu_{0}^{\prime}} \otimes \bA_{\mu_{1}^{\prime}} \otimes \cdots \otimes \bA_{\mu_{p}^{\prime}} 
		\subset
		C_{p}^{(\mu^{\prime})}(A, A_{\nu^{-1}})$ 
		with $\prod \mu_{i}^{\prime} = \mu^{\prime}$.
		We claim that 
			\begin{align*}\lefteqn{}
					(f \star \alpha) (\bb_{1,\,m-p-1})
					= \sum_{\substack{1 \leq i \leq t, \\ 1 \leq j \leq m_{i}}} 
					f( \bb_{1, \,m-p-1} \otimes \overline{u_{j}^{\lambda_{i}} \nu(a_{0})} \otimes \ba_{1, \,p}) v_{j}^{\lambda_{i}} 
					\in A_{\mu \mu^{\prime} \mu_{1} \cdots \mu_{m-p-1} }
			\end{align*}
		holds for any 
		$\bb_{1,\, m-p-1} \in \bA_{\mu_{1}} \otimes \cdots \otimes \bA_{\mu_{m-p-1}}$,
		where the $\mu_{i} $ are elements of $\Lambda$.
		Indeed, we have
			\begin{align*} \lefteqn{}
					&\nu ( \sum_{i, j} 
					f( \bb_{1, \,m-p-1} \otimes \overline{u_{j}^{\lambda_{i}} \nu(a_{0})} \otimes \ba_{1, \,p}) v_{j}^{\lambda_{i}} )\\
					&=
					\sum_{i, j} 
					\nu ( f( \bb_{1, \,m-p-1} \otimes \overline{u_{j}^{\lambda_{i}} \nu(a_{0})} \otimes \ba_{1, \,p})) \nu ( v_{j}^{\lambda_{i}} )\\
					&=
					\sum_{i, j}
					\mu \mu_{1} \cdots \mu_{m-p-1} \lambda_{i} \mu^{\prime} 
					f( \bb_{1, \,m-p-1} \otimes \overline{u_{j}^{\lambda_{i}} \nu(a_{0})} \otimes \ba_{1,\, p}) \lambda_{i}^{-1} v_{j}^{\lambda_{i}} \\
					&=
					\mu \mu^{\prime} \mu_{1} \cdots \mu_{m-p-1} 
					\sum_{i, j}
					f( \bb_{1,\, m-p-1} \otimes \overline{u_{j}^{\lambda_{i}} \nu(a_{0})} \otimes \ba_{1,\, p}) v_{j}^{\lambda_{i}}
			\end{align*}
		and therefore $f \star \alpha \in C_{(\mu \mu^{\prime})}^{m-p-1}(A, A)$.
\end{proof}
%
%
%
%
Put 
$\star_{1} := \star_{1, 1} : 
\mathcal{D}_{(1)}^{\bullet}(A, A) \otimes \mathcal{D}_{(1)}^{\bullet}(A, A) 
\rightarrow
\mathcal{D}_{(1)}^{\bullet}(A, A)$.
Then we have the following result.
%
%
%
%
	\begin{cor} \label{Decomposition-cor1}
		Let $A$ be a finite dimensional Frobenius $k$-algebra. 
		Then
		$(\widehat{\mathrm{HH}}_{(1)}^{\bullet}(A), \star_{1})$ 
		is a graded commutative algebra.
		Furthermore, if the Nakayama automorphism $\nu$ of $A$ is diagonalizable, then
		$(\widehat{\mathrm{HH}}_{(1)}^{\bullet}(A), \star_{1})$ 
		is isomorphic to 
		$(\widehat{\mathrm{HH}}^{\bullet}(A), \star)$
		as graded algebras. 
	\end{cor}

\section{BV structure on the complete cohomology} \label{BV}
%
%
%
%

	Let us recall the definition of Batalin-Vilkovisky algebras.
		\begin{defi} \label{def_1}
			A graded commutative algebra 
			$(\mathcal{H}^{\bullet} = \bigoplus_{r \in \mathbb{Z}} \mathcal{H}^{r}, \smile) $ 
			with $1 \in \mathcal{H}^{0}$ is a {\it Batalin-Vilkovisky algebra} (BV algebra, for short)
			if there exists an operator 
			$\Delta_{*}  : \mathcal{H}^{*} \rightarrow \mathcal{H}^{*-1}$ of degree $-1$ such that:
				\begin{enumerate} 
				\renewcommand{\labelenumi}{(\roman{enumi})}
						\item $\Delta_{r-1} \Delta_{r} = 0$ for any $r \in \mathbb{Z}$;
						\item $\Delta_{0}(1) = 0$;
						\item For homogeneous elements $\alpha, \beta$ and $\gamma$ in $\mathcal{H}^{\bullet}$,
						\begin{align*} \lefteqn{}
									\Delta(\alpha \smile \beta \smile \gamma ) 
									=&\ 
									\Delta(\alpha \smile \beta) \smile \gamma
									+(-1)^{|\alpha|} \alpha \smile \Delta( \beta \smile \gamma )\\
									&+(-1)^{|\beta| (|\alpha|-1)} \beta \smile \Delta( \alpha \smile \gamma ) 
									-\Delta( \alpha ) \smile \beta \smile \gamma\\
									&- (-1)^{|\alpha|} \alpha \smile \Delta( \beta ) \smile \gamma 
									- (-1)^{|\alpha| + |\beta| } \alpha \smile \beta \smile \Delta( \gamma ),
							\end{align*}
						where  		
						$|\alpha|$ denotes the degree of a homogeneous element $\alpha \in \mathcal{H}^{\bullet}$. 
				\end{enumerate}
	\end{defi}
%
%
%
%
	\begin{rem} \label{BV-rem1}  {\rm
		For each BV algebra $(\mathcal{H}^{\bullet}, \smile, \Delta)$, we can associate a graded Lie bracket $[\ , \ ]$ of degree $-1$ 
		as 
			\begin{align*}
				[\alpha, \beta] := (-1)^{|\alpha| |\beta| +|\alpha|+|\beta|} \left (
										(-1)^{ |\alpha| +1 } \Delta( \alpha \smile \beta ) 
										+(-1)^{ |\alpha| } \Delta( \alpha ) \smile \beta
										+ \alpha \smile \Delta( \beta ) \right ),
			\end{align*} 
		where $\alpha, \beta$ are homogeneous elements of $\mathcal{H}^{\bullet}$.
		The equation is said to be the BV identity. 
		It follows from \cite[Proposition 1.2]{Getzler94} that the bracket $[\ , \ ]$ above makes 
		$(\mathcal{H}^{\bullet}, \smile, [\ , \ ])$
		into a Gerstenhaber algebra. 
		}
	\end{rem}
%
%
%
%

	Recall that a {\it symmetric algebra} $A$ is a Frobenius algebra with a non-degenerate bilinear form 
	$\langle\ , \ \rangle : A \otimes A \rightarrow k$
	satisfying 
	$\langle a, b \rangle = \langle b, a \rangle$ for all $a, b \in A$.
	Wang gave the following result for symmetric algebras.
%
%
	\begin{theo}[{\cite[Corollary 6.21]{wang}}] \label{BV-theo1} 
		Let $A$ be a finite dimensional symmetric $k$-algebra.
		Then the complete cohomology ring $(\widehat{\mathrm{HH}}^{\bullet}(A), \star)$ is a BV algebra together with an operator 
		$\widehat{\Delta}_{*} : \widehat{\mathrm{HH}}^{*}(A) \rightarrow \widehat{\mathrm{HH}}^{*-1}(A)$ defined by
			\begin{align*}
				\widehat{\Delta}_{r} = \begin{cases}
												\Delta_{r} & \mbox{if} \ r \geq 1, \\
												0 & \mbox{if} \ r = 0, \\
												(-1)^{r} \ B_{-r-1} & \mbox{if} \ r \leq -1,
											\end{cases}
			\end{align*}
		where 
		$B_{r}$ is the Connes operator defined by
			\begin{align*}
				B_{r}( a_{0} \otimes \ba_{1,\, r} ) 
				= \sum_{i = 1}^{r+1} (-1)^{ i r} 1 \otimes \ba_{i,\,r} \otimes \ba_{0} \otimes \ba_{1,\,i-1}
			\end{align*}
		for any $a_{0} \otimes \ba_{1,\,r} \in C_{r}(A, A_{\nu^{-1}})$, and 
		$\Delta_{r}$ defined in \cite{Tradler} is the dual of the Connes operator $B_{r-1}$, which is equivalent to saying that $\Delta_{r}$ is given by a formula
			\begin{align*}
				\langle \Delta_{r}(f)(\ba_{1,\, r-1}), a_{r} \rangle 
				= \sum_{i = 1}^{r} (-1)^{ i(r-1)} \langle f(\ba_{i,\, r-1} \otimes \ba_{r} \otimes \ba_{1, \,i-1}), 1 \rangle
			\end{align*} 
		for any $f \in C^{r}(A, A)$.
		In particular, the restrictions $\widehat{\mathrm{HH}}^{\geq 0}(A)$ and $\widehat{\mathrm{HH}}^{\leq 0}(A)$  are BV subalgebras of $\widehat{\mathrm{HH}}^{\bullet}(A)$.
	\end{theo}
%
%
%
%
	\begin{rem} \label{BV-rem2}  {\rm
Let $A$ be a finite dimensional symmetric $k$-algebra.
		It follows from Remark \ref{BV-rem1} that the BV differential $\widehat{\Delta}$ in Theorem \ref{BV-theo1} gives rise to a Lie bracket $\{\ ,\ \}$ (of degree $-1$)
		defined by 
			\[
				\{ \alpha , \beta \} 
				:=  (-1)^{|\alpha| |\beta| +|\alpha|+|\beta|}  \left (
					(-1)^{|\alpha| +1} \widehat{\Delta}( \alpha \smile \beta ) 
					+(-1)^{ |\alpha| } \widehat{\Delta}( \alpha ) \smile \beta
					+ \alpha \smile \widehat{\Delta}(\beta) \right )
			\]
		for any homogeneous elements $\alpha, \beta \in \widehat{\mathrm{HH}}^{\bullet}(A)$. 
		Moreover, the Gerstenhaber algebra $(\widehat{\mathrm{HH}}^{\bullet}(A), $ $\star,$ $\{\ ,\ \})$ is isomorphic to 
		$(\underline{\mathrm{Ext}}_{A^{\rm e}}^{\bullet}(A, A),$ $\smile_{{\rm sg}},$ $[\ ,\ ]_{{\rm sg}})$ 
		as Gerstenhaber algebras. 
		}
	\end{rem}
%
%
%
%
	In the rest of this section, we will show the following result on Frobenius algebras whose Nakayama automorphisms are diagonalizable.

\begin{theo} \label{BV-theo2}
Let $A$ be a finite dimensional Frobenius $k$-algebra. 	If the Nakayama automorphism $\nu$ is diagonalizable, 
then the graded commutative ring 
$(\widehat{\mathrm{HH}}_{(1)}^{\bullet}(A), \star_{1} )$
is a BV algebra together with an operator 
$\widehat{\Delta}_{*} : \widehat{\mathrm{HH}}_{(1)}^{*}(A) \rightarrow \widehat{\mathrm{HH}}_{(1)}^{*-1}(A)$ defined by
\begin{align*}
\widehat{\Delta}_{r} 
= 
\begin{cases}
\Delta^{\nu}_{r} & \mbox{if} \ r \geq 1, \\
0 & \mbox{if} \ r = 0, \\
(-1)^{i} \ B_{-r-1}^{\nu^{-1}} & \mbox{if} \ r \leq -1,
\end{cases}
\end{align*}
where $B_{r}^{\nu^{-1}}$ is the twisted Connes operator defined by
\begin{align*}
B_{r}^{\nu^{-1}}( a_{0} \otimes \ba_{1,\,r} ) 
= \sum_{i = 1}^{r+1} 
(-1)^{ i r} 
1 \otimes \ba_{i,\,r} \otimes \ba_{0} \otimes \overline{\nu^{-1} (a_{1})} \otimes \cdots \otimes \overline{\nu^{-1} (a_{i-1})}
\end{align*} 
for any $a_{0} \otimes \ba_{1,\,r} \in C_{r}(A, A_{\nu^{-1}})$, and $\Delta^{\nu}_{r}$ defined in \cite{Lam} is the dual of the twisted 
Connes operator $B_{r-1}^{\nu}$, which is equivalent to saying that $\Delta^{\nu}_{r}$ is given by a formula
\begin{align*}
\langle \Delta_{r}^{\nu}(f)(\ba_{1,\,r-1}), a_{r} \rangle 
= 
\sum_{i = 1}^{r} (-1)^{ i(r-1)} 
\langle f(\ba_{i,\,r-1} \otimes \ba_{r} \otimes \overline{\nu (a_{1})} \otimes \cdots \otimes \overline{\nu (a_{i-1})} ), 1 \rangle
\end{align*} 
for any $f \in C^{r}(A, A)$.
In particular, the restrictions $\widehat{\mathrm{HH}}^{\geq 0}(A)$ and $\widehat{\mathrm{HH}}^{\leq 0}(A)$  are BV subalgebras of $\widehat{\mathrm{HH}}^{\bullet}(A)$.
\end{theo}
	\begin{rem}  {\rm
		Each of the components $\Delta_{*}^{\nu}$ and $B_{*}^{\nu^{-1}}$ is defined on the chain level.
		Corollary \ref{Decomposition-cor2} and Lemma \ref{Decomposition-lem1} imply that we can lift 
		the two components $\Delta_{*}^{\nu}$ and $B_{*}^{\nu^{-1}}$ to the cohomology level when we restrict them to $\mathcal{D}_{(1)}^{*}(A, A)$. 
		}
	\end{rem}
%
%
	Using the isomorphism $\widehat{\mathrm{HH}}_{(1)}^{\bullet}(A) \cong \widehat{\mathrm{HH}}^{\bullet}(A)$ 
	appeared in Corollary \ref{Decomposition-cor1}, 
	we have our main result.

		\begin{cor} \label{BV-cor2} 
			Let $A$ be a finite dimensional Frobenius $k$-algebra. If the Nakayama automorphism of $A$ is diagonalizable, then
			the complete cohomology ring $\widehat{\mathrm{HH}}^{\bullet}(A)$ is a BV algebra.
		\end{cor}

In order to prove Theorem \ref{BV-theo2}, we claim that the bilinear map 
		\[
			\{ \ , \ \} : \widehat{\mathrm{HH}}_{(1)}^{m}(A) \otimes \widehat{\mathrm{HH}}_{(1)}^{n}(A) 
			\rightarrow \widehat{\mathrm{HH}}_{(1)}^{m+n-1}(A)\ (m, n \in \mathbb{Z})
		\]
defined by 
			\[
				\{ \alpha , \beta \} 
				:=  (-1)^{|\alpha| |\beta| +|\alpha|+|\beta|}  \left (
					(-1)^{|\alpha| +1} \widehat{\Delta}( \alpha \star_{1} \beta )
					+(-1)^{ |\alpha| } \widehat{\Delta}( \alpha ) \star_{1} \beta  
					+ \alpha \star_{1} \widehat{\Delta}(\beta) \right )
			\] 
for any  $\alpha \otimes \beta \in \widehat{\mathrm{HH}}_{(1)}^{m}(A) \otimes \widehat{\mathrm{HH}}_{(1)}^{n}(A)$
commutes with the Gerstenhaber bracket 
		\[
		    [\ ,\ ]_{\rm sg} : \underline{\mathrm{Ext}}_{A^{\rm e}}^{m}(A, A) \otimes \underline{\mathrm{Ext}}_{A^{\rm e}}^{n}(A, A) 
			\rightarrow \underline{\mathrm{Ext}}_{A^{\rm e}}^{m+n-1}(A, A).
		\]

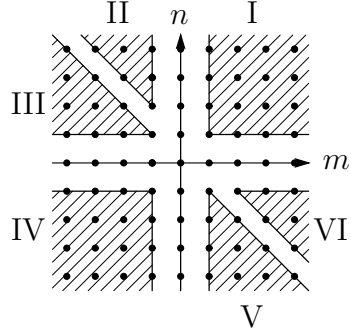
\begin{figure}[h] 
 \centering
\input{Figure1_BV.tex}
    \caption{A plane with six regions}
    \label{fig_1} 
\end{figure} 
By considering whether $m+n-1$ is negative or not together with Figure \ref{fig_1} and using the anti-commutativity of the Gerstenhaber bracket $[\ ,\ ]_{\rm sg}$, we see that it suffices to show our claim for a pair $(m, n)$ of integers $m \leq n$ satisfying one of the following conditions:
\begin{enumerate}
    \item $(m, n)$ is on the lines $m=0$ or $n=0$.
    \item $(m, n)$ belongs to the regions I, IV, V or VI.
\end{enumerate}
Thus our claim can be divided into the five cases Propositions \ref{prop-6.19}, \ref{prop6.17}, \ref{prop-6.20} and \ref{prop-6.21}  and  Remark \ref{remark_1}.
In particular, Propositions \ref{prop-6.19}, \ref{prop6.17}, \ref{prop-6.20} and \ref{prop-6.21} prove our claim for the pairs in the regions VI, V, IV and I, respectively.
Further, we consider the  case (1) in Remark \ref{remark_1}.
Among the four propositions, we prove only the first one.
We also remark that, in the following propositions, the appearing integers $m$ and $n$ are independent of the above argument.

%
%
%
%
	\begin{prop} \label{prop-6.19} 
		Let $A$ be a finite dimensional Frobenius $k$-algebra with the Nakayama automorphism $\nu$ of $A$ diagonalizable, 
		and let $m, n$ be integers such that $m > n \geq 1$, so $m-n-1 \geq 0$.
		Then we have the following commutative diagram. 
	\[
		\xymatrix@C=40pt{
			\widehat{\mathrm{HH}}_{(1)}^{m}(A)  \otimes \widehat{\mathrm{HH}}_{(1)}^{-n}(A)  \ar[r]^-{\{ \ ,\ \}}  \ar[d]^{\cong}  
			&
			\widehat{\mathrm{HH}}_{(1)}^{m-n-1}(A) \ar[d]_{\cong} \\ 
			\mathrm{Ext}_{A^{\rm e}}^{m}(A, A) \otimes \mathrm{Tor}_{n-1}^{A^{\rm e}}(A, A_{\nu^{-1}})  \ar[d]_{\mathrm{id} \otimes \kappa_{n-1, 1} }^{\cong}
			& 
			\mathrm{Ext}_{A^{\rm e}}^{m-n-1}(A, A) \ar[d]^{\varphi_{m-n-1, 0}^{n+1}}_{\cong} \\
			\mathrm{Ext}_{A^{\rm e}}^{m}(A, A) \otimes \mathrm{Ext}_{A^{\rm e}}^{1}(A, \overline{\Omega}^{n+1} (A)) \ar[r]^-{ [\ ,\ ]_{\rm{sg}}  } \ar[d]^{\cong}
			&
			\mathrm{Ext}_{A^{\rm e}}^{m}(A, \overline{\Omega}^{n+1} (A)) \ar[d]_{\cong} \\
			\underline{\mathrm{Ext}}_{A^{\rm e}}^{m}(A, A) \otimes \underline{\mathrm{Ext}}_{A^{\rm e}}^{-n}(A, A) \ar[r]^-{ [\ ,\ ]_{\rm{sg}}  }   & 
			\underline{\mathrm{Ext}}_{A^{\rm e}}^{m-n-1}(A, A),
			}\]
	where $\{ \ , \ \} : \mathcal{D}_{(1)}^{m}(A, A) \otimes \mathcal{D}_{(1)}^{-n}(A, A) \rightarrow \mathcal{D}_{(1)}^{m-n-1}(A, A)$ is defined by
		\begin{align} 
			\{ f, z \} 
			&= (-1)^{ |f| |z| +|f| +|z| } \left (
				(-1)^{|f|+1} \Delta^{\nu}(f \star_{1} z) 
				+ (-1)^{|f|} \Delta^{\nu}(f) \star_{1} z 
				+ (-1)^{|z|} f  \star_{1} B^{\nu^{-1}}(z) \right ) \nonumber \\[3pt]
			&= (-1)^{ |f| |z| +|f| +|z| } \left (
				(-1)^{|f|+1} \widehat{\Delta}(f \star_{1} z) 
				+ (-1)^{|f|} \widehat{\Delta}(f) \star_{1} z
				+  f \star_{1} \widehat{\Delta}(z) \right ) \nonumber 
		\end{align}
 	for $f \otimes z \in \mathcal{D}_{(1)}^{m}(A, A)  \otimes \mathcal{D}_{(1)}^{-n}(A, A)$.
 \end{prop}
%
%
\begin{proof}
It follows from the definition of the induced Gerstenhaber bracket on $\underline{\mathrm{Ext}}_{A^{\rm e}}^{\bullet}(A, A)$ that the bottom square is commutative (see Theorem \ref{thm_1}).
	It remains to show the commutativity of the top diagram. 
	It suffices to prove that a formula
		\begin{align} \label{prop-6.19-eq1}
			([\ ,\ ]_{\rm sg} (\mathrm{id} \otimes \kappa_{n-1,\,1}))(f \otimes z) = \varphi_{m-n-1,\,0}^{n+1}( \{ \ ,\ \}(f \otimes z))			
		\end{align}	
	holds in $\mathrm{Ext}_{A^{\rm e}}^{m}(A, \overline{\Omega}^{n+1}(A))$ 
	for $f \in \mathrm{Ker}\, \delta_{(1)}^{m}$ and $z := a_{0} \otimes \ba_{1,\,n-1}  \in \mathrm{Ker}\, \partial_{n-1}^{(1)}$.
	Denote by $\overline{f}$ the composition of $f : \overline{A}^{\otimes m} \rightarrow A$ with the canonical epimorphism 
	$\pi : A \rightarrow \overline{A}$. 
	For the right hand side of  the formula (\ref{prop-6.19-eq1}), we have, for $\overline{b}_{1,\,m} \in \overline{A}^{\otimes m}$,
		\begin{eqnarray} \label{eq113}
				\lefteqn{
				\varphi_{m-n-1,\,0}^{n+1} ( \{ f, z \})(\bb_{1,\,m}) } \nonumber \\
			&=&
				(-1)^{mn+n+1} \varphi_{m-n-1,\,0}^{n+1} (  \Delta^{\nu}(f \frown z) ) (\bb_{1,\,m})         
				+(-1)^{mn+n}  \varphi_{m-n-1,\,0}^{n+1}  (   \Delta^{\nu}(f) \frown z ) (\bb_{1,\,m})		  
				   \nonumber  \\[2pt] 
			&& 
				+(-1)^{mn+m} \varphi_{m-n-1,\,0}^{n+1} (   f \frown \mathrm{B}^{\nu^{-1}}(z)) (\bb_{1,\,m})  \nonumber \\ 
			&=&
				\sum_{j,\,k} \sum_{i=1}^{n-m} (-1)^{i(n-m+1)+n+1} 
				d( \langle u_{k} \nu a_{0} f( \ba_{1,\,n-1} \otimes \bv_{k} \otimes \bb_{i,\,m-n-1} \otimes \bu_{j} \otimes \overline{\nu b}_{1,\,i-1}), 1 \rangle 
				v_{j}   \nonumber \\ 
			&& \otimes \ \bb_{m-n,\,m} \otimes 1) \nonumber \\  
			&&
				+\sum_{j,\,k} \sum_{i=1}^{n} (-1)^{i(m+1)} 
				d( u_{j} \nu a_{0} \langle f( \ba_{i,\,n-1} \otimes \bv_{j} \otimes \bb_{1,\,m-n-1} \otimes \bu_{k} \otimes \nuina_{1,\,i-1}), 1 \rangle v_{k} 
				\nonumber \\ 
			&& \otimes \  \bb_{m-n,\,m} \otimes 1)  \nonumber \\  
			&&
				+\sum_{j,\,k} \sum_{i=1}^{m-n} (-1)^{(i+n)(m+1)} 
				d( u_{j} \nu a_{0} \langle f( \bb_{i,\,m-n-1} \otimes \bu_{k} \otimes \nuina_{1,\,n-1} \otimes \overline{ \nu v}_{j} \otimes 
				\overline{ \nu b}_{1,\,i-1}), 1 \rangle v_{k}   \nonumber \\ 
			&&  \otimes \ \bb_{m-n,\,m} \otimes 1) 
				 \nonumber \\  
			&&
				+\sum_{j} \sum_{i=1}^{n} (-1)^{(m+i)(n+1)}
   				d( f(\bb_{1,\,m-n-1} \otimes \bu_{j} \otimes \ba_{i,\,n-1} \otimes \ba_{0} \otimes \overline{ \nu^{-1} a}_{1,\,i-1}) v_{j} \otimes \bb_{m-n,\,m} 
				\otimes 1)
				 \nonumber   
		\end{eqnarray}
	in $ \overline{\Omega}^{n+1}(A)$. 
	On the other hand, for the left hand side of the formula (\ref{prop-6.19-eq1}), we get
		\begin{eqnarray} \label{eq112}
				\lefteqn{ 
				[f, \kappa_{n-1,\,1} (z)]_{\rm sg}(\bb_{1,\,m}) } \nonumber \\
			&=&
				(f \bullet \kappa_{n-1,\,1} (z) -(-1)^{(m-1)(-n-1)} \kappa_{n-1,\,1} (z) \bullet f ) (\bb_{1,\,m}) \nonumber \\
			&=& 
				\sum_{j} \sum_{i=1}^{m-n} (-1)^{(n+1)(i+1)} 
				d( f( \bb_{1,\,i-1} \otimes \overline{u_{j} \nu a_{0}} \otimes \ba_{1,\,n-1} \otimes \bv_{j} \otimes \bb_{i,\,m-n-1}) \otimes \bb_{m-n,\,m} 
				\otimes 1)
				 \label{6.19-5-1-m-n}  \\ 
			&&
				+\sum_{j} \sum_{i=m-n+1}^{m} (-1)^{(n+1)(i+1)} 
				d( f( \bb_{1,\,i-1} \otimes \overline{u_{j} \nu a_{0}}  \otimes \ba_{1,\,m-i}) \otimes \ba_{m-i+1,\,n-1} \otimes \bv_{j} \otimes \bb_{i,\,m} 
				\otimes 1) 
				\label{6.19-5-m-n+1-m} \nonumber  \\  
			&&
				+\sum_{j} \sum_{i=1}^{n}(-1)^{i(m+1)}
				d( u_{j} \nu a_{0} \otimes \ba_{1,\,i-1} \otimes \overline{f}(\ba_{i,\,n-1} \otimes \bv_{j} \otimes \bb_{1,\,m-n+i-1}) \otimes \bb_{m-n+i,\,m} 
				\otimes 1). \nonumber  \\
				\label{6.19-4}  
		\end{eqnarray}
	 We will transform $[f, \kappa_{n-1,\,1} (z)](\bb_{1,\,m})$ to $\varphi_{m-n-1,\,0}^{n+1} ( \{ f, z \})(\bb_{1,\,m})$ 
	in $\overline{\Omega}^{n+1}(A)$, using some boundaries. 
   	
	First, we deform $(\ref{6.19-5-1-m-n})$.
	A direct calculation shows 
	\begin{eqnarray*} 
			\lefteqn{ 
			\sum_{j,\,k} \sum_{i=1}^{n-m} (-1)^{i(n-m+1)+n+1} 
			d( \langle u_{k} \nu a_{0} f( \ba_{1,\,n-1} \otimes \bv_{k} \otimes \bb_{i,\,m-n-1} \otimes \bu_{j} \otimes \overline{\nu b}_{1,\,i-1}), 1 \rangle v_{j} 
			\otimes \bb_{m-n, m} } \nonumber \\
		&& \otimes \ 1)  \\ 
		&& \hspace{-30pt} 
			+\sum_{j,\,k} \sum_{i=1}^{m-n} (-1)^{(i+n)(m+1)} 
			d( u_{j} \nu a_{0} \langle f( \bb_{i,\,m-n-1} \otimes \bu_{k} \otimes \nuina_{1,\,n-1} \otimes \overline{ \nu v}_{j} \otimes
			\overline{ \nu b}_{1,\,i-1}), 1 \rangle v_{k} \otimes  \bb_{m-n,\,m} \nonumber \\
		&&  \otimes \ 1)  \\   
		&=&
			\sum_{j} \sum_{i=1}^{m-n-2} \sum_{l=i+2}^{m-n} (-1)^{ i(m+1)+(n+1)l+1} 
			\varphi_{m-n-1,\,0}^{n-1} \bigg ( 
			\delta (( \sum_{j,\,k} \,
			\langle f( \mathrm{id}_{\bA}^{\otimes l-i-1} \otimes \overline{u_{j} \nu a_{0}}  \otimes \ba_{1,\,n-1} \otimes \bv_{j} 
			 \nonumber \\
		&&
			 \otimes \  \mathrm{id}_{\bA}^{\otimes m-n-l} \otimes \bu_{k} \otimes \overline{\nu}^{\otimes i-1}), 1 \rangle v_{k}) \circ t^{i-1}) \bigg ) 
			(\bb_{1,\,m})  
			\nonumber \\ 
		&&
			+\sum_{j,\,k} \sum_{i=1}^{m-n-1} (-1)^{ i(m+1)+(n+1)(i+1)+1} 
			\varphi_{m-n-1,\,0}^{n-1}  \bigg ( 
			\delta ((\sum_{j,\,k} \,
			\langle f( \overline{u_{j} \nu a_{0}}  \otimes \ba_{1,\,n-1} \otimes \bv_{j}     \nonumber \\
		&& \otimes \  \mathrm{id}_{\bA}^{\otimes m-n-i-1} \otimes \bu_{k} \otimes \overline{\nu}^{\otimes i-1}), 1 \rangle v_{k}) \circ t^{i-1} ) \bigg ) 
			(\bb_{1,\,m})
			 \nonumber \\
		&&
			+\sum_{j} \sum_{i=1}^{m-n} (-1)^{(n+1)(i+1)} 
			d( f( \bb_{1,\,i-1} \otimes \overline{u_{j} \nu a_{0}}  \otimes \ba_{1,\,n-1} \otimes \bv_{j} \otimes \bb_{i,\,m-n-1}) \otimes \bb_{m-n,\,m} \otimes 1),
	\end{eqnarray*}
	where the $k$-linear map 
		$t : \overline{A}^{\otimes m-n-2} \rightarrow \overline{A}^{\otimes m-n-2}$
	is given by 
			$t(\bb_{1,\,m-n-2}) = \bb_{2,\,m-n-2} \otimes\,\bb_{1}$
	for $\bb_{1,\,m-n-2} \in \overline{A}^{\otimes m-n-2}$. In particular, we have $t^{i-1}(\bb_{1,\,m-n-2}) = \bb_{i,\,m-n-2} \otimes \bb_{1,\,i-1}$.
	Note that the two maps 
		\[ 
			\varphi_{m-n-1,\,0}^{n-1} \left ( 
			\delta (( \sum_{j,\,k} \,
			\langle f( \overline{u_{j} \nu a_{0}} \otimes \ba_{1,\,n-1} \otimes \bv_{j} \otimes \mathrm{id}_{\bA}^{\otimes m-n-i-1} \otimes \bu_{k} \otimes 
					\overline{\nu}^{\otimes i-1}), 1 \rangle v_{k}) \circ t^{i-1}) \right ),   
		\]		
		\[
			\varphi_{m-n-1,\,0}^{n-1} \left (  
			\delta (( \sum_{j,\,k} \,
			\langle f( \mathrm{id}_{\bA}^{\otimes l-i-1} \otimes \overline{u_{j} \nu a_{0}}  \otimes \ba_{1,\,n-1} \otimes \bv_{j} 
			\otimes \mathrm{id}_{\bA}^{\otimes m-n-l} \otimes \bu_{k} \otimes \overline{\nu}^{\otimes i-1}), 1 \rangle v_{k}) \circ t^{i-1}) \right )  
		\] 
	are zero  in $\mathrm{Ext}_{A^{\rm e}}^{m}(A, \overline{\Omega}^{n+1} (A))$. 
	Hence, we have
		\begin{eqnarray}\label{eq116}
				\lefteqn{ 
				\sum_{j} \sum_{i=1}^{m-n} (-1)^{(n+1)(i+1)} 
				d( f( \bb_{1,\,i-1} \otimes \overline{u_{j} \nu a_{0}}  \otimes \ba_{1,\,n-1} \otimes \bv_{j} \otimes \bb_{i,\,m-n-1}) \otimes \bb_{m-n,\,m} 
				\otimes 1)}
				\nonumber \\ 
			&=&
				\sum_{j} \sum_{i=1}^{m-n-2} \sum_{l=i+2}^{m-n} (-1)^{ i(m+1)+(n+1)l} 
				\varphi_{m-n-1,\,0}^{n-1} \bigg ( 
				\delta (( \sum_{j,\,k} \,
				\langle f( \mathrm{id}_{\bA}^{\otimes l-i-1} \otimes \overline{u_{j} \nu a_{0}}  \otimes \ba_{1,\,n-1} \otimes \bv_{j} 
				 \nonumber \\
			&&
				\otimes \ \mathrm{id}_{\bA}^{\otimes m-n-l} \otimes
				 \bu_{k} \otimes \overline{\nu}^{\otimes i-1}), 1 \rangle v_{k}) 
				\circ
				t^{i-1}) \bigg ) (\bb_{1,\,m})  \nonumber \\ 
			&&
				+\sum_{j,\,k} \sum_{i=1}^{m-n-1} (-1)^{ i(m+1)+(n+1)(i+1)} 
				\varphi_{m-n-1,\,0}^{n-1}  \bigg ( 
				\delta ((\sum_{j,\,k} \,
				\langle f( \overline{u_{j} \nu a_{0}} \otimes \ba_{1,\,n-1} \otimes \bv_{j}    \nonumber \\
			&& \otimes \ \mathrm{id}_{\bA}^{\otimes m-n-i-1} \otimes \bu_{k} \otimes \overline{\nu}^{\otimes i-1}), 1 \rangle v_{k}) \circ t^{i-1}) \bigg )
				 (\bb_{1,\,m})
				 \nonumber \\
			&&
				+\sum_{j,\,k} \sum_{i=1}^{n-m} (-1)^{i(n-m+1)+n+1} 
				d( \langle u_{k} \nu a_{0} f( \ba_{1,\,n-1} \otimes \bv_{k} \otimes \bb_{i,\,m-n-1} \otimes \bu_{j} \otimes \overline{\nu b}_{1,\,i-1}), 1 \rangle v_{j} 
				\nonumber \\  
			&& \otimes \  \bb_{m-n, m} \otimes 1) \nonumber \\  
			&&
				+\sum_{j,\,k} \sum_{i=1}^{m-n} (-1)^{(i+n)(m+1)} 
				d( u_{j} \nu a_{0} \langle f( \bb_{i,\,m-n-1} \otimes \bu_{k} \otimes \nuina_{1,\,n-1} \otimes \overline{ \nu v}_{j} \otimes
				\overline{ \nu b}_{1,\,i-1}), 1 \rangle v_{k}  \nonumber \\ 
			&& \otimes \  \bb_{m-n,\,m} \otimes 1).   
\end{eqnarray}
	Secondly, we deform $(\ref{6.19-5-m-n+1-m})$.
	A direct calculation shows 
	\begin{eqnarray*} 
			\lefteqn{ 
			\sum_{j} \sum_{i=0}^{n-1} \sum_{l=i+1}^{n} (-1)^{ n(m+i+1)+(n+1)(l-i+1)} 
			\delta (d( f( \mathrm{id}_{\bA}^{\otimes m-n+i-1} \otimes \bu_{j} \otimes \ba_{n+i-l+1,\,n-1} \otimes \ba_{0} 
			\otimes \overline{\nu^{-1} a}_{1,\,n-l})}
			\nonumber  \\
		&& \hspace{-30pt}	
			\otimes \  \nuina_{n-l+1,\,n+i-l} \otimes \bv_{j} \otimes \mathrm{id}_{\bA}^{\otimes n-i} \otimes 1)) (\bb_{1,\,m}) 
			\nonumber \\ 
		&=&
			\sum_{j} \sum_{i=m-n+1}^{m} (-1)^{(n+1)(i+1)} 
			d( f( \bb_{1,i-1} \otimes \overline{u_{j} \nu a}_{0} \otimes \ba_{1,\,m-i}) \otimes \ba_{m-i+1,\,n-1} \otimes \bv_{j} \otimes \bb_{i,\,m} \otimes 1) 
			 \nonumber \\  
		&&
			+\sum_{j} \sum_{i=1}^{n} (-1)^{(m+i)(n+1)+1}
			d( f(\bb_{1,\,m-n-1} \otimes \bu_{j} \otimes \ba_{i,\,n-1} \otimes \ba_{0} \otimes \overline{ \nu^{-1} a}_{1,\,i-1}) v_{j} \otimes \bb_{m-n,\,m} 
			\otimes 1)
			 \nonumber \\ 
		&&
			+\sum_{j} \sum_{i=0}^{n-1} (-1)^{m(n+1)+i+1}
			d( f( \bb_{1,\,m-n+i} \otimes \bu_{k} \otimes \nuina_{i+1,\,n-1}) \nu a_{0} \otimes \ba_{1,\,i} \otimes \bv_{k} \otimes \bb_{m-n+i+1,\,m} 
			\nonumber \\  
			&& \otimes \ 1). 
			\nonumber \\
		\end{eqnarray*}
	Hence, we have
		\begin{eqnarray} \label{eq114}
			\lefteqn{ 
			\sum_{j} \sum_{i=m-n+1}^{m} (-1)^{(n+1)(i+1)} 
			d( f( \bb_{1,i-1} \otimes \overline{u_{j} \nu a}_{0} \otimes \ba_{1,\,m-i}) \otimes \ba_{m-i+1,\,n-1} \otimes \bv_{j} \otimes \bb_{i,\,m} \otimes 1) }
			 \nonumber \\  
		&=&
			\sum_{j} \sum_{i=0}^{n-1} \sum_{l=i+1}^{n} (-1)^{ n(m+i+1)+(n+1)(l-i+1)} 
			\delta (d( f( \mathrm{id}_{\bA}^{\otimes m-n+i-1} \otimes \bu_{j} \otimes \ba_{n+i-l+1,\,n-1} \otimes \ba_{0} 
			\nonumber  \\
		&&	\otimes \ \overline{\nu^{-1} a}_{1,\,n-l}) \otimes \nuina_{n-l+1,\,n+i-l} \otimes \bv_{j} \otimes \mathrm{id}_{\bA}^{\otimes n-i} \otimes 1)) 
			(\bb_{1,\,m}) 
			\nonumber \\ 
		&&
			+\sum_{j} \sum_{i=1}^{n} (-1)^{(m+i)(n+1)}
			d( f(\bb_{1,\,m-n-1} \otimes \bu_{j} \otimes \ba_{i,\,n-1} \otimes \ba_{0} \otimes \overline{ \nu^{-1} a}_{1,\,i-1}) v_{j} \otimes \bb_{m-n,\,m} 
			\otimes 1)
			 \nonumber \\ 
		&&
			+\sum_{j} \sum_{i=0}^{n-1} (-1)^{m(n+1)+i}
			d( f( \bb_{1,\,m-n+i} \otimes \bu_{k} \otimes \nuina_{i+1,\,n-1}) \nu a_{0} \otimes \ba_{1,\,i} \otimes \bv_{k} \otimes \bb_{m-n+i+1,\,m} 
			\nonumber  \\
		&&	\otimes \ 1). 
			\nonumber \\[-10pt]
\end{eqnarray}
	Finally, we deform $(\ref{6.19-4})$.
	A direct calculation shows 
		\begin{eqnarray*} 
			\lefteqn{ 
			\sum_{j} \sum_{i=0}^{n-1} \sum_{l=i+1}^{n} (-1)^{ n(m+i+1)+(n+1)(l-i+1)}
			\delta( d( u_{j} \nu a_{0} \otimes \nuina_{1,\,i} \otimes \langle f( \nuina_{n+i-l+1,\,n-1} \otimes \bv_{j} 
			}  \nonumber \\			 			 
		&&  \hspace{-28pt}	
			\otimes \ \mathrm{id}_{\bA}^{\otimes m-n+i-1} \otimes \bu_{k} \otimes \ba_{i+1,\,n-l+i}), 1 \rangle \bv_{k} \otimes 
			\mathrm{id}_{\bA}^{\otimes n-i} \otimes 1))(\bb_{1,\,m}) 
			\nonumber \\ 
		&=&
			\sum_{j} \sum_{i=1}^{n} (-1)^{i(m+1)}
  			d( u_{j} \nu a_{0} \otimes \ba_{1,\,i-1} \otimes \overline{f}(\ba_{i,\,n-1} \otimes \bv_{j} \otimes \bb_{1,\,m-n+i-1}) \otimes \bb_{m-n+i,\,m} 
			\otimes 1)\nonumber \\ \nonumber \\ 
		&&
			+\sum_{j,\,k} \sum_{i=1}^{n} (-1)^{i(m+1)+1} 
			d( u_{j} \nu a_{0} \langle f( \ba_{i,\,n-1} \otimes \bv_{j} \otimes \bb_{1,\,m-n-1} \otimes \bu_{k} \otimes \nuina_{1,\,i-1}), 1 \rangle v_{k} 
			\nonumber  \\
		&& \otimes \ \bb_{m-n,\,m} \otimes 1)  \nonumber \\  
		&&
			+\sum_{j} \sum_{i=0}^{n-1} (-1)^{m(n+1)+i}
			d( f( \bb_{1,\,m-n+i} \otimes \bu_{k} \otimes \nuina_{i+1,\,n-1}) \nu a_{0} \otimes \ba_{1,\,i} \otimes \bv_{k} \otimes \bb_{m-n+i+1,\,m} 
			\nonumber  \\
		&& \otimes \ 1). 
		\end{eqnarray*}
	Thus, we get
		\begin{eqnarray} \label{eq115}
				\lefteqn{ 
  				\sum_{j} \sum_{i=1}^{n} (-1)^{i(m+1)}
  				d( u_{j} \nu a_{0} \otimes \ba_{1,\,i-1} \otimes \overline{f}(\ba_{i,\,n-1} \otimes \bv_{j} \otimes \bb_{1,\,m-n+i-1}) \otimes \bb_{m-n+i,\,m} 
				\otimes 1) }
				\nonumber \\ 
			&=&
				\sum_{j} \sum_{i=0}^{n-1} \sum_{l=i+1}^{n} (-1)^{ n(m+i+1)+(n+1)(l-i+1)}
				\delta( d( u_{j} \nu a_{0} \otimes \nuina_{1,\,i} \otimes \langle f( \nuina_{n+i-l+1,\,n-1} \otimes \bv_{j}   \nonumber \\ 
			&& \otimes \ \mathrm{id}_{\bA}^{\otimes m-n+i-1} \otimes \bu_{k} \otimes \ba_{i+1,\,n-l+i}), 1 \rangle \bv_{k} \otimes \mathrm{id}_{\bA}^{\otimes n-i} \otimes 1))
				(\bb_{1,\,m}) 
				\nonumber \\ 
			&&
				+\sum_{j,\,k} \sum_{i=1}^{n} (-1)^{i(m+1)} 
				d( u_{j} \nu a_{0} \langle f( \ba_{i,\,n-1} \otimes \bv_{j} \otimes \bb_{1,\,m-n-1} \otimes \bu_{k} \otimes \nuina_{1,\,i-1}), 1 \rangle v_{k} \otimes 
				\bb_{m-n,\,m} \nonumber  \\
			&& \otimes 1)  \nonumber \\  
			&&
				+\sum_{j} \sum_{i=0}^{n-1} (-1)^{m(n+1)+i+1}
				d( f( \bb_{1,\,m-n+i} \otimes \bu_{k} \otimes \nuina_{i+1,\,n-1}) \nu a_{0} \otimes \ba_{1,\,i} \otimes \bv_{k} \otimes \bb_{m-n+i+1,\,m} 
				\nonumber  \\
			&& \otimes 1). 
				\nonumber \\
		\end{eqnarray}
	Combining $(\ref{eq114}), (\ref{eq115})$ and $(\ref{eq116})$, we obtain a formula
		\begin{align*}
			[f, \kappa_{n-1,\,1} (z)]_{\rm sg}(\bb_{1,\,m}) + \delta(*)(\bb_{1,\,m}) + \varphi_{m-n-1,\,0}^{n-1}(\delta(*))(\bb_{1,\,m})
			= 
			\varphi_{m-n-1,\,0}^{n+1} ( \{ f, z \})(\bb_{1,\,m})
		\end{align*}
	in $\overline{\Omega}^{n+1}(A)$ for all $\bb_{1,\,m} \in \bA^{\otimes m}$ and therefore
		\begin{align*}
			[ f, \kappa_{n-1,\,1} (z)]_{\rm sg} 
			= 
			\varphi_{m-n-1,\,0}^{n+1}( \{ f, z \})
		\end{align*}
	in $\mathrm{Ext}_{A^{\rm e}}^{m}(A, \overline{\Omega}^{n+1}(A))$ for $f \in \mathrm{Ker}\, \delta_{(1)}^{m}$ and 
	$z = a_{0} \otimes \ba_{1,\,n-1}  \in \mathrm{Ker}\, \partial_{n-1}^{(1)}$. This completes the proof of the statement.
\end{proof}
%
%

%
%
%
	\begin{prop} \label{prop6.17} 
	Let $A$ be a finite dimensional Frobenius $k$-algebra with the Nakayama automorphism $\nu$ of $A$ diagonalizable, 
	and let $m, n$ be integers such that $ n \geq m \geq 1$, so $m-n-1 < 0$.
	Then we have a commutative diagram
		\[\xymatrix@C=40pt{
			\widehat{\mathrm{HH}}_{(1)}^{m}(A) \otimes \widehat{\mathrm{HH}}_{(1)}^{-n}( A ) \ar[r]^-{\{ \ ,\ \}} \ar[d]^{\cong} 
			&
			\widehat{\mathrm{HH}}_{(1)}^{m-n-1}(A) \ar[d]_{\cong} 
			\\ 
			\mathrm{Ext}_{A^{\rm e}}^{m}(A, A) \otimes \mathrm{Tor}_{n-1}^{A^{\rm e}}(A, A_{\nu^{-1}}) \ar[d]_{\mathrm{id} \otimes \kappa_{n-1, 1} }^{\cong}
			& 
			\mathrm{Tor}_{n-m}^{A^{\rm e}}(A, A_{\nu^{-1}}) \ar[d]^{\kappa_{n-m, m} }_{\cong} 
			\\
			\mathrm{Ext}_{A^{\rm e}}^{m}(A, A) \otimes \mathrm{Ext}_{A^{\rm e}}^{1}(A, \overline{\Omega}^{n+1} (A)) \ar[r]^-{ [\ ,\ ]_{\rm{sg}} } \ar[d]^{\cong}
						&
			\mathrm{Ext}_{A^{\rm e}}^{m}(A, \overline{\Omega}^{n+1} (A)) \ar[d]_{\cong} 
			\\
			\underline{\mathrm{Ext}}_{A^{\rm e}}^{m}(A, A) \otimes \underline{\mathrm{Ext}}_{A^{\rm e}}^{-n}(A, A) \ar[r]^-{ [\ ,\ ]_{\rm{sg}} } & \underline{\mathrm{Ext}}_{A^{\rm e}}^{m-n-1}(A, A),
			}\]
	where  $\{ \ , \ \} : \mathcal{D}_{(1)}^{m}(A, A) \otimes \mathcal{D}_{(1)}^{-n}(A, A) \rightarrow \mathcal{D}_{(1)}^{m-n-1}(A, A)$ is defined by
		\begin{align} 
			\{ f, z \} 
			&= (-1)^{ |f| |z| +|f| +|z| } \left (
				(-1)^{ |z|+1 } B^{\nu^{-1}}(f \star_{1} z) 
				+ (-1)^{ |f| } \Delta^{\nu}(f) \star_{1} z 
				+ (-1)^{ |z| } f \star_{1} B^{\nu^{-1}}(z) \right ) \nonumber \\[3pt]
			&= (-1)^{ |f| |z| +|f| +|z| } \left (
				(-1)^{ |f|+1 } \widehat{\Delta}(f \star_{1} z) 
				+ (-1)^{ |f| } \widehat{\Delta}(f) \star_{1} z
				+ f \star_{1} \widehat{\Delta}(z) \right ) \nonumber 
		\end{align}
	for $f \otimes z \in \mathcal{D}_{(1)}^{m}(A, A) \otimes \mathcal{D}_{(1)}^{-n}(A, A)$.
	\end{prop}

%
%
%
%
%
	\begin{prop} \label{prop-6.20}		
		Let $A$ be a finite dimensional Frobenius $k$-algebra with the Nakayama automorphism $\nu$ of $A$ diagonalizable,
		and let $m, n$ be integers such that $m \geq 1$ and $n \geq 1$. 
		Then we have the following commutative diagram:
			\[
			\xymatrix@C=40pt{
				\widehat{\mathrm{HH}}_{(1)}^{-m}( A ) \otimes \widehat{\mathrm{HH}}_{(1)}^{-n}( A ) 
					\ar[r]^-{\{ \ ,\ \}}  \ar[d]^{\cong}  	&
				\widehat{\mathrm{HH}}_{(1)}^{-m-n-1}( A ) \ar[d]_{\cong} \\ 
				\mathrm{Tor}_{m-1}^{A^{\rm e}}(A, A_{\nu^{-1}}) \otimes \mathrm{Tor}_{n-1}^{A^{\rm e}} (A, A_{\nu^{-1}})  
					\ar[d]_{\kappa_{m-1, 1} \otimes \kappa_{n-1, 1} }^{\cong}    & 
				\mathrm{Tor}_{m+n}^{A^{\rm e}}(A, A_{\nu^{-1}}) 
					\ar[d]^{\kappa_{m+n, 1}}_{\cong} \\
				\mathrm{Ext}_{A^{\rm e}}^{1}(A, \overline{\Omega}^{m+1} (A)) \otimes \mathrm{Ext}_{A^{\rm e}}^{1}(A, \overline{\Omega}^{n+1} (A)) 
					\ar[r]^-{ [\ ,\ ]_{\rm{sg}}  } \ar[d]^{\cong}		&
				\mathrm{Ext}_{A^{\rm e}}^{1}(A, \overline{\Omega}^{m+n+2} (A)) 
					\ar[d]_{\cong} \\
				\underline{\mathrm{Ext}}_{A^{\rm e}}^{-m}(A, A) \otimes \underline{\mathrm{Ext}}_{A^{\rm e}}^{-n}(A, A) 
					\ar[r]^-{ [\ ,\ ]_{\rm{sg}}  }   &  
				\underline{\mathrm{Ext}}_{A^{\rm e}}^{-m-n-1}(A, A),
			}\]
	where $\{ \ , \ \} : \mathcal{D}_{(1)}^{-m}(A, A) \otimes \mathcal{D}_{(1)}^{-n}(A, A) \rightarrow \mathcal{D}_{(1)}^{-m-n-1}(A, A)$ is defined by
		\begin{align} 
		\{ w, z \} &= (-1)^{ |w| |z| +|w| +|z| } \left (
						(-1)^{|z|+1} B^{\nu^{-1}}(w \star_{1} z) 
						+  B^{\nu^{-1}}(w) \star_{1} z 
						+ (-1)^{|z|}  w \star_{1} B^{\nu^{-1}}(z) \right ) \nonumber \\[3pt]
					&= (-1)^{ |w| |z| +|w| +|z| } \left (
						(-1)^{|w|+1}  \widehat{\Delta} (w \star_{1} z) 
						+ (-1)^{|w|}  \widehat{\Delta} (w) \star_{1} z
						+  w \star_{1} \widehat{\Delta} (z) \right ) \nonumber 
		\end{align}
 	for $w \otimes z \in \mathcal{D}_{(1)}^{-m}(A, A) \otimes \mathcal{D}_{(1)}^{-n}(A, A)$.
 \end{prop}

The following is a consequence of Lambre-Zhou-Zimmermann.
%
%
%
%
	\begin{prop}[{\cite[Corollary 3.8]{Lam}}]  \label{prop-6.21}		
		Let $A$ be a finite dimensional Frobenius $k$-algebra with the Nakayama automorphism $\nu$ of $A$ diagonalizable, 
		and let $m, n$ be integers such that $m > 0$ and $n > 0$. 
		Then we have the following commutative diagram:
			\[
			\xymatrix@C=40pt{
				\widehat{\mathrm{HH}}_{(1)}^{m}( A ) \otimes \widehat{\mathrm{HH}}_{(1)}^{n}( A ) 
					\ar[r]^-{\{ \ ,\ \}}  \ar[d]^{\cong}  	&
				\widehat{\mathrm{HH}}_{(1)}^{m+n-1}( A ) \ar[d]_{\cong} \\ 
				\mathrm{Ext}_{A^{\rm e}}^{m}(A, A) \otimes \mathrm{Ext}_{A^{\rm e}}^{n}(A, A)
					\ar[d]_{ }^{\cong}  \ar[r]^-{[\ ,\ ]}     & 
				\mathrm{Ext}_{A^{\rm e}}^{m+n-1}(A, A)
					\ar[d]^{ }_{\cong}  \\
				\underline{\mathrm{Ext}}_{A^{\rm e}}^{m}(A, A) \otimes \underline{\mathrm{Ext}}_{A^{\rm e}}^{n}(A, A) 
					\ar[r]^-{[\ ,\ ]_{\rm sg}}   &  
				\underline{\mathrm{Ext}}_{A^{\rm e}}^{m+n-1}(A, A),
			}\]
	where $[\ ,\ ]$ is the Gerstenhaber bracket on Hochschild cohomology and 
	$\{ \ , \ \} : \mathcal{D}_{(1)}^{m}(A, A) \otimes \mathcal{D}_{(1)}^{n}(A, A) \rightarrow \mathcal{D}_{(1)}^{m+n-1}(A, A)$ is defined by
	\begin{align*} 
		\{ f, g \} &= (-1)^{ |f| |g| +|f| +|g| } \left (
						(-1)^{|f|+1} \Delta^{\nu}(f \star_{1} g) 
						+ (-1)^{|f|}  \Delta^{\nu}(f) \star_{1} g 
						+ f \star_{1} \Delta^{\nu}(g) \right )  \\[3pt]
					&= (-1)^{ |f| |g| +|f| +|g| } \left (
						(-1)^{|f|+1}  \widehat{\Delta} (f \star_{1} g) 
						+ (-1)^{|f|}  \widehat{\Delta} (f) \star_{1} g
						+  f \star_{1} \widehat{\Delta} (g) \right )  
	\end{align*}
 for $f \otimes g \in \mathcal{D}_{(1)}^{m}(A, A) \otimes \mathcal{D}_{(1)}^{n}(A, A)$.
 \end{prop}
%
%
	\begin{rem} \label{remark_1} {\rm
		We have to consider the case of either $m = 0$ or $n = 0$. If $m \geq 0$ and $n = 0$, then we will prove that there is a commutative diagram
			\[
			\xymatrix@C=40pt{
				\widehat{\mathrm{HH}}_{(1)}^{m}(A) \otimes \widehat{\mathrm{HH}}_{(1)}^{0}(A) 
					\ar[r]^-{\{ \ ,\ \}}  \ar[d]^{\cong}  	&
				\widehat{\mathrm{HH}}_{(1)}^{m-1}(A) \ar[d]_{\cong} \\ 
				\mathrm{Ext}_{A^{\rm e}}^{m}(A, A) \otimes \widehat{\mathrm{HH}}^{0}(A)
					\ar[d]_{ \mathrm{id} \otimes \varphi_{0,\,0}}^{\cong}   & 
				\mathrm{Ext}_{A^{\rm e}}^{m-1}(A, A)
					\ar[d]_{\cong}^{ \varphi_{m-1, 0}}  \\ 
				\mathrm{Ext}_{A^{\rm e}}^{m}(A, A) \otimes \mathrm{Ext}_{A^{\rm e}}^{1}(A, \overline{\Omega}^{1}(A))
					\ar[d]_{ }^{\cong}  \ar[r]^-{[\ ,\ ]_{\rm sg}}  & 
				\mathrm{Ext}_{A^{\rm e}}^{m}(A, \overline{\Omega}^{1}(A))
					\ar[d]^{ }_{\cong}  \\
				\underline{\mathrm{Ext}}_{A^{\rm e}}^{m}(A, A) \otimes \underline{\mathrm{Ext}}_{A^{\rm e}}^{0}(A, A) 
					\ar[r]^-{[\ ,\ ]_{\rm sg}}   &  
				\underline{\mathrm{Ext}}_{A^{\rm e}}^{m-1}(A, A),
			}\]
		where the vertical isomorphism $\varphi_{0,\,0} : \widehat{\mathrm{HH}}^{0}(A) \rightarrow \mathrm{Ext}_{A^{\rm e}}^{1}(A, \overline{\Omega}^{1}(A))$ is defined in
		Proposition \ref{Frobenius-prop1} and $\{ \ , \ \}$ is defined by
				\begin{align*} 
					\{ f, g \} &= (-1)^{ |f| } \left (
									(-1)^{|f|+1} \Delta^{\nu}(f \star_{1} g) 
									+(-1)^{|f|}  \Delta^{\nu}(f) \star_{1} g  \right ) \nonumber 
				\end{align*}
 		for $f \otimes g \in \mathcal{D}_{(1)}^{m}(A, A) \otimes \mathcal{D}_{(1)}^{0}(A, A)$.
		We must show that 
				\begin{align*} 
					\varphi_{m-1, 0}(\{ \ , \ \} (f \otimes g )) = ([\ , \ ]_{\rm sg} (\mathrm{id} \otimes \varphi_{0,\,0})) (f \otimes g )
				\end{align*}
		in $\mathrm{Ext}_{A^{\rm e}}^{m}(A, \overline{\Omega}^{1}(A))$ for $f \otimes g \in \mathrm{Ker}\, \delta_{(1)}^{m} \otimes \mathrm{Ker}\, \delta_{(1)}^{0}$.
		A direct calculation shows that we have 
			\begin{align*}
					[f, \varphi_{0,\,0}(g)]_{\rm sg} = \varphi_{m-1,\,0}([f, g])
			\end{align*}
		as maps, where $[\ , \ ]$ is the Gerstenhaber bracket on Hochschild cohomology. 
		It follows from \cite[Corollary 3.8]{Lam} that 
		$[f, g] = -\Delta^{\nu}(f \star_{1} g) +  \Delta^{\nu}(f) \star_{1} g$ 
		in $\mathrm{Ext}_{A^{\rm e}}^{m-1}(A, A)$.
		As a result, we obtain a formula in $\mathrm{Ext}_{A^{\rm e}}^{m}(A, \overline{\Omega}^{1}(A))$:
			\begin{align*}
					[f, \varphi_{0,\,0}(g)]_{\rm sg} 
						= \varphi_{m-1,\,0}([f, g]) 
						= \varphi_{m-1,\,0}(-\Delta^{\nu}(f \star_{1} g) +  \Delta^{\nu}(f) \star_{1} g) 
						=  \varphi_{m-1,\,0}(\{ f, g \}).
			\end{align*}
		Similarly, one can prove our claim in the other case $m=0$ and $n \geq 0$.
	}
	\end{rem}

	We are now able to prove Theorem \ref{BV-theo2}.
%
%
	\begin{proof}[{\bf Proof of Theorem \ref{BV-theo2}}]
		It follows from Propositions \ref{prop-6.19}, \ref{prop6.17}, \ref{prop-6.20} and \ref{prop-6.21} and Remark \ref{remark_1} that we have the following commutative diagram
		 	\[\xymatrix@C=40pt{
 				\widehat{\mathrm{HH}}_{(1)}^{m}(A) \otimes \widehat{\mathrm{HH}}_{(1)}^{-n}( A ) \ar[r]^-{\{ \ ,\ \}}  \ar[d]^{\cong}  
				&
				\widehat{\mathrm{HH}}_{(1)}^{m-n-1}(A) \ar[d]_{\cong} 
				\\ 	
				\underline{\mathrm{Ext}}_{A^{\rm e}}^{m}(A, A) \otimes 
				\underline{\mathrm{Ext}}_{A^{\rm e}}^{-n}(A, A) \ar[r]^-{ [\ ,\ ]_{\rm{sg}} }   &  \underline{\mathrm{Ext}}_{A^{\rm e}}^{m-n-1}(A, A),
			}\]
where $m, n$ are arbitrary integers.
Since $(\underline{\mathrm{Ext}}_{A^{\rm e}}^{\bullet}(A, A), \smile_{\rm sg}, [\ ,\ ]_{\rm sg})$ is a Gerstenhaber algebra, we have 
\begin{align*} \lefteqn{}
[f, g \smile_{{\rm sg}} h]_{{\rm sg}} = [f, g]_{{\rm sg}} \smile_{{\rm sg}} h +(-1)^{(|f|-1) |g| } g \smile_{{\rm sg}} [f, h]_{{\rm sg}}
\end{align*}
for arbitrary homogeneous elements $f, g$ and $h \in \underline{\mathrm{Ext}}_{A^{\rm e}}^{\bullet}(A, A)$.
Since we have proved that $[\ ,\ ]_{\rm sg}$ commutes with $\{\ ,\ \}$, using
the defining formula for $\{\ ,\ \}$ and the formula
\begin{align*} \lefteqn{}
\{ f, g \star_{1} h \} =& (-1)^{r} \left(
						(-1)^{|f|+1} \widehat{\Delta}(f \star_{1} g \star_{1} h) + (-1)^{|f|}  \widehat{\Delta}(f) \star_{1} g \star_{1} h
						+ f \star_{1} \widehat{\Delta}(g \star_{1} h) \right)
\end{align*}
with $ r = |f|(|g| + |h|) +|f| +|g| + |h|$, we obtain
\begin{align*} \lefteqn{}
\widehat{\Delta}(f \star_{1} g \star_{1} h ) 
=&\     
\widehat{\Delta}(f \star_{1} g) \star h
+(-1)^{|f|} f \star_{1} \widehat{\Delta}( g \star_{1} h )  
+(-1)^{|g| (|f|-1)} g \star_{1} \widehat{\Delta}( f \star_{1} h ) \\
&- \widehat{\Delta}( f ) \star_{1} g \star_{1} h 
1)^{|f|} f \star_{1} \widehat{\Delta}( g ) \star_{1} h 
1)^{|f| + |g| } f \star_{1} g \star_{1} \widehat{\Delta}( h )
\end{align*}
for arbitrary homogeneous elements $f, g$ and $h \in \widehat{\mathrm{HH}}_{(1)}^{\bullet}(A)$. 
Finally, by the definition of the operator $\widehat{\Delta}$, we get			$\widehat{\Delta}^{2} = 0$ and $\widehat{\Delta}_{0}(1) = 0.$
	\end{proof}
	
	\begin{rem}  {\rm
		Recall that the Nakayama automorphism $\nu$ of $A$ is {\it semisimple} if the map 
		$\nu \otimes \mathrm{id}_{\overline{k}} : A \otimes \overline{k} \rightarrow A \otimes \overline{k}$ is diagonalizable over 
		 the algebraic closure $\overline{k}$ of $k$. 
		The results of Lambre-Zhou-Zimmermann \cite[Section 4]{Lam} and an easy calculation imply that 
		the complete cohomology ring of a Frobenius algebra is a BV algebra when the Nakayama automorphism is semisimple.
		}
	\end{rem}


\section{Examples} \label{Examples}
	Throughout this section, we assume that $k$ is an algebraically closed field whose characteristic $\mathrm{char}\,k$ is $p$. 
	Lambre-Zhou-Zimmermann \cite{Lam} showed that there are many examples of Frobenius algebras with diagonal Nakayama automorphisms. 
	This section is devoted to computing the graded commutative ring structure and 
	the BV structure of the complete cohomology for three certain self-injective Nakayama algebras whose Nakayama automorphisms are diagonalizable.
	Lambre-Zhou-Zimmermann \cite{Lam} also gave an useful and combinatorial criterion to check that the Nakayama automorphism is diagonalizable:
	let $A = kQ/I$ be the algebra given by a quiver with relations. Let $Q_{0}$ be the set of vertices in $Q$. 	
	It is well-known that we can choose a $k$-basis $\mathcal{B}$ of $A$ such that
	$\mathcal{B}$ contains a $k$-basis of the socle of the right regular $A$-module $A$. Suppose that $A$ is a Frobenius algebra. It follows from
	\cite[Proposition 2.8]{HolmZimm} that we can construct an associative and non-degenerate bilinear form 
	$\langle \ ,\ \rangle : A \otimes A \rightarrow k$ by defining $\langle a, b \rangle := tr(ab)$ for $a, b \in A$, where $tr : A \rightarrow k$ is given by
		\begin{align*}
			 {tr}(p)= \begin{cases}
						1	& \mbox{if } p \in \mathcal{B} \cap \mathrm{soc}\, A_{A}, \\
						0	& \mbox{otherwise.}
					\end{cases}
		\end{align*}
	Suppose that $\mathcal{B}$ satisfies two additional conditions:
		\begin{enumerate}
		\renewcommand{\labelenumi}{(\roman{enumi})}
			\item For any two paths $p, q \in \mathcal{B}$, there exist a path $r \in \mathcal{B}$ and a constant $\lambda \in k$ such that
					$p \cdot q = \lambda r$ in $A$;
			\item For every path $p \in \mathcal{B}$, there uniquely exists a path $p^{\prime} \in \mathcal{B}$ such that 
					$0 \not = p \cdot p^{\prime}  \in \mathrm{soc}\, A_{A}$.
		\end{enumerate}
	
		\begin{criterion}[{\cite[Criterion 5.1]{Lam}}] \label{examples-criterion1}
			Under the situation above, assume that $k$ is an algebraically closed field of characteristic zero or of characteristic $p$ larger than 
			the number of arrows of $Q$. Then the Nakayama automorphism of $A$ associated with the bilinear form 
			$\langle \ ,\ \rangle : A \otimes A \rightarrow k$	given above is diagonalizable over $k$.
		\end{criterion}
	Suppose that $A = kQ/I$ is a self-injective Nakayama algebra. 
	It is known that the ordinary quiver $Q$ of $A$ is a cyclic quiver with $|Q_{0}| = s$, and an admissible ideal $I$ of	$kQ$
	is of the form $R_{Q}^{N}$, where $R_{Q}$ is the arrow ideal of $kQ$ and $N \geq 2$.  
	Obviously, we can take a $k$-basis $\mathcal{B}$ of $A$ consisting of paths contains a $k$-basis of $\mathrm{soc}\, A_{A}$. 
	Since any indecomposable projective $A$-module is uniserial, $\mathcal{B}$ satisfies the two condition (i) and (ii). 
	 Hence, we can rewrite Criterion \ref{examples-criterion1} as follows:
%
%
		\begin{criterion} \label{examples-criterion2}
			Let $A= kQ/R_{Q}^{N}$ be a self-injective Nakayama algebra. 
			If the characteristic of $k$ is zero or $p$ larger than the number of arrows of $Q$, 
			then the Nakayama automorphism of $A$ is diagonalizable over $k$.
		\end{criterion}
%
%
		\begin{rem}  \label{examples-remark1}  {\rm
			If $A= kQ/R_{Q}^{N}$ is a self-injective Nakayama algebra, then the exponent $N$ does not affect
			Criterion \ref{examples-criterion2}, and only the number of arrows of $Q$ is important.
		}
		\end{rem}
	We will compute BV algebras of Nakayama algebras $A = kQ/R_{Q}^{N}$ with $|Q_{0}| =s$ for three cases.
%
%
	\subsection{The case $s =2, N=2$.}
	 Let $Q$ be a quiver 
		\[
			\xymatrix{
				1 \ar@<0.5ex>[r]^-{\alpha_{1}} & 2 \ar@<0.5ex>[l]^-{\alpha_{2}}.
			}			
		\]
	Consider the algebra  $A := kQ / R_{Q}^{2}$. Thus, $A$ is a self-injective Nakayama algebra and, moreover, a truncated algebra.
	It follows from Criterion \ref{examples-criterion2} that the Nakayama automorphism of $A$ is diagonalizable if and only if $\mathrm{char}\,k \not = 2$.
	Thus, we suppose that $\mathrm{char}\,k \not = 2$. 
	Note that we need the assumption on $\mathrm{char}\,k$ only if  we construct BV differential.
	However, we  assume that  $\mathrm{char}\,k \not = 2$ in advance.
	We denote by $e_{i}$ the primitive idempotent of $A$ corresponding to a vertex $i$ of $Q$ such that 
	$e_{i} \alpha_{i} e_{i+1} = \alpha_{i}$ holds, where we regard the subscripts $i$ of $e_{i}$ and $\alpha_{i}$ modulo 2.	
   	Take a $k$-basis $\mathcal{B} = (u_{1}, u_{2}, u_{3}, u_{4}) = (e_{1}, e_{2}, \alpha_{1}, \alpha_{2})$ of $A$. 
	Clearly, it contains a $k$-basis $\{ \alpha_{1}, \alpha_{2} \}$ of $\mathrm{soc}\, A_{A}$. We hence get an associative and non-degenerate bilinear form 
	$\langle \ ,\ \rangle : A \otimes A \rightarrow k$ and the dual basis 
	$\mathcal{B}^{*} = (v_{1}, v_{2}, v_{3}, v_{4}) = (\alpha_{2}, \alpha_{1}, e_{1}, e_{2})$ of $A$ 
	such that $\langle v_{i} ,u_{j} \rangle = \delta_{ij}$, where $\delta_{ij}$ denotes Kronecker's delta. 
	Under the basis $\mathcal{B}$, the representation matrix of the Nakayama automorphism $\nu$ of $A$ is 
		\begin{align*}
			\begin{pmatrix}
				0&1&0&0\\
				1&0&0&0\\
				0&0&0&1\\
				0&0&1&0
			\end{pmatrix}
		\end{align*}
	and is similar to a diagonal matrix
		\begin{align*}
			\begin{pmatrix}
				1&0&0&0\\
				0&1&0&0\\
				0&0&-1&0\\
				0&0&0&-1
			\end{pmatrix}.
		\end{align*}
	Moreover, we have a decomposition $A = A_{1} \oplus A_{-1} $ of $A$ by two $k$-vector spaces
		\begin{align*} \lefteqn{}
			&A_{1} = \mathrm{Ker}\, (\nu-\mathrm{id}) = k\,1_{A} \oplus k\,(\alpha_{1} + \alpha_{2}), \\
			&A_{-1} = \mathrm{Ker}\, (\nu+\mathrm{id}) = k\,(e_{1} -e_{2}) \oplus k\,(\alpha_{1} - \alpha_{2}).
		\end{align*}
	
	Let us recall that a set $\{ Ae_{i} \otimes e_{j}A \, |\,  i, j \in Q_{0} \}$ is a complete set of pairwise non-isomorphic finitely generated 
	indecomposable projective 
	$A$-bimodules, and we denote by $P(i, j)$ the indecomposable projective $A$-bimodule $Ae_{i} \otimes e_{j}A$.
	It follows from \cite{Bardzell} that a minimal projective resolution $\mathcal{P}_{\bullet}$ of $A$ as an $A$-bimodule is
	an exact sequence
		\begin{align*}			
			\cdots 
			\rightarrow	P_{2r+1} 
			\xrightarrow{\phi_{2r+1}}	P_{2r} 
			\xrightarrow{\phi_{2r}}	P_{2r-1} 
			\rightarrow	\cdots 
			\rightarrow	P_{1} 
			\xrightarrow{\phi_{1}}	P_{0} 
			\xrightarrow{\phi_0}	A 
			\rightarrow	0,		
		\end{align*}
	where 
		\begin{align*}
			 P_{n} := \begin{cases}
						P(1, 2) \oplus P(2, 1)	& \mbox{if $n$ is odd},  \\
						P(1, 1) \oplus P(2, 2)	& \mbox{if $n$ is even}
					\end{cases}
		\end{align*}
		and $A$-bimodule homomorphisms $\phi_{*} : P_{*} \rightarrow P_{*-1}$ are defined as follows:
		\begin{align*} \lefteqn{}
			&\phi_0(e_{i} \otimes e_{i}) = e_{i}; \quad\\
			&\phi_{2r}(e_{i} \otimes e_{i}) = e_{i} \otimes \alpha_{i+1} + \alpha_{i} \otimes e_{i}; \quad\\
			&\phi_{2r+1}(e_{i} \otimes e_{i+1}) = \alpha_{i} \otimes e_{i+1} -e_{i} \otimes \alpha_{i}.
		\end{align*}
	For a finite dimensional $k$-vector space $V$ and a $k$-basis $\mathcal{B}$ of $V$, given a basis vector $p \in \mathcal{B}$, we denote by $p^{*}$ 
	the $k$-linear map $V \rightarrow k$ sending $q \in \mathcal{B}$ to 1 if $q = p$ and to $0$ otherwise.
	Applying the exact functor $D = \Hom(-, k)$ to $\mathcal{P}_{\bullet}$ and 
	twisting each term of $D(\mathcal{P}_{\bullet})$ by the automorphism $\nu^{-1}$ on the right hand side, we get an exact sequence
	$D(\mathcal{P}_{\bullet})_{\nu^{-1}}$ as follows:
		\begin{align*}			
			0	\rightarrow
			D(A)_{\nu^{-1}}	\xrightarrow{D(\phi_0)}
			\cdots \rightarrow
			D(P_{2r-1})_{\nu^{-1}} \xrightarrow{D(\phi_{2r})}	
			D(P_{2r})_{\nu^{-1}} \xrightarrow{D(\phi_{2r+1})}
			D(P_{2r+1})_{\nu^{-1}} \rightarrow
			\cdots,					
		\end{align*}
	where
		\begin{align*}
			 D(P_{n})_{\nu^{-1}} = \begin{cases}
										A ( \alpha_{2} \otimes \alpha_{2})^{*} A \oplus A ( \alpha_{1} \otimes \alpha_{1})^{*} A	& \mbox{if $n$ is odd},  \\
										A ( \alpha_{2} \otimes \alpha_{1})^{*} A \oplus A ( \alpha_{1} \otimes \alpha_{2})^{*} A 	& \mbox{if $n$ is even}
									\end{cases}				
		\end{align*}
	and $A$-bimodule homomorphisms $D(\phi_{*}) : D(P_{*-1})_{\nu^{-1}} \rightarrow D(P_{*})_{\nu^{-1}}$ are defined as follows:
		\begin{align*} \lefteqn{}
			&D(\phi_0)(\langle -, 1_{A} \rangle) 
				= \alpha_{1}( \alpha_{2} \otimes \alpha_{1})^{*} +( \alpha_{2} \otimes \alpha_{1})^{*}  \alpha_{1}
					+\alpha_{2}( \alpha_{1} \otimes \alpha_{2})^{*}+( \alpha_{1} \otimes \alpha_{2})^{*}  \alpha_{2}; \\
			&D(\phi_{2r})( (\alpha_{i} \otimes \alpha_{i})^{*} ) 
				= \alpha_{i+1} (\alpha_{i} \otimes \alpha_{i+1})^{*} + (\alpha_{i+1} \otimes \alpha_{i})^{*} \alpha_{i} ; \\
			&D(\phi_{2r+1})((\alpha_{i} \otimes \alpha_{i+1})^{*} ) 
				=  (\alpha_{i} \otimes \alpha_{i})^{*} \alpha_{i+1} -\alpha_{i+1} (\alpha_{i+1} \otimes \alpha_{i+1})^{*}.
		\end{align*} 
	Therefore, we obtain an exact sequence $X_{\bullet}$
		\[
			\xymatrix{				
				\cdots	\ar[r]		& 
				P_{2}	\ar[r]^-{\phi_{2}}		&
				P_{1}	\ar[r]^-{\phi_{1}}		&
				P_{0}	\ar[r]^-{\mu}	\ar[d]_-{\phi_0}		&
				D(P_{0})_{\nu^{-1}} \ar[r]^-{D(\phi_{1})}		&
				D(P_{1})_{\nu^{-1}} \ar[r]^-{D(\phi_{2})}		&
				D(P_{2})_{\nu^{-1}} \ar[r]	&
				\cdots 	\\
				&
				&
				&
				A 	\ar[r]^-{\cong}	&
				D(A)_{\nu^{-1}}	\ar[u]_-{D(\phi_0)}		\ar@{}[lu]|{\circlearrowright}	&
				&
				&
			}
		\]
	of which the composition $\mu$ is defined by
		\begin{align*} \lefteqn{}
			&\mu(e_{i} \otimes e_{i}  ) 
				=  \alpha_{i} \otimes e_{i} +\alpha_{i+1} \otimes e_{i+1}
		\end{align*} 		
	and whose term $P_{n}$ is of degree $n \geq 0$.	
	Observe that there are $A$-bimodule isomorphisms
		\begin{align*} \lefteqn{}			
			D(P(i, j))_{\nu^{-1}} &= D( {}_{\nu^{-1}} Ae_{i} \otimes e_{j}A )
									\cong \Hom( e_{j}A, D({}_{\nu^{-1}}Ae_{i})) \\
									&\cong D(e_{j}A) \otimes D(Ae_{i})_{\nu^{-1}}
									\cong Ae_{j+1} \otimes e_{i+1}A_{\nu^{-1}} \\
									&\cong Ae_{j+1} \otimes e_{i} A
									=P(j+1, i),				
		\end{align*} 
	where the fourth isomorphism is induced by the $A$-bimodule isomorphism $A_{\nu} \cong D(A)$ and the fact that $\nu = \nu^{-1}$. 
	Since $A^{\rm e}$ is injective as an $A$-bimodule, the contravariant functor $\Hom_{A^{\rm e}}(-, A^{\rm e})$ is exact, 
	so that the exact sequence $X_{\bullet}$
	is a complete resolution of $A$. 
	Before applying the functor $\Hom_{A^{\rm e}}(-, A)$ to $X_{\bullet}$, we notice that 
	there are isomorphisms
		\begin{align*} \lefteqn{}			
			\Hom_{A^{\rm e}}(D(P(i, j))_{\nu^{-1}}, A)
				&\cong \Hom_{A^{\rm e}}(D(P(i, j)), D(A))
				\cong	D(A \otimes_{A^{\rm e}} D(P(i, j))) \\
				&\cong \Hom_{A^{\rm e}}(A, P(i, j))
				\cong	\Hom_{A^{\rm e}}(A, A^{\rm e}) \otimes_{A^{\rm e}} P(i, j) \\
				&\cong	A_{\nu^{-1}} \otimes_{A^{\rm e}} P(i, j)				
		\end{align*} 
	for any $i, j \in Q_{0}$. Using these isomorphisms, we have the following commutative diagram with exact rows:\\[-20pt]
	\begin{center}
		\scalebox{0.9}{	
			\xymatrix@!C=168pt{
				\Hom_{A^{\rm e}}(D(P_{2r+1})_{\nu^{-1}}, A) \ar[r]^-{\Hom (D(\phi_{2r+1}), A)} \ar[d]^{\cong} &
				\Hom_{A^{\rm e}}(D(P_{2r})_{\nu^{-1}}, A) \ar[r]^-{\Hom (D(\phi_{2r}), A)} \ar[d]^{\cong} &
				\Hom_{A^{\rm e}}(D(P_{2r-1})_{\nu^{-1}}, A) \ar[d]^{\cong}  \\
				A_{\nu^{-1}} \otimes_{A^{\rm e}}  P_{2r+1} \ar[r]^-{\mathrm{id} \otimes  \phi_{2r+1}} & 
				A_{\nu^{-1}} \otimes_{A^{\rm e}}  P_{2r} \ar[r]^-{\mathrm{id} \otimes  \phi_{2r}} \ar@{}[lu]|{\circlearrowright}	&
				A_{\nu^{-1}} \otimes_{A^{\rm e}}  P_{2r-1}  \ar@{}[lu]|{\circlearrowright},	
			}}\\[10pt]
	\end{center}
	where the $A$-bimodules $A_{\nu^{-1}} \otimes_{A^{\rm e}}  P_{*}$ are given by
		\begin{align*} \lefteqn{}
			&A_{\nu^{-1}} \otimes_{A^{\rm e}}  P_{2r} 
				= k\,( e_{2} \otimes_{A^{\rm e}} e_{1} \otimes e_{2}  ) \oplus k\,(e_{1} \otimes_{A^{\rm e}} e_{2} \otimes e_{1} ); \\
			&A_{\nu^{-1}} \otimes_{A^{\rm e}}  P_{2r+1} 
				=  k\,( \alpha_{2} \otimes_{A^{\rm e}} e_{1} \otimes e_{1}  ) \oplus k\,(\alpha_{1} \otimes_{A^{\rm e}} e_{2} \otimes e_{2} ),
		\end{align*} 	
	and the $k$-linear maps $\mathrm{id} \otimes_{A^{\rm e}} \phi_{*}$ are given by
		\begin{align*} \lefteqn{}
			&\mathrm{id} \otimes  \phi_{2r}( \alpha_{i} \otimes_{A^{\rm e}} e_{i} \otimes e_{i} ) 
				= 0; \\
			&\mathrm{id} \otimes \phi_{2r+1}(e_{i} \otimes_{A^{\rm e}} e_{i+1} \otimes e_{i}  ) 
				=  \alpha_{i} \otimes_{A^{\rm e}} e_{i} \otimes e_{i} -\alpha_{i+1} \otimes_{A^{\rm e}} e_{i+1} \otimes e_{i+1}.
		\end{align*} 		
		 Hence, the complex $\Hom_{A^{\rm e}}(X_{\bullet}, A)$ can be identified with a complex 
	\begin{align*}\lefteqn{}
		\cdots \rightarrow 
		A_{\nu^{-1}} \otimes_{A^{\rm e}}  P_{1}  &\xrightarrow{\mathrm{id} \otimes \phi_{1}}
		A_{\nu^{-1}} \otimes_{A^{\rm e}}  P_{0}\\ &\xrightarrow{\Hom (\mu, A)}
		\Hom_{A^{\rm e}}(P_{0}, A) \xrightarrow{\Hom (\phi_{1}, A)}
		\Hom_{A^{\rm e}}(P_{1}, A) \rightarrow \cdots
	\end{align*}
	of which the remaining terms and differentials are given by
		\begin{align*} \lefteqn{}
			 &\Hom_{A^{\rm e}}(P_{n}, A)
				\cong \begin{cases}
						 e_1 Ae_2 \oplus e_2 Ae_1 = k\,\alpha_{1} \oplus k\,\alpha_{2}	& \mbox{if $n$ is odd}, \\
						e_1 Ae_1\oplus e_2 Ae_2	 = k\,e_{1} \oplus k\,e_{2}				& \mbox{if $n$ is even};
					\end{cases} \\	
			&\Hom_{A^{\rm e}}(\phi_{2r+1}, A)(e_{i} ) 
				=  \alpha_{i+1} -\alpha_{i}; \\	
			&\Hom_{A^{\rm e}}(\phi_{2r}, A)( \alpha_{i} ) 
				= 0 ;  \\
			&\Hom_{A^{\rm e}}(\mu, A)( \alpha_{i} \otimes_{A^{\rm e}} e_{i} \otimes e_{i}  ) 
				= 0
		\end{align*} 
	and whose term $\Hom_{A^{\rm e}}(P_{n}, A)$ is of degree $n \geq 0$.	

	Therefore, the complete cohomology groups $\widehat{\mathrm{HH}}^{*}(A)$ are given as follows: for $n \geq 0$
		\begin{align} 
			 &\widehat{\mathrm{HH}}^{n}(A) \label{eq_1}
				= \begin{cases}
						 k\,\overline{\alpha_{1}}  	&  \mbox{if $n$ is odd}, \\
						 k\,\overline{1_{A}} 			& \mbox{if $n$ is even};
					\end{cases}	\\[5pt] 
			 &\widehat{\mathrm{HH}}^{-n}(A) \label{eq_2}
				= \begin{cases}			 
						k\,\overline{\alpha_{1} \otimes_{A^{\rm e}} e_{1} \otimes e_{1}} 					& \mbox{if $n$ is odd}, \\
						k\,\overline{ e_{1} \otimes_{A^{\rm e}} e_{2} \otimes e_{1} + e_{2} \otimes_{A^{\rm e}} e_{1} \otimes e_{2}}	& \mbox{if $n >0$ is even}.
					\end{cases}	
		\end{align} 	
	Observe that we have $\widehat{\mathrm{HH}}^{0}(A) = \HH^{0}(A)$ and $\widehat{\mathrm{HH}}^{-1}(A) = \Hh_{0}(A, A_{\nu^{-1}})$.

%
%
%
%
	From now on, we fix a $k$-basis 
	    \[
	        (u_{1}, u_{2}, u_{3}, u_{4}) = (1, \alpha_{1}+\alpha_{2}, e_{1} -e_{2}, \alpha_{1}-\alpha_{2})
	    \] 
    of $A$ consisting of eigenvectors 
	associated with the eigenvalues of the diagonalizable Nakayama automorphism $\nu$ of $A$. 
	Then we have its dual basis 
	    \[
	        (v_{1}, v_{2}, v_{3}, v_{4}) = \bigl((1/2)(\alpha_{1}+\alpha_{2}), 1/2, (1/2)(\alpha_{1}-\alpha_{2}), (1/2)(e_{1}-e_{2}) \bigr)
	    \] 
	of $A$.
	Following \cite{AmesCaglieroTirao}, we will construct comparison morphisms between the minimal projective resolution 
	$\mathcal{P}_{\bullet}$ and the normalized bar resolution $\mathrm{Bar}_{\bullet}(A)$ of $A$ 
	(cf.\,\cite{RedondoRoman2018} for monomial algebras in general). 
	Let ${\bf F}_{0} $ be the canonical inclusion $P_{0} \hookrightarrow A \otimes A$, and for each $n >0$, we define 
	${\bf F}_{n} : P_{n} \rightarrow A \otimes \bA^{\otimes n} \otimes A$ in the following way:
	if $n = 2r $, then let
		\begin{align*} 
			 {\bf F}_{2r}(e_{i} \otimes e_{i} )
				=  1 \otimes \overbrace{\balpha_{i} \otimes \balpha_{i+1} \otimes \cdots \otimes \balpha_{i} \otimes \balpha_{i+1}}^{2r} \otimes1,			 
		\end{align*} 	
	where $\balpha_{i}$ and $\balpha_{i+1}$ appear each other. 
	If $n = 2r+1$, then let
		\begin{align*} 
			 {\bf F}_{2r+1}(e_{i} \otimes e_{i+1} )
				=  1 \otimes  
					\overbrace{\balpha_{i} \otimes \balpha_{i+1} \otimes \cdots \otimes \balpha_{i} \otimes \balpha_{i+1} \otimes \balpha_{i}}^{2r+1}  
					\otimes 1.			 
		\end{align*} 
	On the other hand, let ${\bf G}_{0}$ be the canonical projection $A \otimes A \rightarrow P_{0}$, and for each $n >0$, 
	${\bf G}_{n} : A \otimes \bA^{\otimes n} \otimes A \rightarrow P_{n}$ is given as follows:
	if $n = 2r $, then let
		\begin{align*} 
			 {\bf G}_{2r}(w)
			= \begin{cases}
					e_{i} \otimes e_{i} & 
					\mbox{if $w = 1 \otimes \balpha_{i} \otimes \balpha_{i+1} \otimes \cdots \otimes \balpha_{i} \otimes \balpha_{i+1} \otimes 1 $},  
					\\						
					0 & \mbox{otherwise}.
					\end{cases}	
		\end{align*} 	
	If $n = 2r+1 $, then let
		\begin{align*} 
			 {\bf G}_{2r+1}(w)
			= \begin{cases}
					e_{i} \otimes e_{i+1} & 
					\mbox{if 
					$w = 1 \otimes \balpha_{i} \otimes \balpha_{i+1} \otimes \cdots \otimes \balpha_{i} \otimes \balpha_{i+1} \otimes \balpha_{i} \otimes 1 $},  
					\\						
					0 & \mbox{otherwise}.
					\end{cases}	
		\end{align*} 	
	One can easily check that ${\bf F}$ and ${\bf G}$ are comparison morphisms.
%
%
%
%
	Using these comparison morphisms and the definition of the $\star$-product $\star$, we have the following result.
%
%
		\begin{prop} \label{examples-prop1}
			For every $i \in \mathbb{Z}$, the $n$-th complete cohomology group $\widehat{\mathrm{HH}}^{n}(A)$ of $A$ is of dimension one, and 
			the complete cohomology ring $(\widehat{\mathrm{HH}}^{\bullet}(A), \star)$ is isomorphic to
				\begin{align*}  \lefteqn{}
					k[\alpha, \beta, \gamma] /\langle \alpha \gamma - 1, \beta^2 \rangle
				\end{align*}
			with $| \alpha | = 2, |\beta|= 1$ and $|\gamma| = -2$, where $\alpha$, $\beta$ and $\gamma$ correspond to   $\overline{1_{A}} \in \widehat{\mathrm{HH}}^{2}(A)$ in  (\ref{eq_1}), $\overline{\alpha_1} \in \widehat{\mathrm{HH}}^{1}(A)$ in  (\ref{eq_1}) and $\overline{ e_{1} \otimes_{A^{\rm e}} e_{2} \otimes e_{1} + e_{2} \otimes_{A^{\rm e}} e_{1} \otimes e_{2}} \in \widehat{\mathrm{HH}}^{-2}(A)$ in  (\ref{eq_2}), respectively.
		\end{prop}
%
%
		\begin{rem}  \label{examples-remark4}	 {\rm	
			As we have seen before, the complete cohomology groups $\widehat{\mathrm{HH}}^{n}(A)$ with $n \geq 0$ of $A$ coincide with 
			the Hochschild cohomology
			groups $\HH^{n}(A)$ of $A$. Hence, the Hochschild cohomology ring $(\HH^{\bullet}(A), \smile)$ of $A$ is a subring of the complete 
			cohomology ring $(\widehat{\mathrm{HH}}^{\bullet}(A), \star)$.
		}
		\end{rem}
%
%
	\begin{rem}	  \label{examples-remark2} {\rm
			We have another description of the complete cohomology ring above as follows:
				\begin{align*}  \lefteqn{}
					k[\alpha, \beta, \alpha^{-1}] /\langle \beta^2 \rangle
				\end{align*}
			where $|\alpha| = 2, |\beta|= 1$ and $| \alpha ^{-1}| = -2$. Therefore, we will write $\alpha^{-1} $ for $\gamma$.
	}
	\end{rem}
%
%
%
%
	Following our main result, we now construct a BV operator $\widehat{\Delta}_{i} : \widehat{\mathrm{HH}}^{i}(A) \rightarrow \widehat{\mathrm{HH}}^{i-1}(A)$ for all $i \in \mathbb{Z}$.
	It follows from Proposition \ref{examples-prop1} that 
	\[\widehat{\mathrm{HH}}^{2l}(A) = k\,\alpha^{l} \quad \mbox{and} \quad \widehat{\mathrm{HH}}^{2l+1}(A) = k\,\beta\alpha^{l} \]
	for all $l \in \mathbb{Z}$.
	Note that the number of the generators contained in the basis element of $\widehat{\mathrm{HH}}^{i}(A)$ is at least $3$ except for $-4 \leq i \leq 4$.
	Thus, one can use the operators $\widehat{\Delta}_{i} : \widehat{\mathrm{HH}}^{i}(A) \rightarrow \widehat{\mathrm{HH}}^{i-1}(A)$ for $-4 \leq i \leq 4$ and the formulas in Definition \ref{def_1} to obtain the remaining operators $\widehat{\Delta}_{*} : \widehat{\mathrm{HH}}^{*}(A) \rightarrow \widehat{\mathrm{HH}}^{*-1}(A)$.
	From this point of view, it suffices to construct  $\widehat{\Delta}_{i}$ only for $i = -4, -2, -1, 1, 2, 3, 4$.
	We will show a way of constructing $\widehat{\Delta}_{1}$ and $\widehat{\Delta}_{-1}$. The others can be constructed in a similar way. 
	Let us recall that every complete cohomology group has a decomposition associated with the product of eigenvalues and in particular,
	except for the cohomology associated with the product of eigenvalues equal to $1_{A}$, the other vanish. 
	Moreover, the BV operator defined on the chain level can be lifted to the cohomology level when we restrict it to the subcomplex associated with the product of eigenvalues equal to $1_{A}$.

	We first compute $\widehat{\Delta}_{1} : \widehat{\mathrm{HH}}^{1}(A) \rightarrow \widehat{\mathrm{HH}}^{0}(A)$. Consider a diagram
		\[\xymatrix{
			\Hom_{A^{\rm e}}(P_{1}, A)	\ar[d]_-{\Hom_{A^{\rm e}}({\bf G}_{1}, A)}	\ar[r]^{\widehat{\Delta}_{1} } & \Hom_{A^{\rm e}}(P_{0}, A)	\\
			\Hom(\bA, A)	\ar[d]_-{\cong}	&	 A 	\ar[u]_-{\Hom_{A^{\rm e}}({\bf F}_{0}, A)}\\
			D(A_{\nu} \otimes \bA)	\ar[r]^-{D(B_{0}^{\nu})}	&	 D(A_{\nu}) 	\ar@{}[luu]|{\circlearrowright}	\ar[u]_-{\cong}.	\\
		}\]
	
	We know $\widehat{\mathrm{HH}}^{1}(A) = k \,\overline{\alpha_{1}}$ and hence deal with only $\alpha_{1}$.
	Put 
		\[
			f_{\alpha_{1}} := \Hom_{A^{\rm e}}({\bf G}_{1}, A)(\alpha_{1}), \quad 
			f_{u_{2}} := \Hom_{A^{\rm e}}({\bf G}_{1}, A)(u_{2}), \quad
			f_{u_{4}} := \Hom_{A^{\rm e}}({\bf G}_{1}, A)(u_{4}).
		\] 
	Namely, each of $f_{\alpha_{1}}, f_{u_{2}}$ and $f_{u_{4}}$ sends $\bx \in \bA $ with $x \in \mathcal{B}$ to
		\begin{align*} 
			 f_{\alpha_{1}}(\bx)
			= \begin{cases}
					\alpha_{1} & 
					\mbox{if $\bx = \overline{\alpha_{1}} $},  
					\\						
					0 & \mbox{otherwise},
					\end{cases} \quad	
			 f_{u_{2}}(\bx)
			= \begin{cases}
					\alpha_{1} & 
					\mbox{if $\bx = \overline{\alpha_{1}} $},  
					\\				
					\alpha_{2} & 
					\mbox{if $\bx = \overline{\alpha_{2}} $},  
					\\						
					0 & \mbox{otherwise},
					\end{cases}	 \quad	
			 f_{u_{4}}(\bx)
			= \begin{cases}
					\alpha_{1} & 
					\mbox{if $\bx = \overline{\alpha_{1}} $},  
					\\				
					-\alpha_{2} & 
					\mbox{if $\bx = \overline{\alpha_{2}} $},  
					\\						
					0 & \mbox{otherwise}.
					\end{cases}
		\end{align*} 	
	 Then we have $f_{\alpha_{1}} = (1/2)\,f_{u_{2}}+(1/2)\,f_{u_{4}}, f_{u_{2}} \in C_{(1)}^{1}(A, A)$ and $f_{u_{4}} \in C_{(-1)}^{1}(A, A)$.
	Since it is sufficient to only consider the image of $(1/2)f_{u_{2}}$,
	a direct computation shows that 
		\[
		\widehat{\Delta}_{1}(\beta) = 	\widehat{\Delta}_{1}(\overline{\alpha_{1}}) = (1/2)\,\overline{1_{A}} =  1/2
		\]
		in $\widehat{\mathrm{HH}}^{0}(A)$. 
	On the other hand, consider a diagram
		\[\xymatrix{
			A_{\nu^{-1}} \otimes_{A^{\rm e}} P_{0}	\ar[d]_-{\mathrm{id} \otimes_{A^{\rm e}} {\bf F}_{0}}	\ar[r]^{\widehat{\Delta}_{-1} } & A_{\nu^{-1}} \otimes_{A^{\rm e}} P_{1}	\\
			A_{\nu^{-1}} 	\ar[r]^-{-B_{0}^{\nu^{-1}}}	&	 A_{\nu^{-1}} \otimes \bA	\ar@{}[lu]|{\circlearrowright}	\ar[u]_-{\mathrm{id} \otimes_{A^{\rm e}} {\bf G}_{1}}.	
		}\]
	We know that $\widehat{\mathrm{HH}}^{-1}(A) = k\, \overline{\alpha_{1} \otimes_{A^{\rm e}} e_{2} \otimes e_{1}}$ holds and hence handle 
	$\alpha_{1} \otimes_{A^{\rm e}} e_{2} \otimes e_{1}$. The element 
	$(\mathrm{id} \otimes_{A^{\rm e}} {\bf F}_{0})(\alpha_{1} \otimes_{A^{\rm e}} e_{2} \otimes e_{1}) = \alpha_{1}$ can be decomposed as 
	$\alpha_{1} = (1/2)\,u_{2} +(1/2)\,u_{4}$ in $A_{\nu^{-1}}$, where $u_{2} \in C_{0}^{(1)}(A, A_{\nu^{-1}})$ and $u_{4} \in C_{0}^{(-1)}(A, A_{\nu^{-1}})$.
	Thus, a direct calculation gives us the formula 
		\[
			\widehat{\Delta}_{-1}(\overline{\alpha_{1} \otimes_{A^{\rm e}} e_{2} \otimes e_{1}}) 
			=
			(-1/2)\,\overline{e_{1} \otimes_{A^{\rm e}} e_{2} \otimes e_{1} + e_{2} \otimes_{A^{\rm e}} e_{1} \otimes e_{2} }
		\]
	 in $\widehat{\mathrm{HH}}^{-2}(A)$.
	Thus, we have $\widehat{\Delta}_{-1}(\alpha^{-1} \beta )$ $=$ $	(-1/2)\, \alpha^{-1}$. 
Combining the formulas in Definition \ref{def_1}, we have the following result.

%
%
	\begin{prop} \label{examples-prop2}
	The nonzero BV differentials $\widehat{\Delta}_{*}$ on $\widehat{\mathrm{HH}}^{\bullet}(A)$ are 
				\begin{align*}  \lefteqn{}
					\widehat{\Delta}_{2n+1}(\alpha^{n} \beta ) = ((2n+1)/2)\, \alpha^{n}  
				\end{align*} 	
	with $n \in \mathbb{Z}$. In particular, the nonzero Gerstenhaber brackets are induced by
			\begin{align*}  \lefteqn{}
				&\{ \alpha, \beta \} = \alpha, \quad \{ \beta, \alpha^{-1} \} = \alpha^{-1}. 
			\end{align*} 				
	\end{prop}
%
%
		\begin{rem}  \label{examples-remark5}	  {\rm	
			Since the non-negative part $\widehat{\mathrm{HH}}^{\geq 0}(A)$ of the complete cohomology $\widehat{\mathrm{HH}}^{\bullet}(A)$ is the Hochschild cohomology of $A$,
			the non-negative BV differential $\widehat{\Delta}_{\geq 0}$ gives rise to a BV differential on the Hochschild cohomology ring of $A$, which means that
			there is a non-trivial example for our main theorem and for the theorem of Lambre-Zhou-Zimmermann \cite[Theorem 4.1]{Lam}.
		}
		\end{rem} 
%
%
%
%
	\subsection{The case $s =3, N=2$.}		
	Let $Q$ be a quiver 	 
		\[\xymatrix@ur{
			3 \ar[r]^-{\alpha_{3}}  & 1 \ar[d]^-{\alpha_{1}} \\
			 & 2 \ar[ul]^-{\alpha_{2}}&
		}\]
	and $A$ the algebra $kQ/R_{Q}^{2}$. 
	It follows from Criterion \ref{examples-criterion2} and the fact that a primitive root of a polynomial $x^{3} - 1$ is not equal to $1 \in k$ when 
	$\mathrm{char}\,k = 2$ that the Nakayama automorphism of $A$ is diagonalizable if and only if $\mathrm{char}\,k \not = 3$.
	Hence we assume that $\mathrm{char}\,k \not = 3$.
	We see that $A$ is a self-injective Nakayama algebra of which the representation matrix of the Nakayama
	automorphism $\nu$ is 
		\begin{align*}
			\begin{pmatrix}
				0&0&1&0&0&0\\
				1&0&0&0&0&0\\
				0&1&0&0&0&0\\
				0&0&0&0&0&1\\
				0&0&0&1&0&0\\
				0&0&0&0&1&0\\
			\end{pmatrix}
		\end{align*}
	under a $k$-basis $(e_{1}, e_{2}, e_{3}, \alpha_{1}, \alpha_{2}, \alpha_{3})$ of $A$. This matrix is similar to a diagonal matrix
		\begin{align*}
			\begin{pmatrix}
				1&0&0&0&0&0\\
				0&1&0&0&0&0\\
				0&0&\omega&0&0&0\\
				0&0&0&\omega&0&0\\
				0&0&0&0&\omega^{2}&0\\
				0&0&0&0&0&\omega^{2}\\
			\end{pmatrix}
		\end{align*}
	where the element $\omega \in k$ is one of roots of a polynomial $x^{2} +x +1$.
	Moreover, we can decompose $A = A_{1} \oplus A_{\omega} \oplus A_{\omega^{2}}$, where
		\begin{align*} \lefteqn{}
			&A_{1} 
				= \mathrm{Ker}\, (\nu-\mathrm{id}) 
				= k\,1_{A} \oplus k\,(\sum_{i=1}^{3} \alpha_{i}),	\\
			&A_{\omega} 
				= \mathrm{Ker}\, (\nu-\omega\, \mathrm{id}) 
				= k\,(\omega^{2} e_{1} +\omega e_{2} +e_{3}) \oplus k\,(\omega^{2} \alpha_{1} +\omega \alpha_{2} +\alpha_{3}),	\\
			&A_{\omega^{2}} 
				= \mathrm{Ker}\, (\nu-\omega^{2}\, \mathrm{id}) 
				= k\,(\sum_{i=1}^{3} \omega^{i} e_{i}) \oplus k\,(\sum_{i=1}^{3} \omega^{i} \alpha_{i}).
		\end{align*}
	Let $l \geq 0$ be an integer. In a similar way to the first example, we have a complete resolution of $A$ as follows:
		\[
			\xymatrix{							
				\cdots	\ar[r]		& 
				P_{2}	\ar[r]^-{\phi_{2}}		&
				P_{1}	\ar[r]^-{\phi_{1}}		&
				P_{0}	\ar[r]^-{\mu}	\ar[d]_-{\phi_0}		&
				D(P_{0})_{\nu^{-1}} \ar[r]^-{D(\phi_{1})}		&
				D(P_{1})_{\nu^{-1}} \ar[r]^-{D(\phi_{2})}		&
				D(P_{2})_{\nu^{-1}} \ar[r]	&
				\cdots 	\\
				&
				&
				&
				A 	\ar[r]^-{\cong}	&
				D(A)_{\nu^{-1}}	\ar[u]_-{D(\phi_0)}		\ar@{}[lu]|{\circlearrowright}	&
				&
				&
			}
		\]
	 where each of the $P_{n}$ and the $D(P_{n})_{\nu^{-1}}$ is given by
		\begin{align*} \lefteqn{}
			P_{n} &= \begin{cases}
						\bigoplus_{i=1}^{3} P(i, i)	&  \mbox{if $n = 3l $,} \\	
						\bigoplus_{i=1}^{3} P(i, i+1)	& \mbox{if $n = 3l+1$,} \\	
						\bigoplus_{i=1}^{3} P(i, i+2)	& \mbox{if $n = 3l+2$,} \\	
					\end{cases} \\
			D(P_{n})_{\nu^{-1}} &= \begin{cases}
										\bigoplus_{i=1}^{3} A (\alpha_{i} \otimes \alpha_{i+1})^{*} A	& \mbox{if $n = 3l $,} \\	
										\bigoplus_{i=1}^{3} A (\alpha_{i} \otimes \alpha_{i+2})^{*} A	& \mbox{if $n = 3l+1 $,} \\	
										\bigoplus_{i=1}^{3} A (\alpha_{i} \otimes \alpha_{i})^{*} A		& \mbox{if $n = 3l+2 $,} \\	
									\end{cases}
		\end{align*}
	 each $A$-bimodule homomorphism $\phi_{*} : P_{*} \rightarrow P_{*-1}$ given by
		\begin{align*}   \lefteqn{}
			&\phi_{6l+1}(e_{i} \otimes e_{i+1}) = \alpha_{i} \otimes e_{i+1} -e_{i} \otimes \alpha_{i};
			\quad	
			\phi_{6l+2}(e_{i} \otimes e_{i+2}) = e_{i} \otimes \alpha_{i+1} +\alpha_{i} \otimes e_{i+2}; 	\\
			&\phi_{6l+3}(e_{i} \otimes e_{i}) = \alpha_{i} \otimes e_{i} -e_{i} \otimes \alpha_{i+2};
			\quad	
			\phi_{6l+4}(e_{i} \otimes e_{i+1}) = e_{i} \otimes \alpha_{i} +\alpha_{i} \otimes e_{i+1}; 	\\
			&\phi_{6l+5}(e_{i} \otimes e_{i+2}) = \alpha_{i} \otimes e_{i+2} -e_{i} \otimes \alpha_{i+1};
			\quad
			\phi_{6l+6}(e_{i} \otimes e_{i}) = e_{i} \otimes \alpha_{i+2} +\alpha_{i} \otimes e_{i},
		\end{align*}
	 each $A$-bimodule homomorphism $D(\phi_{*}) : D(P_{*-1})_{\nu^{-1}} \rightarrow D(P_{*})_{\nu^{-1}}$ given by
		\begin{align*}  \lefteqn{}	
			&D(\phi_{6l+1})( (\alpha_{i} \otimes \alpha_{i+1})^{*} ) 
				= (\alpha_{i+2} \otimes \alpha_{i+1})^{*} \alpha_{i} -\alpha_{i+2}(\alpha_{i} \otimes \alpha_{i+2})^{*}; \\
			&D(\phi_{6l+2})( (\alpha_{i} \otimes \alpha_{i+2})^{*} ) 
				= \alpha_{i} (\alpha_{i} \otimes \alpha_{i})^{*} + (\alpha_{i+2} \otimes \alpha_{i+2})^{*} \alpha_{i};	\\		
			&D(\phi_{6l+3})( (\alpha_{i} \otimes \alpha_{i})^{*} ) 
				= (\alpha_{i+2} \otimes \alpha_{i})^{*} \alpha_{i} -\alpha_{i+1}(\alpha_{i} \otimes \alpha_{i+1})^{*}; \\
			&D(\phi_{6l+4})( (\alpha_{i} \otimes \alpha_{i+1})^{*} ) 
				= \alpha_{i+2} (\alpha_{i} \otimes \alpha_{i+2})^{*} + (\alpha_{i+2} \otimes \alpha_{i+1})^{*} \alpha_{i+2};	\\		
			&D(\phi_{6l+5})( (\alpha_{i} \otimes \alpha_{i+2})^{*} ) 
				= (\alpha_{i+2} \otimes \alpha_{i+2})^{*} \alpha_{i} -\alpha_{i}(\alpha_{i} \otimes \alpha_{i})^{*};	\\
			&D(\phi_{6l+6})( (\alpha_{i} \otimes \alpha_{i})^{*} ) 
				= \alpha_{i+1} (\alpha_{i} \otimes \alpha_{i+1})^{*} + (\alpha_{i+2} \otimes \alpha_{i})^{*} \alpha_{i},	
		\end{align*}
	the $A$-bimodule homomorphism $\phi_0 : P_{0} \rightarrow A$ given by the multiplication of $A$, 
	the $A$-bimodule homomorphism $D(\phi_0) : D(A)_{\nu^{-1}} \rightarrow D(P_{0})_{\nu^{-1}}$ given by 
		\begin{align*}  \lefteqn{}
			D(\phi_0) (\langle-, 1 \rangle)
				= \sum_{i=1}^{3} \alpha_{i+1} (\alpha_{i} \otimes \alpha_{i+1})^{*} +(\alpha_{i+2} \otimes \alpha_{i})^{*} \alpha_{i},
		\end{align*}
	and the composition $\mu : P_{0} \rightarrow D(P_{0})_{\nu^{-1}}$ given by
		\begin{align*}  \lefteqn{}
			\mu(e_{i} \otimes e_{i})
				= \alpha_{i} (\alpha_{i+2} \otimes \alpha_{i})^{*} +(\alpha_{i+1} \otimes \alpha_{i+2})^{*} \alpha_{i-1}.
		\end{align*}
	A complex which is used to compute complete cohomology groups $\widehat{\mathrm{HH}}^{*}(A)$ is a complex
		\begin{align*}\lefteqn{}
			\cdots \rightarrow 
			A_{\nu^{-1}} \otimes_{A^{\rm e}}  P_{1}  &\xrightarrow{\mathrm{id} \otimes \phi_{1}}
			A_{\nu^{-1}} \otimes_{A^{\rm e}}  P_{0}\\ &\xrightarrow{\Hom (\mu, A)}
			\Hom_{A^{\rm e}}(P_{0}, A) \xrightarrow{\Hom (\phi_{1}, A)}
			\Hom_{A^{\rm e}}(P_{1}, A) \rightarrow \cdots
		\end{align*}
	of which the terms and the nonzero differentials of the non-negative part are determined by
		\begin{align*} \lefteqn{}
			 &\Hom_{A^{\rm e}}(P_{n}, A)
				\cong \begin{cases}
						\bigoplus_{i=1}^{3} e_{i} A e_{i}			&  \mbox{if $n = 3l $,} \\	
						\bigoplus_{i=1}^{3} e_{i} A e_{i+1}		&  \mbox{if $n = 3l+1 $,} \\	
						 0											&  \mbox{if $n = 3l+2 $,} \\	
					\end{cases}	\\
			&\Hom (\phi_{6l+1}, A)( e_{i} ) = \alpha_{i+2} -\alpha_{i};  \quad
			\Hom (\phi_{6l+4}, A)(e_{i} ) 	=  \alpha_{i} +\alpha_{i+2}	
		\end{align*} 
	and that of the negative part are given by
		\begin{align*} \lefteqn{}
			&A_{\nu^{-1}} \otimes_{A^{\rm e}} P_{n}
				= \begin{cases}
						0																					& \mbox{if $n = 3l $,} \\	
						\bigoplus_{i=1}^{3} \alpha_{i+1} \otimes_{A^{\rm e}} e_{i} \otimes e_{i+1}		& \mbox{if $n = 3l+1 $,} \\	
						\bigoplus_{i=1}^{3} e_{i+2} \otimes_{A^{\rm e}} e_{i} \otimes e_{i+2}				& \mbox{if $n = 3l+2 $,} \\	
					\end{cases}	\\[5pt]	 			
			&\mathrm{id} \otimes \phi_{6l+2}( e_{i+2} \otimes_{A^{\rm e}} e_{i} \otimes e_{i+2} ) 
				= \alpha_{i+2} \otimes_{A^{\rm e}} e_{i+1} \otimes e_{i+2} +\alpha_{i+1} \otimes_{A^{\rm e}} e_{i} \otimes e_{i+1};	\\
			&\mathrm{id} \otimes \phi_{6l+5}(e_{i} \otimes_{A^{\rm e}} e_{i+1} \otimes e_{i}) 
				= \alpha_{i+2} \otimes_{A^{\rm e}} e_{i+1} \otimes e_{i+2} -\alpha_{i+1} \otimes_{A^{\rm e}} e_{i} \otimes e_{i+1}.
		\end{align*} 
	Here the term $\Hom_{A^{\rm e}}(P_{n}, A)$ is of degree $n \geq 0$. 
	Note that the two morphisms $\Hom (\phi_{6l+4}, A)$ and $\mathrm{id} \otimes \phi_{6l+2}$ are isomorphisms when $\mathrm{char}\,k \not = 2$.
	We can see that the complete cohomology groups $\widehat{\mathrm{HH}}^{*}(A)$ of $A$ 
	are divided into two cases: for $l \geq 0$,
		\begin{enumerate}
			\item $\mathrm{char}\,k \not = 2, 3$
						\begin{align*} \lefteqn{} 
							 &\widehat{\mathrm{HH}}^{n}(A)
								= \begin{cases}
										 k\,\overline{1_{A}}  			& \mbox{if $n \equiv 0 \pmod{6} $,} \\	
										 k\,\overline{\alpha_{1}} 		& \mbox{if $n \equiv 1 \pmod{6} $,} \\
										0								& \mbox{otherwise,} \\
									\end{cases}	\\			 &\widehat{\mathrm{HH}}^{-n}(A)
						= \begin{cases}
                                k\,\overline{\alpha_{2} \otimes_{A^{\rm e}} e_{1} \otimes e_{2} }			 		& \mbox{if  $n \equiv 5 \pmod{6}$,} \\
								k\,\overline{\sum_{i=1}^{3} e_{i+2} \otimes_{A^{\rm e}} e_{i} \otimes e_{i+2} } 	& \mbox{if $n \equiv 0 \pmod{6}, n \geq 1,$} \\
										0								& \mbox{otherwise,} \\
									\end{cases}	
						\end{align*} 
			\item $\mathrm{char}\,k = 2$
						\begin{align*} \lefteqn{}
							 &\widehat{\mathrm{HH}}^{n}(A)
								= \begin{cases}
										 k\,\overline{1_{A}}  			& \mbox{if $n = 3l $,} \\	
										 k\,\overline{\alpha_{1}} 		& \mbox{if $n = 3l+1$,} \\
										0								& \mbox{if $n = 3l+2$;} \\
									\end{cases}	\\
							 &\widehat{\mathrm{HH}}^{-n}(A)
								= \begin{cases}
										0																					& \mbox{if $n = 3l+1$,} \\
										k\,\overline{\alpha_{2} \otimes_{A^{\rm e}} e_{1} \otimes e_{2} }	 				& \mbox{if  $n = 3l+2$,} \\
										k\,\overline{\sum_{i=1}^{3} e_{i+2} \otimes_{A^{\rm e}} e_{i} \otimes e_{i+2} } 	& \mbox{if  $n = 3l+3$.} \\
									\end{cases}	
						\end{align*} 
		\end{enumerate}	
	As can be seen, the complete cohomology groups have the period six if $\mathrm{char}\,k \not = 2, 3$ and the period three if $\mathrm{char}\,k = 2$.
	We omit the constructions of two comparison morphisms between the minimal projective resolution  and the normalized bar resolution of $A$.
	 However, they are constructed in a similar way to the first example. 
	We have the graded commutative ring structure and the BV structure on the complete cohomology of $A$.
%
%
		\begin{prop} \label{examples-prop3}
			If $\mathrm{char}\,k \not = 2, 3$, then the complete cohomology ring $(\widehat{\mathrm{HH}}^{\bullet}(A), \star)$ is isomorphic to
				\begin{align*}  \lefteqn{}
					k[\alpha, \beta, \alpha^{-1}] /\langle \beta^{2} \rangle
				\end{align*}
			where $| \alpha | = 6, |\beta|= 1$ and $| \alpha ^{-1}| = -6$.
			Further, if this is the case, then the nonzero BV differentials $\widehat{\Delta}_{*}$ on $\widehat{\mathrm{HH}}^{\bullet}(A)$ are 
				\begin{align*}  \lefteqn{}
					\widehat{\Delta}_{6l+1}(\alpha^{l} \beta ) = ((6l+1)/3) \,\alpha^{l}, \quad	\widehat{\Delta}_{-6l-5}(\alpha^{-l-1} \beta ) = ((-6l-5)/3)\, \alpha^{-l-1} 
				\end{align*} 	
			with $l \geq 0$. In particular, the nonzero Gerstenhaber brackets are induced by
				\begin{align*}  \lefteqn{}
					&\{ \alpha, \beta \} = 2 \alpha, \quad \{ \beta, \alpha^{-1} \} = 2 \alpha^{-1}.
				\end{align*} 					 
		\end{prop}	
%
%
	\begin{prop} \label{examples-prop4}
			If $\mathrm{char}\,k = 2$, then the complete cohomology ring $(\widehat{\mathrm{HH}}^{\bullet}(A), \star)$ is isomorphic to
				\begin{align*}  \lefteqn{}
					k[\alpha, \beta, \alpha^{-1}] /\langle \beta^{2} \rangle
				\end{align*}
			where $|\alpha| = 3, |\beta| = 1$ and $ |\alpha^{-1}| = -3$.
			Further, if this is the case, then the nonzero BV differentials on $\widehat{\mathrm{HH}}^{\bullet}(A)$ are
				\begin{align*}  \lefteqn{}
					\widehat{\Delta}_{6l+1}(\alpha^{2l} \beta ) = \alpha^{2l}, \quad	\widehat{\Delta}_{-3l-2}(\alpha^{-l-1} \beta ) = \alpha^{-l-1} 
				\end{align*} 	
			with $l \geq 0$. In particular, the nonzero Gerstenhaber brackets are induced by
				\begin{align*}  \lefteqn{}
					&\{ \alpha, \beta \} =  \alpha.
				\end{align*} 					 
	\end{prop}
%
%
%
%
	\subsection{The case $s =3, N=3$.}
	Let $Q$ be a quiver 
		\[\xymatrix@ur{
			3 \ar[r]^-{\alpha_{3}}  & 1 \ar[d]^-{\alpha_{1}} \\
			 & 2 \ar[ul]^-{\alpha_{2}}&
		}\]
	and $A$ the algebra $kQ/R_{Q}^{3}$. 
	It follows from Criterion \ref{examples-criterion2} and Remark \ref{examples-remark1} that the Nakayama automorphism of $A$ is diagonalizable 
	if and only if $\mathrm{char}\,k \not = 3$.
	Hence, we assume that $\mathrm{char}\,k \not = 3$.
	We see that $A$ is a self-injective Nakayama algebra of which the representation matrix of the Nakayama
	automorphism $\nu$ is 
		\begin{align*}
			\begin{pmatrix}
				0&1&0&0&0&0&0&0&0\\
				0&0&1&0&0&0&0&0&0\\
				1&0&0&0&0&0&0&0&0\\
				0&0&0&0&1&0&0&0&0\\
				0&0&0&0&0&1&0&0&0\\
				0&0&0&1&0&0&0&0&0\\
				0&0&0&0&0&0&0&1&0\\
				0&0&0&0&0&0&0&0&1\\
				0&0&0&0&0&0&1&0&0
			\end{pmatrix}
		\end{align*}
	under a $k$-basis 
	$(e_{1}, e_{2}, e_{3}, \alpha_{1}, \alpha_{2}, \alpha_{3}, \alpha_{1} \alpha_{2}, \alpha_{2} \alpha_{3}, \alpha_{3} \alpha_{1})$ of $A$. 
	This matrix is similar to a diagonal matrix
		\begin{align*}
			\begin{pmatrix}
				1&0&0&0&0&0&0&0&0\\
				0&1&0&0&0&0&0&0&0\\
				0&0&1&0&0&0&0&0&0\\
				0&0&0&\omega&0&0&0&0&0\\
				0&0&0&0&\omega&0&0&0&0\\
				0&0&0&0&0&\omega&0&0&0\\
				0&0&0&0&0&0&\omega^{2}&0&0\\
				0&0&0&0&0&0&0&\omega^{2}&0\\
				0&0&0&0&0&0&0&0&\omega^{2}\\
			\end{pmatrix}
		\end{align*}
	where the element $\omega \in k$ is one of roots of a polynomial $x^{2} +x +1$.
	Moreover, we can decompose $A = A_{1} \oplus A_{\omega} \oplus A_{\omega^{2}}$, where
		\begin{align*} \lefteqn{}
			&A_{1} 
				= \mathrm{Ker}\, (\nu-\mathrm{id}) 
				= k\,1_{A} \oplus k\,(\sum_{i=1}^{3} \alpha_{i}) \oplus k\,(\sum_{i=1}^{3} \alpha_{i} \alpha_{i+1}), \\
			&A_{\omega} 
				= \mathrm{Ker}\, (\nu-\omega\, \mathrm{id}) 
				= k\,(\sum_{i=1}^{3} \omega^{i} e_{i}) \oplus k\,(\sum_{i=1}^{3} \omega^{i} \alpha_{i}) \oplus k\,(\sum_{i=1}^{3} \omega^{i} \alpha_{i} \alpha_{i+1}), \\
			&A_{\omega^{2}} 
				= \mathrm{Ker}\, (\nu-\omega^{2} \,\mathrm{id})\\ 
			&\hspace{19pt} 
				= k\,(\omega^{2} e_{1} +\omega e_{2} +e_{3}) \oplus k\,(\omega^{2} \alpha_{1} +\omega \alpha_{2} +\alpha_{3}) \oplus 
					k\,(\omega^{2} \alpha_{1} \alpha_{2}  +\omega \alpha_{2} \alpha_{3} +\alpha_{3} \alpha_{1}).
		\end{align*}
	In a similar way to the first example, we have a complete resolution of $A$ as follows:
		\[
			\xymatrix{						
				\cdots	\ar[r]		& 
				P_{2}	\ar[r]^-{\phi_{2}}		&
				P_{1}	\ar[r]^-{\phi_{1}}		&
				P_{0}	\ar[r]^-{\mu}	\ar[d]_-{\phi_0}		&
				D(P_{0})_{\nu^{-1}} \ar[r]^-{D(\phi_{1})}		&
				D(P_{1})_{\nu^{-1}} \ar[r]^-{D(\phi_{2})}		&
				D(P_{2})_{\nu^{-1}} \ar[r]	&
				\cdots 	\\
				&
				&
				&
				A 	\ar[r]^-{\cong}	&
				D(A)_{\nu^{-1}}	\ar[u]_-{D(\phi_0)}		\ar@{}[lu]|{\circlearrowright}	&
				&
				&
			}
		\]
	 where each of the $P_{n}$ and the $D(P_{n})_{\nu^{-1}}$ is given by
		\begin{align*} \lefteqn{}
			P_{n} &= \begin{cases}
						\bigoplus_{i=1}^{3} P(i, i+1)	& \mbox{if $n$ is odd,} \\
						\bigoplus_{i=1}^{3} P(i, i)	& \mbox{if $n$ is even;} 
					\end{cases} \\
			D(P_{n})_{\nu^{-1}} &= \begin{cases}
										\bigoplus_{i=1}^{3} A (\alpha_{i} \alpha_{i+1} \otimes \alpha_{i} \alpha_{i+1} )^{*} A	& \mbox{if $n$ is odd,} \\
										\bigoplus_{i=1}^{3} A (\alpha_{i} \alpha_{i+1} \otimes \alpha_{i+2} \alpha_{i+3})^{*} A	& \mbox{if $n$ is even,} 
									\end{cases}
		\end{align*}
	 each $A$-bimodule homomorphism $\phi_{*} : P_{*} \rightarrow P_{*-1}$ given by
		\begin{align*}  \lefteqn{}
			&\phi_{2r+1}(e_{i} \otimes e_{i+1}) = \alpha_{i} \otimes e_{i+1} -e_{i} \otimes \alpha_{i}; \\
			&\phi_{2r}(e_{i} \otimes e_{i}) = e_{i} \otimes \alpha_{i+1} \alpha_{i+2} + \alpha_{i} \otimes \alpha_{i+2} +\alpha_{i} \alpha_{i+1} \otimes e_{i},
		\end{align*}
	 each $A$-bimodule homomorphism $D(\phi_{*}) : D(P_{*-1})_{\nu^{-1}} \rightarrow D(P_{*})_{\nu^{-1}}$ given by
		\begin{align*}  \lefteqn{}
			D(\phi_{2r+1})&((\alpha_{i} \alpha_{i+1} \otimes \alpha_{i+2} \alpha_{i+3})^{*}) \\
			    &= (\alpha_{i+2} \alpha_{i} \otimes \alpha_{i+2} \alpha_{i} )^{*} \alpha_{i+1} 
					-\alpha_{i+1}(\alpha_{i} \alpha_{i+1} \otimes \alpha_{i} \alpha_{i+1} )^{*}; \\
			D(\phi_{2r})&( (\alpha_{i} \alpha_{i+1} \otimes \alpha_{i} \alpha_{i+1} )^{*} ) \\
				&= \alpha_{i+2} \alpha_{i+3} (\alpha_{i} \alpha_{i+1} \otimes \alpha_{i+2} \alpha_{i+3} )^{*} 
					+ \alpha_{i+2} (\alpha_{i-1} \alpha_{i} \otimes \alpha_{i+1} \alpha_{i+2} )^{*} \alpha_{i-2}\\
					&\hspace{5.2cm} +(\alpha_{i-2} \alpha_{i-1} \otimes \alpha_{i} \alpha_{i+1} )^{*} \alpha_{i-3} \alpha_{i-2},
		\end{align*}
		the $A$-bimodule homomorphism $\phi_0 : P_{0} \rightarrow A$ given by the multiplication of $A$, 
		the $A$-bimodule homomorphism $D(\phi_0) : D(A)_{\nu^{-1}} \rightarrow D(P_{0})_{\nu^{-1}}$ given by 
		\begin{align*}  \lefteqn{}
			D(\phi_0) (\langle-, 1 \rangle)
				=&\ \sum_{i=1}^{3} \big( \alpha_{i} \alpha_{i+1} (\alpha_{i-2} \alpha_{i-1} \otimes \alpha_{i} \alpha_{i+1})^{*} 
					+\alpha_{i+1} (\alpha_{i-2} \alpha_{i-1} \otimes \alpha_{i} \alpha_{i+1})^{*} \alpha_{i} \\
				  &+(\alpha_{i-2} \alpha_{i-1} \otimes \alpha_{i} \alpha_{i+1})^{*} \alpha_{i-3} \alpha_{i-2} \big),
		\end{align*}
	and the composition $\mu : P_{0} \rightarrow D(P_{0})_{\nu^{-1}}$ given by
		\begin{align*}  \lefteqn{}
			\mu(e_{i} \otimes e_{i})
				=&\  \alpha_{i} \alpha_{i+1} (\alpha_{i-2} \alpha_{i-1} \otimes \alpha_{i} \alpha_{i+1} )^{*} 
					+(\alpha_{i-1} \alpha_{i} \otimes \alpha_{i+1} \alpha_{i+2})^{*} \alpha_{i-2} \alpha_{i-1} \\
					&+ \alpha_{i} (\alpha_{i-3} \alpha_{i-2} \otimes \alpha_{i-1} \alpha_{i} )^{*} \alpha_{i-4}.
		\end{align*}
	Moreover, a complex which is used to compute complete cohomology groups is a complex
	\begin{align*}\lefteqn{}
		\cdots \rightarrow 
		A_{\nu^{-1}} \otimes_{A^{\rm e}}  P_{1}  &\xrightarrow{\mathrm{id} \otimes \phi_{1}}
		A_{\nu^{-1}} \otimes_{A^{\rm e}}  P_{0}\\ &\xrightarrow{\Hom (\mu, A)}
		\Hom_{A^{\rm e}}(P_{0}, A) \xrightarrow{\Hom (\phi_{1}, A)}
		\Hom_{A^{\rm e}}(P_{1}, A) \rightarrow \cdots
	\end{align*}
	of which the terms and the nonzero differentials are determined by
		\begin{align*} \lefteqn{}
			 &\Hom_{A^{\rm e}}(P_{n}, A)
				\cong \begin{cases}
						\bigoplus_{i=1}^{3} k\,\alpha_{i}  	& \mbox{if $n$ is odd,} \\
						\bigoplus_{i=1}^{3} k\,e_{i}	& \mbox{if $n$ is even};
					\end{cases}	\\
			&A_{\nu^{-1}} \otimes_{A^{\rm e}} P_{n}
				= \begin{cases}
						\bigoplus_{i=1}^{3} k\, e_{i+1} \otimes_{A^{\rm e}} e_{i} \otimes e_{i+1}&\mbox{if $n$ is odd,} \\
						\bigoplus_{i=1}^{3} k\,\alpha_{i} \otimes_{A^{\rm e}} e_{i} \otimes e_{i}&\mbox{if $n$ is even};
					\end{cases}	\\
			&\Hom(\phi_{2r+1}, A)(e_{i} ) 
				=  \alpha_{i+1} -\alpha_{i}; \\
			&\mathrm{id} \otimes \phi_{2r+1}(e_{i} \otimes_{A^{\rm e}} e_{i+1} \otimes e_{i}  ) 
				=  \alpha_{i} \otimes_{A^{\rm e}} e_{i} \otimes e_{i} -\alpha_{i+1} \otimes_{A^{\rm e}} e_{i+1} \otimes e_{i+1}
		\end{align*} 
	and whose term $\Hom_{A^{\rm e}}(P_{n}, A)$ is of degree $n \geq 0$. Therefore, we have, for $n \geq 0$,
		\begin{align*} \lefteqn{}
			 &\widehat{\mathrm{HH}}^{n}(A)
				= \begin{cases}
						 k\,\overline{\alpha_{1}}  	& \mbox{if $n$ is odd,} \\
						 k\,\overline{1_{A}} & \mbox{if $n$ is even};
					\end{cases}	 
			\quad \\
			& \widehat{\mathrm{HH}}^{-n}(A)
				= \begin{cases}			
						k\,\overline{\alpha_{1} \otimes_{A^{\rm e}} e_{1} \otimes e_{1}} & \mbox{if $n$ is odd,} \\
						k\,\overline{ \sum_{i=1}^{3} e_{i+1} \otimes_{A^{\rm e}} e_{i} \otimes e_{i+1}}	& \mbox{if $n >0$ is even}.
					\end{cases}
		\end{align*} 
	 We omit the description of comparison morphisms between the minimal projective resolution and the normalized bar resolution of $A$, because
	 it is not easy to write the two comparison morphisms.
	 However, a direct calculation shows the graded commutative ring structure and the BV structure on the complete cohomology of $A$.
%
%
		\begin{prop} \label{examples-prop5}
			The complete cohomology ring $(\widehat{\mathrm{HH}}^{\bullet}(A), \star)$ is isomorphic to
				\begin{align*}  \lefteqn{}
					k[\alpha, \beta, \alpha^{-1}] /\langle \beta^2 \rangle
				\end{align*}
			where $| \alpha | = 2, |\beta|= 1$ and $| \alpha ^{-1}| = -2$. 
			Moreover, the nonzero BV differentials on $\widehat{\mathrm{HH}}^{\bullet}(A)$ are
				\begin{align*}  \lefteqn{}
					\widehat{\Delta}_{2l+1}(\alpha^{l} \beta ) = ((3l+2)/3) \alpha^{l}, \quad	
					\widehat{\Delta}_{-2l-1}(\alpha^{-l-1} \beta ) = \begin{cases}
																		(-1/3)\, \alpha^{-1} & \mbox{ if $l =0$,} \\
																		((-3l-2)/3)\, \alpha^{-l-1} & \mbox{ if $l \not =0$}
																	\end{cases}
				\end{align*} 	
			with $l \geq 0$. In particular, the nonzero Gerstenhaber brackets are induced by
				\begin{align*}  \lefteqn{}
					&\{ \alpha, \beta \} = \alpha, \quad 	\{ \beta, \alpha^{-1} \} = \alpha^{-1}.
				\end{align*} 				
		\end{prop}


\section*{Acknowledgments} 
    The authors would like to express our gratitude to the referee for a careful reading of the paper and for helpful comments and suggestions.
	The authors also would like to thank Professor Guodong Zhou for giving us many valuable suggestions and discussions for improvement.
	Third author is grateful to Professor Zhengfang Wang for interesting suggestions and comments.
	First author's research was partially supported by JSPS Grant-in-Aid for Young Scientists (B) 17K14175.	
	Second author's research was partially supported by JSPS Grant-in-Aid for Scientific Research (C) 17K05211.

%
%
%
%
   
\end{document}

%% file: Figure1_BV.tex
{\unitlength 0.1in%
\begin{picture}(15.9300,15.9600)(18.3800,-30.3300)%
%
\special{pn 8}%
\special{pa 2025 2294}%
\special{pa 3366 2294}%
\special{fp}%
\special{sh 1}%
\special{pa 3366 2294}%
\special{pa 3299 2274}%
\special{pa 3313 2294}%
\special{pa 3299 2314}%
\special{pa 3366 2294}%
\special{fp}%
%
\special{pn 8}%
\special{pa 2696 2963}%
\special{pa 2696 1624}%
\special{fp}%
\special{sh 1}%
\special{pa 2696 1624}%
\special{pa 2676 1691}%
\special{pa 2696 1677}%
\special{pa 2716 1691}%
\special{pa 2696 1624}%
\special{fp}%
%
\special{pn 4}%
\special{sh 1}%
\special{ar 2696 2294 16 16 0 6.2831853}%
\special{sh 1}%
\special{ar 2547 2294 16 16 0 6.2831853}%
\special{sh 1}%
\special{ar 2398 2294 16 16 0 6.2831853}%
\special{sh 1}%
\special{ar 2249 2294 16 16 0 6.2831853}%
\special{sh 1}%
\special{ar 2100 2294 16 16 0 6.2831853}%
\special{sh 1}%
\special{ar 2100 2145 16 16 0 6.2831853}%
\special{sh 1}%
\special{ar 2100 1996 16 16 0 6.2831853}%
\special{sh 1}%
\special{ar 2100 1847 16 16 0 6.2831853}%
\special{sh 1}%
\special{ar 2100 1698 16 16 0 6.2831853}%
\special{sh 1}%
\special{ar 2249 1698 16 16 0 6.2831853}%
\special{sh 1}%
\special{ar 2249 1847 16 16 0 6.2831853}%
\special{sh 1}%
\special{ar 2249 1996 16 16 0 6.2831853}%
\special{sh 1}%
\special{ar 2249 2145 16 16 0 6.2831853}%
\special{sh 1}%
\special{ar 2249 2442 16 16 0 6.2831853}%
\special{sh 1}%
\special{ar 2249 2591 16 16 0 6.2831853}%
\special{sh 1}%
\special{ar 2249 2740 16 16 0 6.2831853}%
\special{sh 1}%
\special{ar 2249 2888 16 16 0 6.2831853}%
\special{sh 1}%
\special{ar 2100 2888 16 16 0 6.2831853}%
\special{sh 1}%
\special{ar 2100 2740 16 16 0 6.2831853}%
\special{sh 1}%
\special{ar 2100 2591 16 16 0 6.2831853}%
\special{sh 1}%
\special{ar 2100 2442 16 16 0 6.2831853}%
\special{sh 1}%
\special{ar 2398 2442 16 16 0 6.2831853}%
\special{sh 1}%
\special{ar 2398 2591 16 16 0 6.2831853}%
\special{sh 1}%
\special{ar 2398 2740 16 16 0 6.2831853}%
\special{sh 1}%
\special{ar 2398 2888 16 16 0 6.2831853}%
\special{sh 1}%
\special{ar 2547 2888 16 16 0 6.2831853}%
\special{sh 1}%
\special{ar 2547 2740 16 16 0 6.2831853}%
\special{sh 1}%
\special{ar 2547 2591 16 16 0 6.2831853}%
\special{sh 1}%
\special{ar 2547 2442 16 16 0 6.2831853}%
\special{sh 1}%
\special{ar 2547 2294 16 16 0 6.2831853}%
\special{sh 1}%
\special{ar 2547 2294 16 16 0 6.2831853}%
\special{sh 1}%
\special{ar 2547 2145 16 16 0 6.2831853}%
\special{sh 1}%
\special{ar 2547 1996 16 16 0 6.2831853}%
\special{sh 1}%
\special{ar 2547 1847 16 16 0 6.2831853}%
\special{sh 1}%
\special{ar 2547 1847 16 16 0 6.2831853}%
\special{sh 1}%
\special{ar 2547 1847 16 16 0 6.2831853}%
\special{sh 1}%
\special{ar 2547 1698 16 16 0 6.2831853}%
\special{sh 1}%
\special{ar 2398 1698 16 16 0 6.2831853}%
\special{sh 1}%
\special{ar 2398 1847 16 16 0 6.2831853}%
\special{sh 1}%
\special{ar 2398 1996 16 16 0 6.2831853}%
\special{sh 1}%
\special{ar 2398 2145 16 16 0 6.2831853}%
\special{sh 1}%
\special{ar 2696 2145 16 16 0 6.2831853}%
\special{sh 1}%
\special{ar 2696 1996 16 16 0 6.2831853}%
\special{sh 1}%
\special{ar 2696 1847 16 16 0 6.2831853}%
\special{sh 1}%
\special{ar 2696 1698 16 16 0 6.2831853}%
\special{sh 1}%
\special{ar 2845 1698 16 16 0 6.2831853}%
\special{sh 1}%
\special{ar 2845 1847 16 16 0 6.2831853}%
\special{sh 1}%
\special{ar 2845 1996 16 16 0 6.2831853}%
\special{sh 1}%
\special{ar 2845 2145 16 16 0 6.2831853}%
\special{sh 1}%
\special{ar 2845 2294 16 16 0 6.2831853}%
\special{sh 1}%
\special{ar 2845 2442 16 16 0 6.2831853}%
\special{sh 1}%
\special{ar 2845 2591 16 16 0 6.2831853}%
\special{sh 1}%
\special{ar 2845 2740 16 16 0 6.2831853}%
\special{sh 1}%
\special{ar 2845 2888 16 16 0 6.2831853}%
\special{sh 1}%
\special{ar 2696 2888 16 16 0 6.2831853}%
\special{sh 1}%
\special{ar 2696 2740 16 16 0 6.2831853}%
\special{sh 1}%
\special{ar 2696 2591 16 16 0 6.2831853}%
\special{sh 1}%
\special{ar 2696 2442 16 16 0 6.2831853}%
\special{sh 1}%
\special{ar 2994 2442 16 16 0 6.2831853}%
\special{sh 1}%
\special{ar 2994 2591 16 16 0 6.2831853}%
\special{sh 1}%
\special{ar 2994 2740 16 16 0 6.2831853}%
\special{sh 1}%
\special{ar 2994 2888 16 16 0 6.2831853}%
\special{sh 1}%
\special{ar 2994 2145 16 16 0 6.2831853}%
\special{sh 1}%
\special{ar 2994 1996 16 16 0 6.2831853}%
\special{sh 1}%
\special{ar 2994 1847 16 16 0 6.2831853}%
\special{sh 1}%
\special{ar 2994 1698 16 16 0 6.2831853}%
\special{sh 1}%
\special{ar 3142 1698 16 16 0 6.2831853}%
\special{sh 1}%
\special{ar 3142 1847 16 16 0 6.2831853}%
\special{sh 1}%
\special{ar 3142 1996 16 16 0 6.2831853}%
\special{sh 1}%
\special{ar 3142 2145 16 16 0 6.2831853}%
\special{sh 1}%
\special{ar 3142 2294 16 16 0 6.2831853}%
\special{sh 1}%
\special{ar 3142 2442 16 16 0 6.2831853}%
\special{sh 1}%
\special{ar 3142 2591 16 16 0 6.2831853}%
\special{sh 1}%
\special{ar 3142 2740 16 16 0 6.2831853}%
\special{sh 1}%
\special{ar 3142 2888 16 16 0 6.2831853}%
\special{sh 1}%
\special{ar 2994 2294 16 16 0 6.2831853}%
\special{sh 1}%
\special{ar 3292 2294 16 16 0 6.2831853}%
\special{sh 1}%
\special{ar 3292 2442 16 16 0 6.2831853}%
\special{sh 1}%
\special{ar 3292 2591 16 16 0 6.2831853}%
\special{sh 1}%
\special{ar 3292 2740 16 16 0 6.2831853}%
\special{sh 1}%
\special{ar 3292 2888 16 16 0 6.2831853}%
\special{sh 1}%
\special{ar 3292 2145 16 16 0 6.2831853}%
\special{sh 1}%
\special{ar 3292 1996 16 16 0 6.2831853}%
\special{sh 1}%
\special{ar 3292 1847 16 16 0 6.2831853}%
\special{sh 1}%
\special{ar 3292 1698 16 16 0 6.2831853}%
\special{sh 1}%
\special{ar 3292 1698 16 16 0 6.2831853}%
\put(26.9000,-15.3000){\makebox(0,0){$n$}}%
\put(35.1000,-22.9000){\makebox(0,0){$m$}}%
%
\special{pn 8}%
\special{pa 2845 2145}%
\special{pa 2845 1624}%
\special{fp}%
%
\special{pn 8}%
\special{pa 2845 2145}%
\special{pa 3366 2145}%
\special{fp}%
%
\special{pn 8}%
\special{pa 2994 2442}%
\special{pa 3366 2442}%
\special{fp}%
%
\special{pn 8}%
\special{pa 2994 2442}%
\special{pa 3366 2814}%
\special{fp}%
%
\special{pn 8}%
\special{pa 2845 2442}%
\special{pa 3366 2963}%
\special{fp}%
%
\special{pn 8}%
\special{pa 2845 2963}%
\special{pa 2845 2442}%
\special{fp}%
%
\special{pn 8}%
\special{pa 2547 2145}%
\special{pa 2025 1624}%
\special{fp}%
%
\special{pn 8}%
\special{pa 2547 2442}%
\special{pa 2025 2442}%
\special{fp}%
%
\special{pn 8}%
\special{pa 2547 2442}%
\special{pa 2547 2963}%
\special{fp}%
%
\special{pn 8}%
\special{pa 2025 2145}%
\special{pa 2025 1772}%
\special{ip}%
%
\special{pn 8}%
\special{pa 2025 1624}%
\special{pa 2025 1624}%
\special{ip}%
%
\special{pn 8}%
\special{pa 2845 1624}%
\special{pa 3366 1624}%
\special{ip}%
%
\special{pn 8}%
\special{pa 3366 1624}%
\special{pa 3366 2145}%
\special{ip}%
%
\special{pn 8}%
\special{pa 3366 2442}%
\special{pa 3366 2814}%
\special{ip}%
%
\special{pn 8}%
\special{pa 3366 2963}%
\special{pa 2845 2963}%
\special{ip}%
%
\special{pn 8}%
\special{pa 2547 2963}%
\special{pa 2025 2963}%
\special{ip}%
%
\special{pn 8}%
\special{pa 2025 2963}%
\special{pa 2025 2442}%
\special{ip}%
\put(30.7100,-15.0200){\makebox(0,0){I}}%
\put(23.6700,-15.0200){\makebox(0,0){II}}%
\put(18.9800,-19.7900){\makebox(0,0){III}}%
\put(18.9800,-26.3700){\makebox(0,0){IV}}%
\put(30.6900,-30.9800){\makebox(0,0){V}}%
\put(34.8100,-26.3700){\makebox(0,0){VI}}%
%
\special{pn 8}%
\special{pa 2025 2145}%
\special{pa 2547 2145}%
\special{fp}%
%
\special{pn 8}%
\special{pa 2025 2136}%
\special{pa 2025 1616}%
\special{ip}%
%
\special{pn 8}%
\special{pa 2547 1996}%
\special{pa 2175 1624}%
\special{fp}%
%
\special{pn 8}%
\special{pa 2547 1624}%
\special{pa 2547 1996}%
\special{fp}%
%
\special{pn 8}%
\special{pa 2175 1624}%
\special{pa 2547 1624}%
\special{ip}%
%
\special{pn 4}%
\special{pa 2540 1698}%
\special{pa 2398 1839}%
\special{fp}%
\special{pa 2547 1757}%
\special{pa 2428 1877}%
\special{fp}%
\special{pa 2547 1824}%
\special{pa 2461 1910}%
\special{fp}%
\special{pa 2547 1892}%
\special{pa 2495 1943}%
\special{fp}%
\special{pa 2547 1959}%
\special{pa 2528 1977}%
\special{fp}%
\special{pa 2543 1627}%
\special{pa 2360 1809}%
\special{fp}%
\special{pa 2398 1705}%
\special{pa 2327 1776}%
\special{fp}%
\special{pa 2413 1624}%
\special{pa 2293 1742}%
\special{fp}%
\special{pa 2346 1624}%
\special{pa 2261 1709}%
\special{fp}%
\special{pa 2279 1624}%
\special{pa 2227 1675}%
\special{fp}%
\special{pa 2212 1624}%
\special{pa 2193 1642}%
\special{fp}%
\special{pa 2480 1624}%
\special{pa 2406 1698}%
\special{fp}%
%
\special{pn 4}%
\special{pa 2286 1884}%
\special{pa 2029 2140}%
\special{fp}%
\special{pa 2100 2003}%
\special{pa 2025 2078}%
\special{fp}%
\special{pa 2219 1817}%
\special{pa 2025 2010}%
\special{fp}%
\special{pa 2186 1784}%
\special{pa 2025 1943}%
\special{fp}%
\special{pa 2153 1750}%
\special{pa 2025 1877}%
\special{fp}%
\special{pa 2119 1717}%
\special{pa 2025 1809}%
\special{fp}%
\special{pa 2086 1683}%
\special{pa 2025 1742}%
\special{fp}%
\special{pa 2052 1649}%
\special{pa 2025 1675}%
\special{fp}%
\special{pa 2249 1855}%
\special{pa 2108 1996}%
\special{fp}%
\special{pa 2242 1996}%
\special{pa 2100 2136}%
\special{fp}%
\special{pa 2353 1952}%
\special{pa 2160 2145}%
\special{fp}%
\special{pa 2387 1984}%
\special{pa 2227 2145}%
\special{fp}%
\special{pa 2421 2018}%
\special{pa 2293 2145}%
\special{fp}%
\special{pa 2454 2052}%
\special{pa 2360 2145}%
\special{fp}%
\special{pa 2487 2085}%
\special{pa 2428 2145}%
\special{fp}%
\special{pa 2520 2119}%
\special{pa 2495 2145}%
\special{fp}%
\special{pa 2320 1918}%
\special{pa 2249 1988}%
\special{fp}%
%
\special{pn 4}%
\special{pa 3351 1624}%
\special{pa 2845 2129}%
\special{fp}%
\special{pa 3366 1675}%
\special{pa 2897 2145}%
\special{fp}%
\special{pa 3366 1742}%
\special{pa 2964 2145}%
\special{fp}%
\special{pa 3366 1809}%
\special{pa 3031 2145}%
\special{fp}%
\special{pa 3366 1877}%
\special{pa 3098 2145}%
\special{fp}%
\special{pa 3366 1943}%
\special{pa 3165 2145}%
\special{fp}%
\special{pa 3366 2010}%
\special{pa 3232 2145}%
\special{fp}%
\special{pa 3366 2078}%
\special{pa 3299 2145}%
\special{fp}%
\special{pa 3283 1624}%
\special{pa 2845 2062}%
\special{fp}%
\special{pa 3217 1624}%
\special{pa 3146 1695}%
\special{fp}%
\special{pa 3150 1624}%
\special{pa 2845 1929}%
\special{fp}%
\special{pa 2990 1851}%
\special{pa 2848 1992}%
\special{fp}%
\special{pa 3082 1624}%
\special{pa 2845 1862}%
\special{fp}%
\special{pa 3015 1624}%
\special{pa 2845 1794}%
\special{fp}%
\special{pa 2949 1624}%
\special{pa 2845 1728}%
\special{fp}%
\special{pa 2882 1624}%
\special{pa 2845 1661}%
\special{fp}%
\special{pa 3139 1702}%
\special{pa 2998 1843}%
\special{fp}%
%
\special{pn 4}%
\special{pa 3336 2442}%
\special{pa 3165 2614}%
\special{fp}%
\special{pa 3366 2479}%
\special{pa 3198 2647}%
\special{fp}%
\special{pa 3366 2546}%
\special{pa 3232 2681}%
\special{fp}%
\special{pa 3366 2614}%
\special{pa 3265 2714}%
\special{fp}%
\special{pa 3366 2681}%
\special{pa 3299 2747}%
\special{fp}%
\special{pa 3366 2747}%
\special{pa 3332 2781}%
\special{fp}%
\special{pa 3269 2442}%
\special{pa 3131 2580}%
\special{fp}%
\special{pa 3202 2442}%
\special{pa 3098 2546}%
\special{fp}%
\special{pa 3131 2447}%
\special{pa 3065 2513}%
\special{fp}%
\special{pa 3068 2442}%
\special{pa 3031 2479}%
\special{fp}%
%
\special{pn 4}%
\special{pa 3090 2688}%
\special{pa 2845 2934}%
\special{fp}%
\special{pa 3124 2721}%
\special{pa 2882 2963}%
\special{fp}%
\special{pa 3158 2755}%
\special{pa 2949 2963}%
\special{fp}%
\special{pa 3191 2789}%
\special{pa 3015 2963}%
\special{fp}%
\special{pa 3225 2821}%
\special{pa 3082 2963}%
\special{fp}%
\special{pa 3258 2855}%
\special{pa 3150 2963}%
\special{fp}%
\special{pa 3288 2892}%
\special{pa 3217 2963}%
\special{fp}%
\special{pa 3325 2922}%
\special{pa 3283 2963}%
\special{fp}%
\special{pa 3057 2654}%
\special{pa 2845 2867}%
\special{fp}%
\special{pa 3023 2621}%
\special{pa 2845 2800}%
\special{fp}%
\special{pa 2986 2591}%
\special{pa 2845 2732}%
\special{fp}%
\special{pa 2956 2553}%
\special{pa 2845 2665}%
\special{fp}%
\special{pa 2922 2520}%
\special{pa 2852 2591}%
\special{fp}%
\special{pa 2889 2487}%
\special{pa 2845 2531}%
\special{fp}%
\special{pa 2855 2454}%
\special{pa 2845 2464}%
\special{fp}%
%
\special{pn 4}%
\special{pa 2547 2494}%
\special{pa 2079 2963}%
\special{fp}%
\special{pa 2547 2561}%
\special{pa 2145 2963}%
\special{fp}%
\special{pa 2547 2628}%
\special{pa 2212 2963}%
\special{fp}%
\special{pa 2547 2695}%
\special{pa 2279 2963}%
\special{fp}%
\special{pa 2547 2762}%
\special{pa 2346 2963}%
\special{fp}%
\special{pa 2547 2829}%
\special{pa 2413 2963}%
\special{fp}%
\special{pa 2547 2896}%
\special{pa 2480 2963}%
\special{fp}%
\special{pa 2532 2442}%
\special{pa 2025 2949}%
\special{fp}%
\special{pa 2465 2442}%
\special{pa 2025 2881}%
\special{fp}%
\special{pa 2394 2447}%
\special{pa 2253 2587}%
\special{fp}%
\special{pa 2331 2442}%
\special{pa 2025 2747}%
\special{fp}%
\special{pa 2096 2744}%
\special{pa 2025 2814}%
\special{fp}%
\special{pa 2264 2442}%
\special{pa 2025 2681}%
\special{fp}%
\special{pa 2197 2442}%
\special{pa 2025 2614}%
\special{fp}%
\special{pa 2130 2442}%
\special{pa 2025 2546}%
\special{fp}%
\special{pa 2063 2442}%
\special{pa 2025 2479}%
\special{fp}%
\special{pa 2246 2594}%
\special{pa 2104 2736}%
\special{fp}%
\end{picture}}%